\documentclass[11pt]{amsart}

\usepackage{  
  amsmath,
  amsfonts,
  amssymb,
  mvsymb,   
  stmaryrd,
  graphicx, 
  times,
  color,
  mathptmx, 
  pinlabel,
  hyphenat}

\input xy
\xyoption{all}

\newcommand{\digonup}{\digon\mspace{-3.5mu}}
\newcommand{\hatdigonup}{\hatdigon\mspace{-3.5mu}}
\newcommand{\hatarcup}{\hatarc\mspace{-3.5mu}}

\newcommand{\brak}[1]{\langle #1\rangle}
\newcommand{\n}{\noindent}

\DeclareMathOperator{\stens}{\!\text{\raisebox{0.5pt}{${\scriptstyle\otimes }$}}\!}

\newcommand{\figins}[3] 
{\raisebox{#1pt}{\includegraphics[height=#2 in]{figs/#3.eps}}}

\newcommand{\figs}[2] 
{{\includegraphics[scale =#1]{figs/#2}}}

\newtheorem{thm}{Theorem}[section]
\newtheorem{lem}[thm]{Lemma}

\newtheorem{exe}[thm]{Example}
\newtheorem{conj}{Conjecture}

\theoremstyle{definition}
\newtheorem{defn}[thm]{Definition}

\newcommand{\xra}[1]{\xrightarrow{#1}}

\newcommand{\bC}{\mathbb{C}}

\newcommand{\cC}{\mathcal{C}}

\newcommand{\cD}{\mathcal{D}}

\DeclareMathOperator{\id}{id}

\newcommand{\ot}{\otimes}

\newcommand{\qbin}[2]{\left[{{#1}\atop {#2}}\right]}

\def\deg{\mathop{\rm deg}}

\catcode`\@=11
\long\def\@makecaption#1#2{%
    \vskip 10pt
    \setbox\@tempboxa\hbox{%
\small{#1: }\ignorespaces #2}%
    \ifdim \wd\@tempboxa >\captionwidth {%
        \rightskip=\@captionmargin\leftskip=\@captionmargin
        \unhbox\@tempboxa\par}%
      \else
        \hbox to\hsize{\hfil\box\@tempboxa\hfil}%
    \fi}
\newdimen\@captionmargin\@captionmargin=2\parindent
\newdimen\captionwidth\captionwidth=\hsize
\catcode`\@=12
\title{The 1,2-coloured HOMFLY-PT link homology}
\author{Marco Mackaay}
\address{Departamento de Matem\'{a}tica\\ Universidade do Algarve\\ 
Campus de Gambelas\\ 8005-139 Faro\\ Portugal and CAMGSD\\Instituto Superior T\'{e}cnico\\ Avenida Rovisco Pais\\ 
1049-001 Lisboa\\ Portugal}
\email{mmackaay@ualg.pt}
\author{Marko Sto\v si\'c }
\address{Instituto de Sistemas e Rob\'{o}tica and CAMGSD\\Instituto Superior T\'{e}cnico\\ Avenida Rovisco Pais\\ 
1049-001 Lisboa\\ Portugal}
\email{mstosic@math.ist.utl.pt}
\author{Pedro Vaz}
\address{Departamento de Matem\'{a}tica\\ Universidade do Algarve\\ 
Campus de Gambelas\\ 8005-139 Faro\\ Portugal and  
CAMGSD\\Instituto Superior T\'{e}cnico\\ Avenida Rovisco Pais\\ 
1049-001 Lisboa\\ Portugal}
\email{pfortevaz@ualg.pt}

\begin{document}
\newdimen\captionwidth\captionwidth=\hsize
%
%

\begin{abstract} In this paper we define the 1,2-coloured HOMFLY-PT link homology and prove 
that it is a link invariant. 
We conjecture that this homology categorifies the coloured HOMFLY-PT polynomial for links whose 
components 
are labelled 1 or 2.
\end{abstract}

\maketitle

\section{Introduction}
\label{sec:intro}
In this paper we define the coloured HOMFLY-PT link homology for links whose components are 
labelled 1 or 2. Since we did not find a complete recursive combinatorial 
calculus as in \cite{kim} and \cite{MOY} for the 
coloured HOMFLY-PT link polynomial just for the colours 1 and 2, we cannot prove that our link homology categorifies it. However, we show that all the 
relations which are necessary for proving invariance of the link polynomial are 
categorified in our setting. Therefore, we conjecture that our homology categorifies the 
1,2-coloured HOMFLY-PT link polynomial.  
For the definition of the 1,2-coloured HOMFLY-PT link homology we need to use a braid presentation of 
the link. We prove invariance of the homology under the second and third braidlike Reidemeister moves and 
the Markov moves. 

The 1,2-coloured HOMFLY-PT link homology is a triply graded link homology just as the ordinary 
HOMFLY-PT link homology due to Khovanov and Rozansky~\cite{KR2}. In 
\cite{GIKV} such a link homology was conjectured to exist from the physics 
point of view. We follow the approach using bimodules and 
Hochschild homology as was done for the ordinary HOMFLY-PT link homology by Khovanov in~\cite{Kh}. 
With Rasmussen's results for the ordinary HOMFLY-PT homology in mind, we conjecture that, in a certain 
sense, our link homology 
is the limit of the 1,2-coloured $sl(N)$ link homology, yet to be defined, when $N$ 
goes to infinity. If this is true, the colours 1 and 2 have a nice 
interpretation. They are the fundamental 
representation of $sl(N)$ and its second exterior power.

Although these $sl(N)$ link homologies have not been defined yet, there has been made some progress 
towards their definition in~\cite{wu-MOY,yonezawa}. In those papers the matrix factorization approach is followed. 
One should also be able to define the 1,2-coloured HOMFLY-PT link homology using matrix factorizations 
and in some sense this should be equivalent to our approach. 
For the 1,2-coloured HOMFLY-PT link homology one would only 
need one part of the matrix factorizations, which is somehow the easiest part. 
For technical reasons the bimodule approach is slightly easier, which is why we have not used 
matrix factorizations.     
 
In the last section of this paper we have sketched how to define the 
coloured HOMFLY-PT link homology for arbitrary colours and how to prove its invariance. The underlying 
ideas are the same, but the actual calculations are much harder. One needs a different technique 
to handle those calculations for arbitrary colours. Therefore we leave the general coloured HOMFLY-PT link 
homology for a next paper.  

To compute the 1,2-coloured HOMFLY-PT link homology is very hard. In 
\cite{GIKV} there is a conjecture for the Hopf link. 
One road to follow would be the one 
Rasmussen showed for the ordinary HOMFLY-PT link homology~\cite{rasmussen}. Once the $sl(N)$-link homologies 
have been defined, there should be spectral sequences from the coloured HOMFLY-PT link homologies to the 
$sl(N)$-link homologies. For ``small'' knots these spectral sequences should collapse for low values of 
$N$, which might make them computable. Another road to follow would be the one inniated by Webster and 
Williamson~\cite{webster}. For starters this would require a geometric interpretation of the bimodules and 
their Hochschild homology used in this paper. This approach might give some results for certain 
classes of knots, like the torus knots. 

Finally, let us briefly sketch an outline of our paper. In Section~\ref{sec:MOY} we recall some basic 
facts from~\cite{MOY} about the combinatorial calculus of the 1,2-coloured HOMFLY-PT polynomial and define one version of part of it
that we shall categorify. In 
Section~\ref{sec:triply} we categorify this part of the calculus by using 
bimodules. In 
Section~\ref{sec:HOMFLY-PT} we define the 1,2-coloured HOMFLY-PT link homology. As said before, we use 
a braid presentation of the link. In Section~\ref{sec:invariance} we prove invariance of 
the 1,2-coloured HOMFLY-PT link homology under the 
braidlike Reidemeister moves II and III. In Section~\ref{sec:Markov} we prove its invariance under the 
Markov moves. In 
Section~\ref{sec:conjectures} we sketch the definition of the coloured HOMFLY-PT link homology for 
arbitrary colours and 
conjecture its invariance under the second and third Reidemeister moves and the Markov moves.     

We assume familiarity with \cite{Kh, KR, KR2, kim, MOY}.

\section{The MOY calculus}
\label{sec:MOY}

In this section we recall part of the MOY calculus for the 1,2-coloured HOMFLY-PT link polynomial. As 
already remarked in the introduction this part of the calculus does probably not give a complete 
recursive calculus 
for the polynomial invariant. At least we do not know any proof of such a fact. We have simply picked 
those relations that are necessary for proving the invariance of the polynomial. As it turns out their  
categorifications prove the invariance of the related link homology. 

The calculus uses labelled trivalent graphs, which we call MOY webs, and is similar to the MOY calculus 
from \cite{MOY} and generalizes the calculus used by Khovanov and Rozansky in \cite{KR2} to 
define triply-graded link homology.
The resolutions of link diagrams consist of MOY webs, whose edges are labelled by positive integers, such 
that at each trivalent vertex the sum of the labels of the outgoing edges 
equals the sum of the labels of ingoing edges. Although the theory can be 
extended to allow for general labellings, in this paper, we shall 
only consider the ones where the labellings of the edges are from the set 
$\{1,2,3,4\}$. 

First we introduce the calculus for such graphs. This is an extension of the one 
with labellings being only 1 and 2 (see \cite{KR2}), and a variant of the one 
from \cite{MOY}. The value of a closed graph is given by requiring the following 
axioms to hold:
\begin{align}
\tag{A1} 
\unknot_k &= \prod_{i=1}^k \frac{1+t^{-1}q^{2i-1}}{1-q^{2i}} \label{ax1} 
\displaybreak[0]\\[1.2ex]
\tag{A2}
\raisebox{-18pt}{
\labellist
\tiny\hair 2pt
\pinlabel $i$   at -10  15
\pinlabel $i$   at -10 205
\pinlabel $j$   at 112 110
\pinlabel $i+j$ at -15 110
\endlabellist 
\figs{0.2}{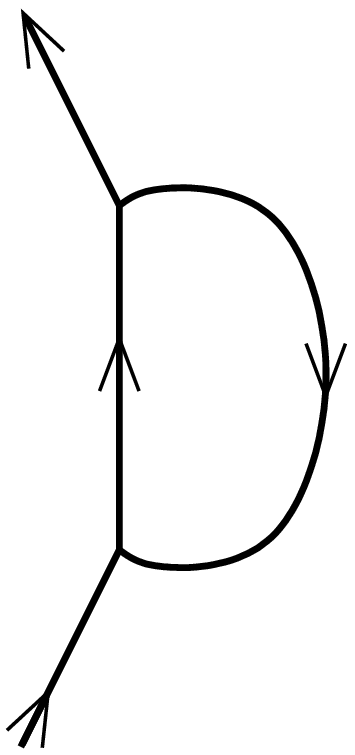}}\mspace{10mu} & = \prod_{l=1}^j \frac{1+t^{-1}q^{2i+2l-1}}{1-q^{2l}}
\raisebox{-18pt}{
\labellist
\tiny\hair 2pt
\pinlabel $i$ at -10 15
\endlabellist
\figs{0.2}{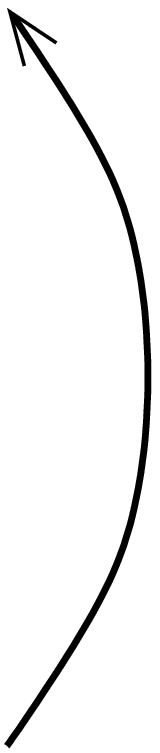}} \label{ax2}
\displaybreak[0]\\[1.2ex]
\tag{A3}
\raisebox{-18pt}{
\labellist
\tiny\hair 2pt
\pinlabel $i+j$ at  80  14
\pinlabel $i+j$ at  80 202
\pinlabel $i$   at -15 110
\pinlabel $j$   at  84 110 
\endlabellist
\figs{0.191}{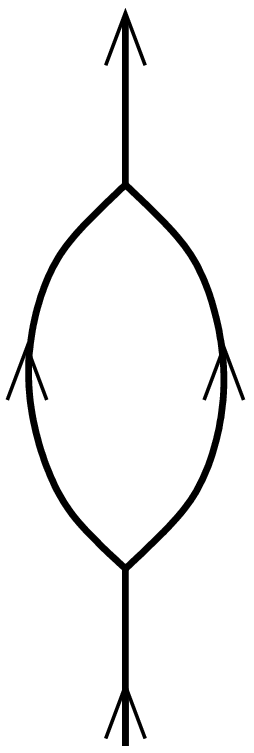}}\mspace{18mu}
&= \left[ i+j \atop i \right]
\raisebox{-18pt}{
\labellist
\tiny\hair 2pt
\pinlabel $i+j$ at 50 15
\endlabellist
\figs{0.191}{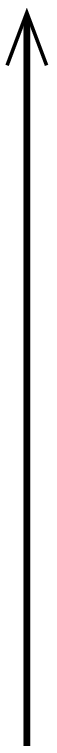}} \label{ax3}
\displaybreak[0]\\[1.2ex]
\tag{A4}
\raisebox{-18pt}{
\labellist
\tiny\hair 2pt
\pinlabel $i+j+k$ at 158  14
\pinlabel $i$     at  -9 202
\pinlabel $k$     at 190 202
\pinlabel $j$     at  86 202
\pinlabel $i+j$   at  14  86
\endlabellist
\figs{0.2}{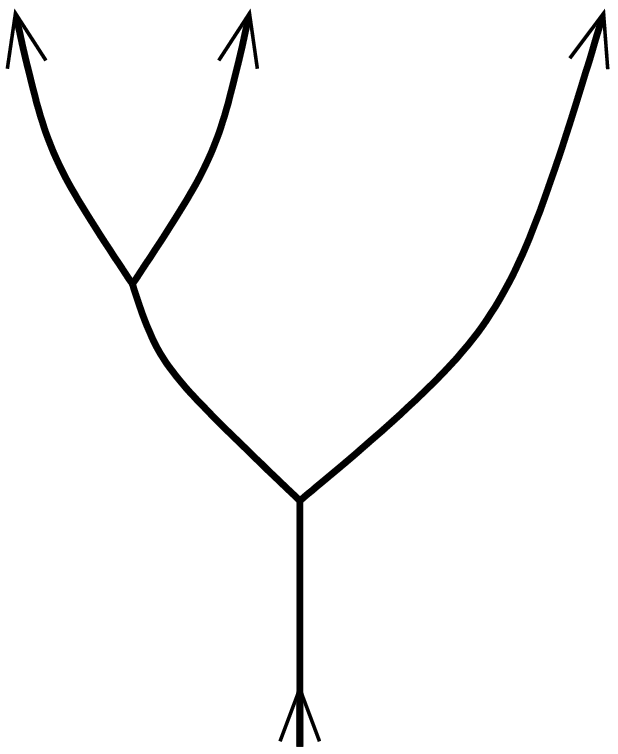}}\mspace{18mu}
&=
\mspace{6mu}
\raisebox{-18pt}{
\labellist
\tiny\hair 2pt
\pinlabel $i+j+k$ at 162  14
\pinlabel $i$     at  -9 202
\pinlabel $k$     at 190 202
\pinlabel $j$     at  86 202
\pinlabel $i+j$   at 160  86
\endlabellist
\figs{0.2}{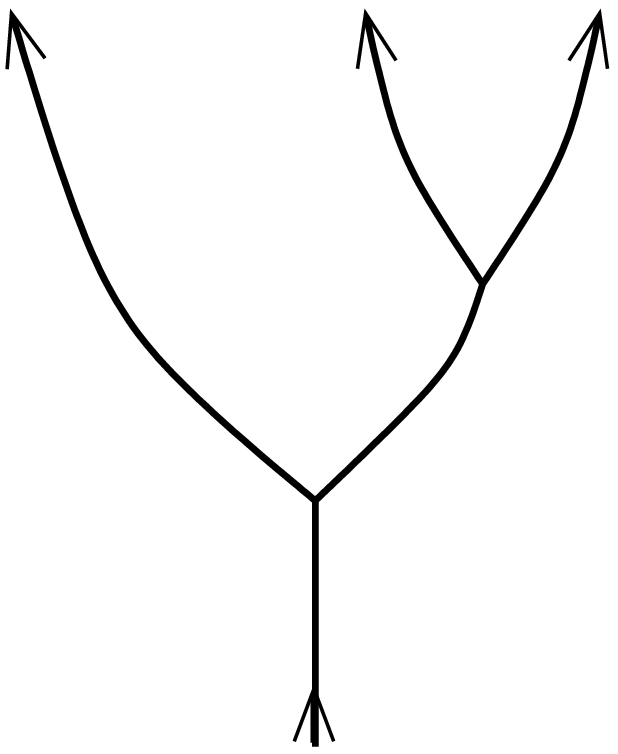}} \label{ax4}
\displaybreak[0]\\[1.2ex]
\tag{A5}
\raisebox{-18pt}{
\labellist
\tiny\hair 2pt
\pinlabel $2$ at -10  14
\pinlabel $2$ at -10 205
\pinlabel $1$ at 138  14
\pinlabel $1$ at 138 205
\pinlabel $1$ at -10 105
\pinlabel $2$ at 141 105
\pinlabel $1$ at  65 180
\pinlabel $1$ at  65  30
\endlabellist
\figs{0.2}{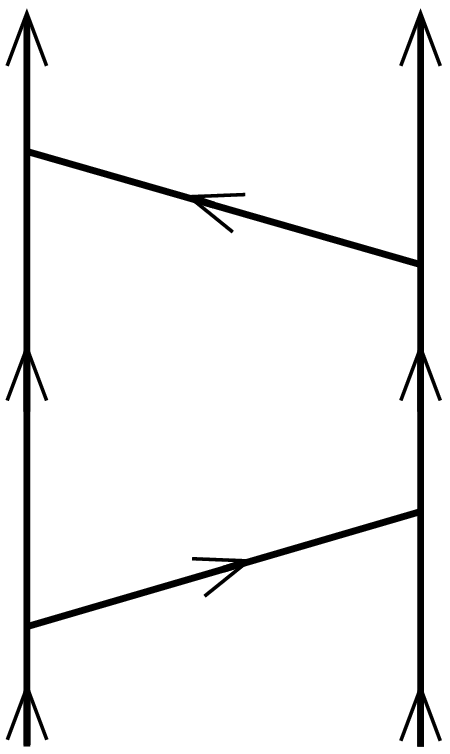}}\mspace{18mu}
&=
\mspace{6mu}
\raisebox{-18pt}{
\labellist
\tiny\hair 2pt
\pinlabel $2$ at -10  14
\pinlabel $2$ at -10 205
\pinlabel $1$ at 130  14
\pinlabel $1$ at 130 205
\pinlabel $3$ at  85 110 
\endlabellist
\figs{0.191}{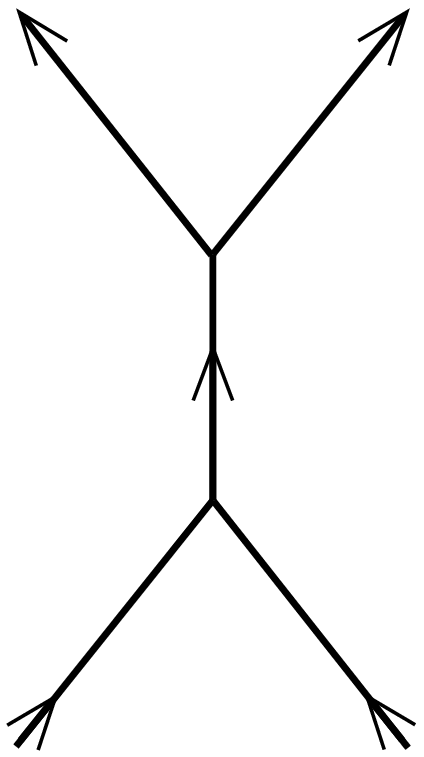}}\mspace{10mu}
+ q^2\
\raisebox{-18pt}{
\labellist
\tiny\hair 2pt
\pinlabel $2$ at  25 14
\pinlabel $1$ at 138 14
\endlabellist
\figs{0.192}{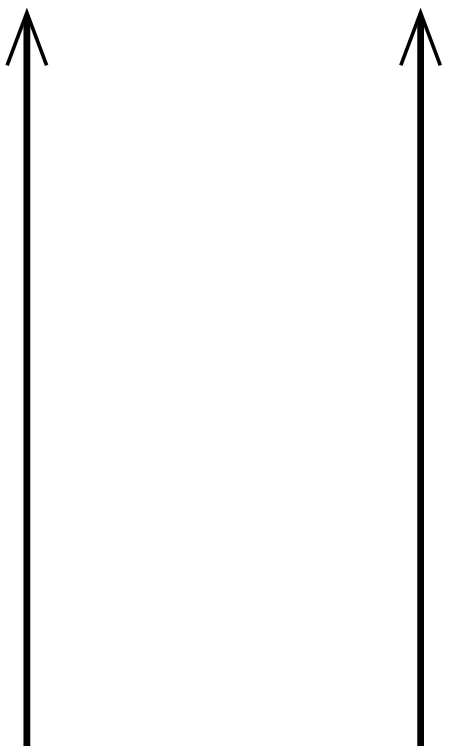}} \label{ax6}
\displaybreak[0]\\[1.2ex]
\tag{A6}
\raisebox{-18pt}{
\labellist
\tiny\hair 2pt
\pinlabel $3$ at -10  14
\pinlabel $3$ at -10 205
\pinlabel $1$ at 138  14
\pinlabel $1$ at 138 205
\pinlabel $2$ at -10 105
\pinlabel $2$ at 141 105
\pinlabel $1$ at  65 180
\pinlabel $1$ at  65  30
\endlabellist
\figs{0.2}{sqweb}}\mspace{18mu}
&=
\mspace{6mu}
\raisebox{-18pt}{
\labellist
\tiny\hair 2pt
\pinlabel $3$ at -10  14
\pinlabel $3$ at -10 205
\pinlabel $1$ at 130  14
\pinlabel $1$ at 130 205
\pinlabel $4$ at  85 110 
\endlabellist
\figs{0.191}{dumbell}}\mspace{10mu}
+ (q^2+q^4)\
\raisebox{-18pt}{
\labellist
\tiny\hair 2pt
\pinlabel $3$ at  25 14
\pinlabel $1$ at 138 14
\endlabellist
\figs{0.192}{idweb}}\label{ax5}
\displaybreak[0]\\[1.2ex]
\tag{A7}
\raisebox{-18pt}{
\labellist
\tiny\hair 2pt
\pinlabel $2$ at -10  14
\pinlabel $1$ at -10 205
\pinlabel $2$ at 138  14
\pinlabel $3$ at 138 205
\pinlabel $3$ at -10 105
\pinlabel $1$ at 141 105
\pinlabel $2$ at  65 180
\pinlabel $1$ at  65  30
\endlabellist
\figs{0.2}{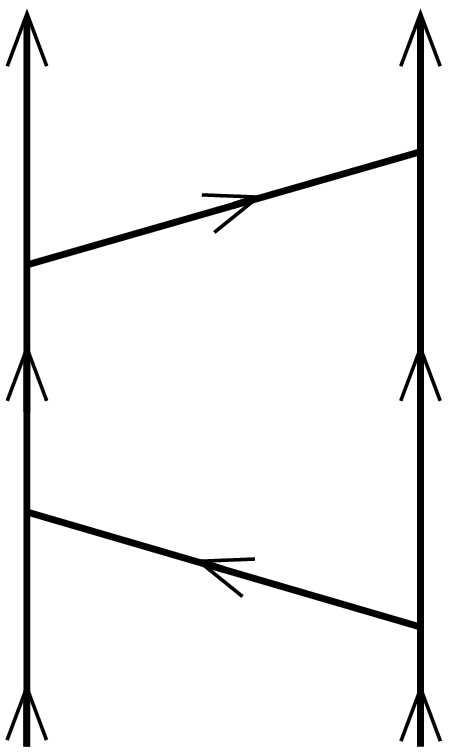}}\mspace{18mu}
&=
\mspace{6mu}
\raisebox{-18pt}{
\labellist
\tiny\hair 2pt
\pinlabel $2$ at -10  14
\pinlabel $1$ at -10 205
\pinlabel $2$ at 130  14
\pinlabel $3$ at 130 205
\pinlabel $4$ at  85 110 
\endlabellist
\figs{0.191}{dumbell}}\mspace{10mu}
+ q^2\
\raisebox{-18pt}{
\labellist
\tiny\hair 2pt
\pinlabel $2$ at -12  14
\pinlabel $1$ at -12 205
\pinlabel $2$ at 135  14
\pinlabel $3$ at 135 205
\pinlabel $1$ at  57 127
\endlabellist
\figs{0.192}{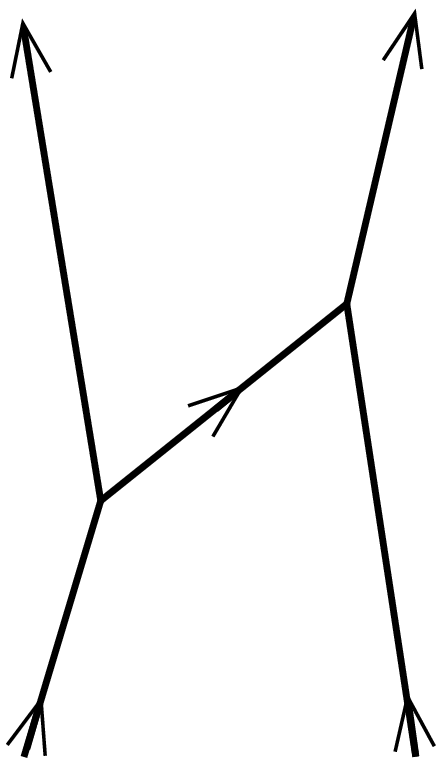}} \label{ax7}
\end{align}
In this paper we use the following (non-standard) convention for the quantum integers:
$$[n]=1+q^2+\ldots+q^{2(n-1)}=\frac{1-q^{2n}}{1-q^2}.$$
We define the quantum factorial and the binomial coefficients in the 
standard way:
$$[n]!=[1][2]\cdots[n],\quad \left[ n \atop m \right ]=\frac{[n]!}{[m]![n-m]!}.$$  
Although we have defined the axioms A1-A3, for arbitrary $i$, $j$ and $k$, in this paper we shall only need them in the cases 
where the indices $i$, $j$ and $k$ are from the set $\{1,2\}$.

The invariant is defined for link projections which have the form of the closure 
of a braid. All positive and negative crossings between strands labelled 2 are resolved in three different resolutions, 
and the value of the bracket is defined using the following relations:
\begin{gather}
\begin{split}
\raisebox{0pt}{
\labellist
\pinlabel $=$ at 245 112
\pinlabel $q^6$ at 340 115
\pinlabel $-\ q^{2}$ at 660 115
\pinlabel $+$ at 1030 115
\tiny\hair 2pt
\pinlabel $2$ at -12  19
\pinlabel $2$ at 175  19
\pinlabel $2$ at 382  19
\pinlabel $2$ at 570  19
\pinlabel $2$ at 757  19
\pinlabel $2$ at 758 206
\pinlabel $2$ at 947  19
\pinlabel $2$ at 945 206
\pinlabel $3$ at 780 110
\pinlabel $1$ at 927 110
\pinlabel $1$ at 852  38
\pinlabel $1$ at 852 177
\pinlabel $2$ at 1093  16
\pinlabel $2$ at 1095 206
\pinlabel $2$ at 1275  16
\pinlabel $2$ at 1272 206
\pinlabel $4$ at 1204 110
\endlabellist
\centering
\figs{0.23}{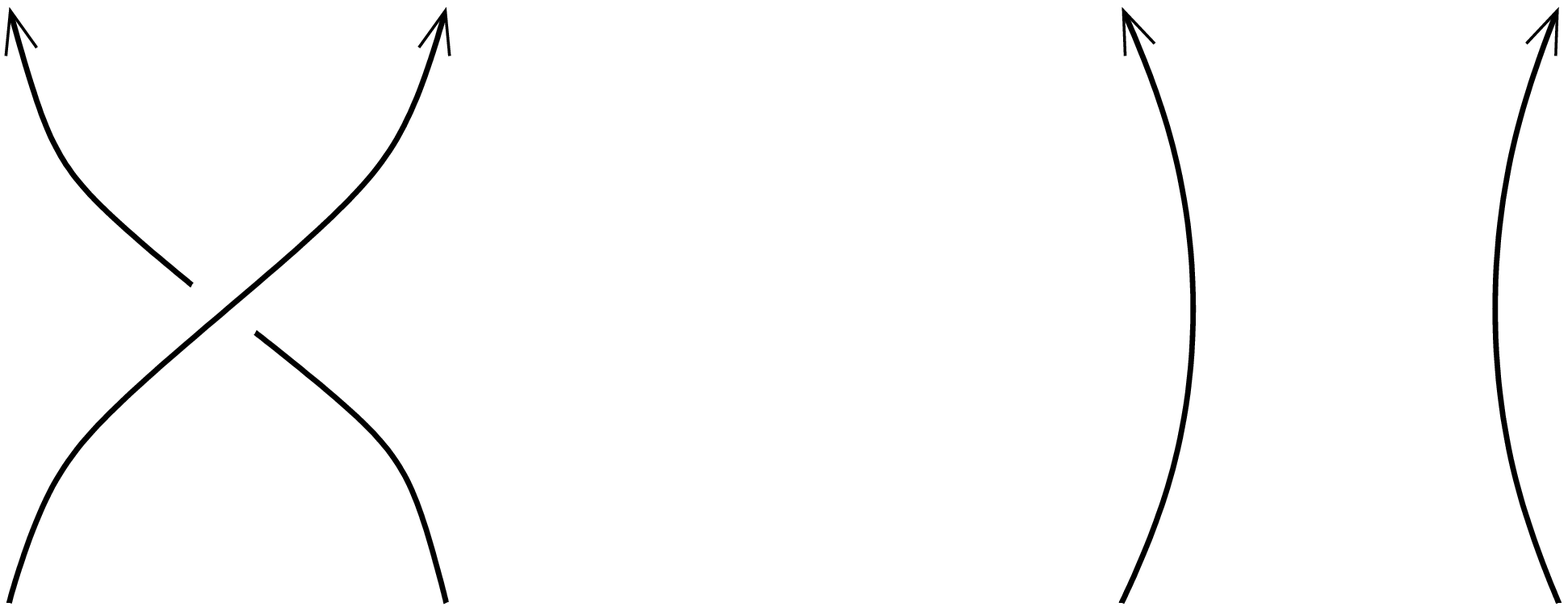}}
\\[1.5ex]
\raisebox{0pt}{
\labellist
\pinlabel $=$ at 245 112
\pinlabel $q^{-8}$ at 335 115
\pinlabel $-\ q^{-8}$ at 655 115
\pinlabel $+\ q^{-6}$ at 1010 115
\tiny\hair 2pt
\pinlabel $2$ at -12 17
\pinlabel $2$ at 175 17
\pinlabel $2$ at 386 17
\pinlabel $2$ at 569 17
\pinlabel $2$ at 385 204
\pinlabel $2$ at 568 204
\pinlabel $4$ at 496 110
\pinlabel $2$ at 731 17
\pinlabel $2$ at 735 204
\pinlabel $2$ at 923 17
\pinlabel $2$ at 920 204
\pinlabel $3$ at 755 107
\pinlabel $1$ at 903 107
\pinlabel $1$ at 829  40
\pinlabel $1$ at 829 176
\pinlabel $2$ at 1092 17
\pinlabel $2$ at 1277 17
\endlabellist
\centering
\figs{0.23}{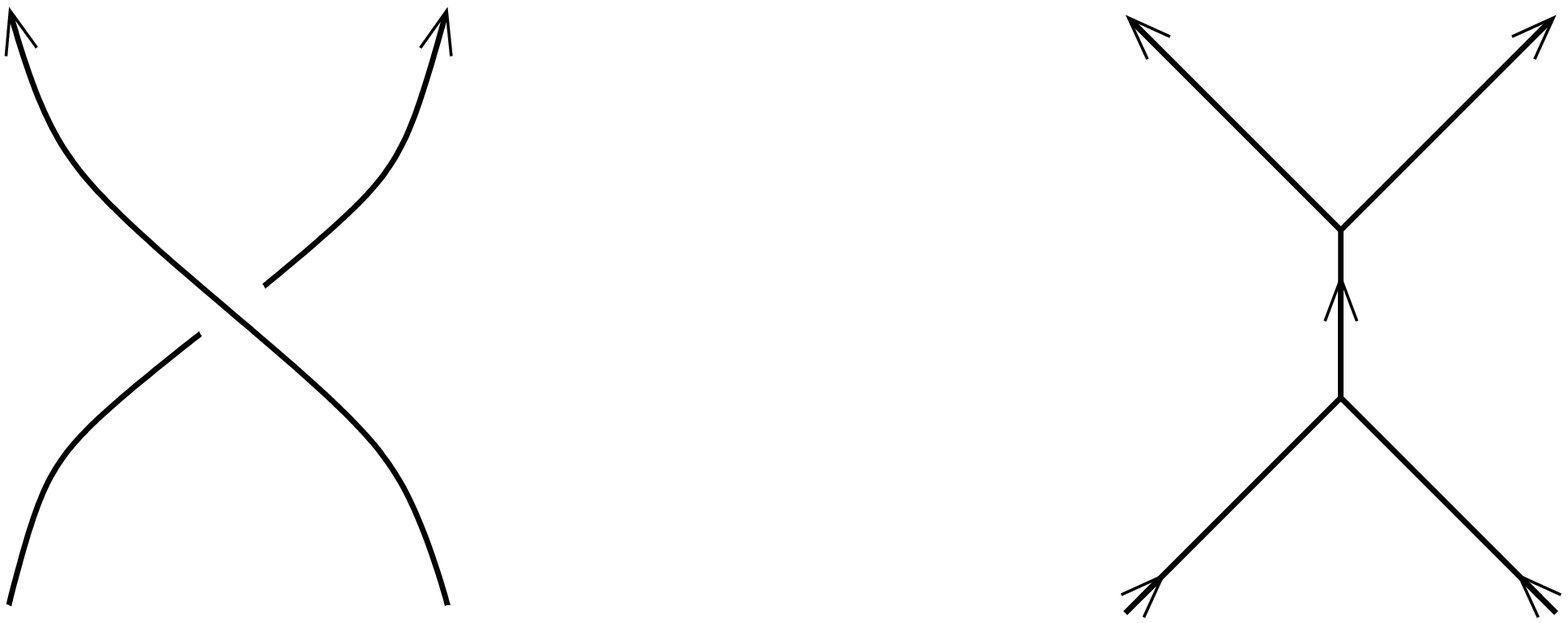}}
\medskip
\end{split}\label{eq:Xing22}
\end{gather}
The resolutions of the positive and negative crossings between strands labelled 1 and 2 are given by:
\begin{gather}
\begin{split}
\raisebox{0pt}{
\labellist
\pinlabel $=$ at 225 112
\pinlabel $-q^{2}$ at 315 115
\pinlabel $+$ at 660 115
\tiny\hair 2pt
\pinlabel $2$ at -12  16
\pinlabel $1$ at 175  16
\pinlabel $2$ at 376  16
\pinlabel $1$ at 560  16
\pinlabel $1$ at 378 206
\pinlabel $2$ at 561 206
\pinlabel $1$ at 469 125
\pinlabel $2$ at 752  16
\pinlabel $1$ at 754 206
\pinlabel $1$ at 935  16
\pinlabel $2$ at 935 206
\pinlabel $3$ at 863 110
\endlabellist
\centering
\figs{0.23}{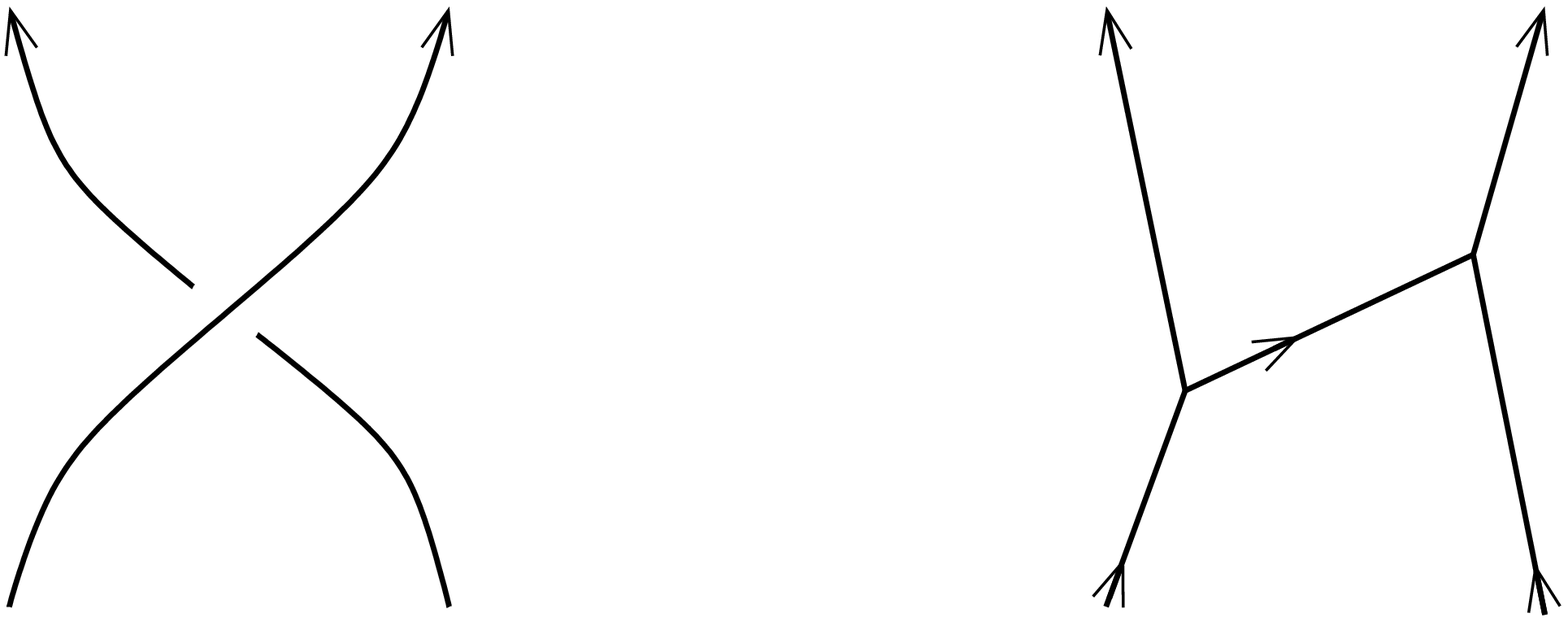}}
\\[1.5ex]
\raisebox{0pt}{
\labellist
\pinlabel $=$ at 225 112
\pinlabel $q^{-4}$ at 315 115
\pinlabel $-\ q^{-4}$ at 660 115
\tiny\hair 2pt
\pinlabel $2$ at -12  16
\pinlabel $1$ at 175  16
\pinlabel $2$ at 375  16
\pinlabel $1$ at 560  16
\pinlabel $1$ at 377 206
\pinlabel $2$ at 560 206
\pinlabel $3$ at 490 110
\pinlabel $2$ at 752  16
\pinlabel $1$ at 754 206
\pinlabel $1$ at 935  16
\pinlabel $2$ at 935 206
\pinlabel $1$ at 840 125
\endlabellist
\centering
\figs{0.23}{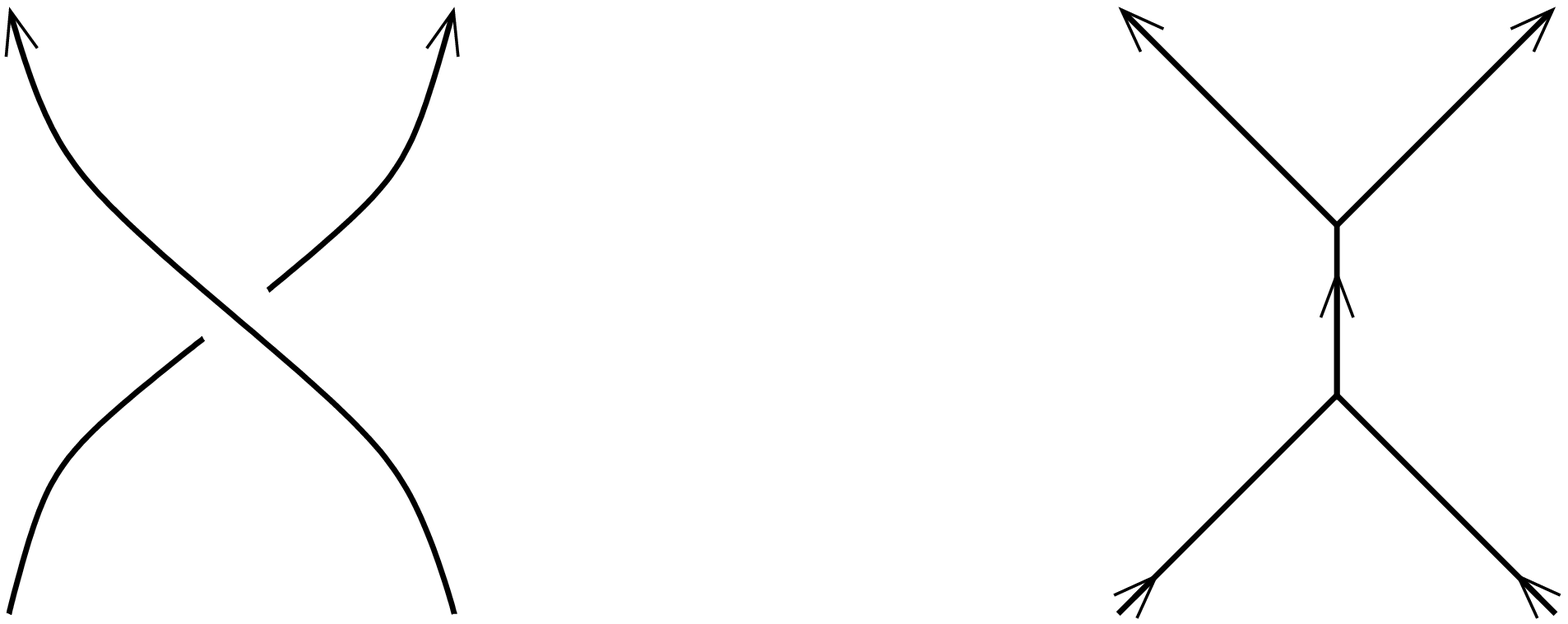}}
\medskip
\end{split}\label{eq:Xing12}
\end{gather}
The resolutions when the labels 1 and 2 are swapped one obtains by rotation around 
the y-axis. 

Finally the case of the crossings when both strands are labelled 1 is the 
same as in \cite{KR2}:
\begin{gather}
\begin{split}
\raisebox{0pt}{
\labellist
\pinlabel $=$ at 225 112
\pinlabel $-q^{2}$ at 315 115
\pinlabel $+$ at 660 115
\tiny\hair 2pt
\pinlabel $1$ at -12  16
\pinlabel $1$ at 175  16
\pinlabel $1$ at 376  16
\pinlabel $1$ at 560  16
\pinlabel $1$ at 752  16
\pinlabel $1$ at 754 206
\pinlabel $1$ at 935  16
\pinlabel $1$ at 935 206
\pinlabel $2$ at 863 110
\endlabellist
\centering
\figs{0.23}{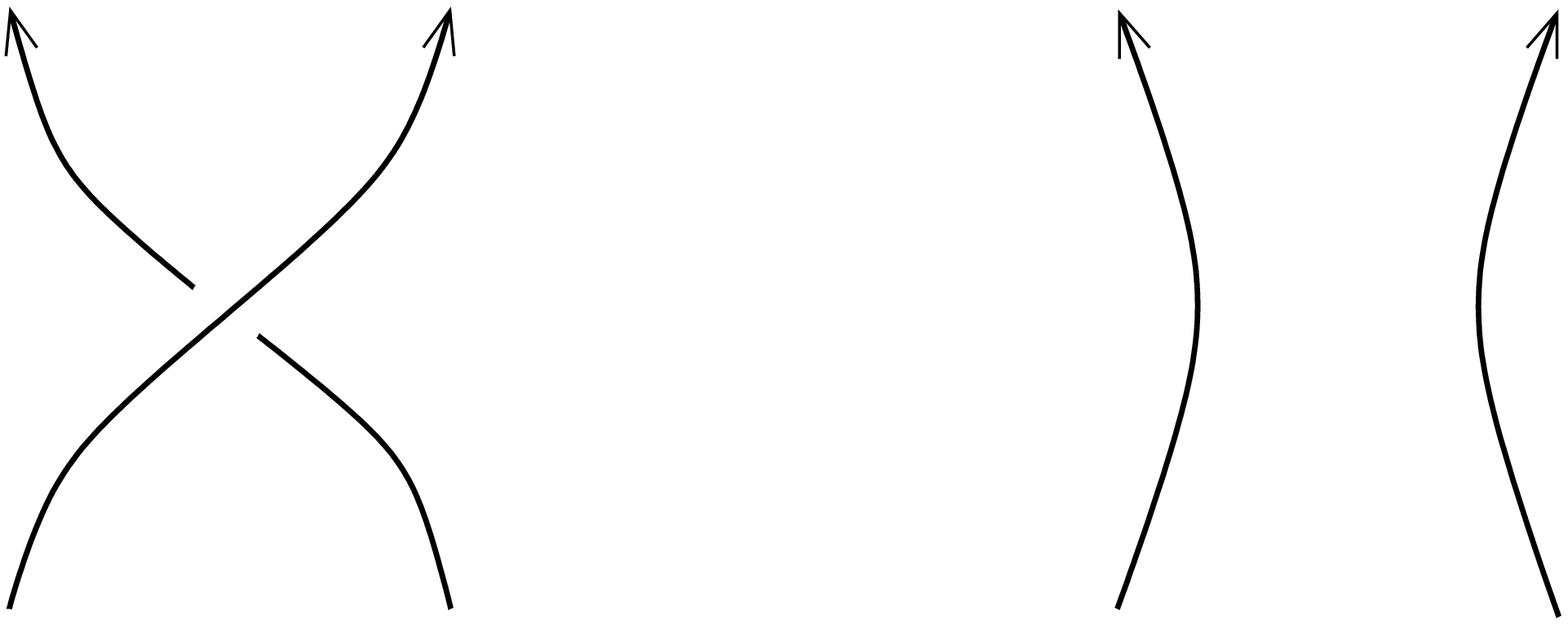}}
\\[1.5ex]
\raisebox{0pt}{
\labellist
\pinlabel $=$ at 225 112
\pinlabel $q^{-2}$ at 315 115
\pinlabel $-\ q^{-2}$ at 660 115
\tiny\hair 2pt
\pinlabel $1$ at -12  16
\pinlabel $1$ at 175  16
\pinlabel $1$ at 375  16
\pinlabel $1$ at 560  16
\pinlabel $1$ at 377 206
\pinlabel $1$ at 560 206
\pinlabel $2$ at 490 110
\pinlabel $1$ at 752  16
\pinlabel $1$ at 935  16
\endlabellist
\centering
\figs{0.23}{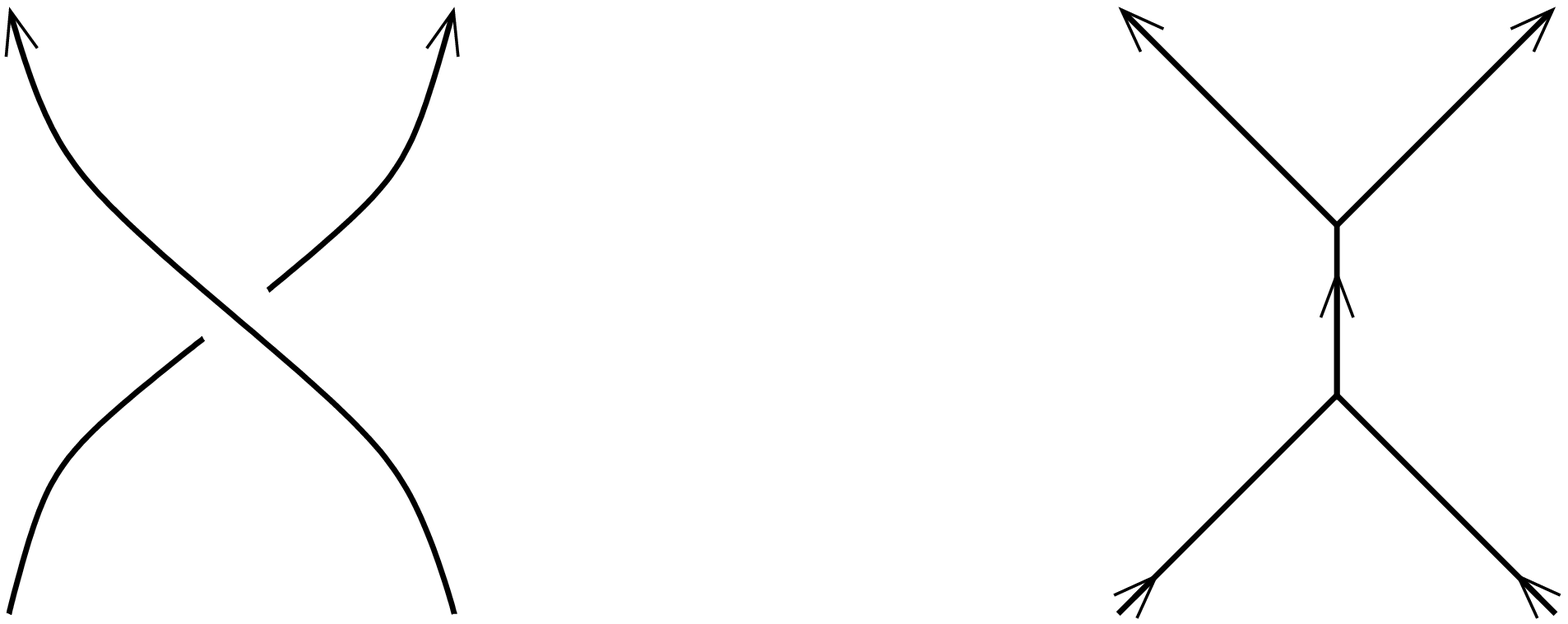}}
\medskip
\end{split}\label{eq:Xing11}
\end{gather}
Given the evaluation of trivalent graphs satisfying axioms A1-A7, we can define 
a polynomial $\brak{D}$ for each link diagram $D$ with components labelled by 1 and 2, 
using the resolutions in \eqref{eq:Xing22}, \eqref{eq:Xing12} and \eqref{eq:Xing11}. Analogously as in~\cite{MOY}, it 
can be shown that it is invariant under the second and the third Reidemeister move and has the 
following simple behaviour under the first Reidemeister move:
\begin{equation*}
\labellist
\tiny\hair 2pt
\pinlabel $2$ at  20 14
\pinlabel $2$ at 260 14
\endlabellist
\raisebox{-18pt}{
\figs{0.2}{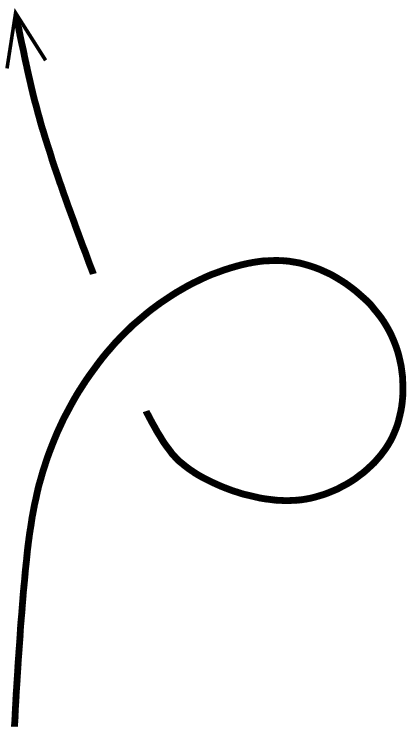}}
\
= 
\
\raisebox{-18pt}{
\figs{0.2}{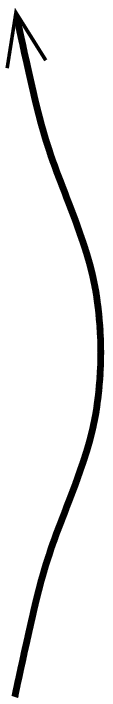}}
\mspace{20mu},\mspace{60mu}
\labellist
\tiny\hair 2pt
\pinlabel $2$ at  25 14
\pinlabel $2$ at 420 14
\endlabellist
\raisebox{-18pt}{
\figs{0.2}{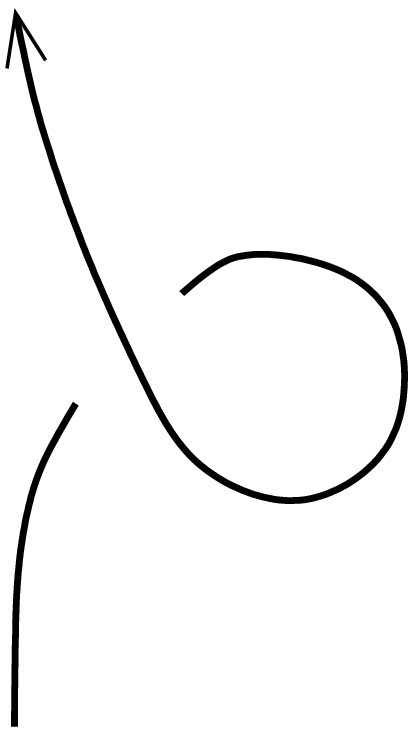}}
\
= 
\ t^{-2}q^{-2}
\raisebox{-18pt}{
\figs{0.2}{arcbent-ur}}
\end{equation*}
when the strand is labelled 2 and 
\begin{equation*}
\labellist
\tiny\hair 2pt
\pinlabel $1$ at  20 14
\pinlabel $1$ at 260 14
\endlabellist
\raisebox{-18pt}{
\figs{0.2}{R1plus}}
\
= 
\
\raisebox{-18pt}{
\figs{0.2}{arcbent-ur}}
\mspace{20mu},\mspace{60mu}
\labellist
\tiny\hair 2pt
\pinlabel $1$ at  25 14
\pinlabel $1$ at 420 14
\endlabellist
\raisebox{-18pt}{
\figs{0.2}{R1minus}}
\
= 
\ -t^{-1}q^{-1}
\raisebox{-18pt}{
\figs{0.2}{arcbent-ur}}
\end{equation*}
when the strand is labelled 1. 

To obtain a genuine knot invariant, we therefore have to multiply the bracket 
by the following overall factor:

\begin{equation*}
I(D)=(-tq)^{\frac{-n^{1}_++ n^{1}_-+s_1(D) -2n^{2}_+ +2n^{2}_-+2s_2(D)}{2}} \langle D\rangle,
\end{equation*}
where $n^{i}_+$ and $n^{i}_-$ denote the number of positive and  
negative crossings, respectively, between two strands labelled $i$ and $s_i(D)$ denotes the number of 
strands labelled $i$, for $i=1,2$.

\section{The categorification of the MOY calculus}
\label{sec:triply}
In this section we show which bimodule to associate to a web and show that 
these bimodules satisfy axioms A3-A7 up to isomorphism. The proof that A1 
and A2 are also satisfied will be given in Section~\ref{sec:Markov} after we have explained 
the Hochschild homology. We only explain the general idea and work out the bits which involve edges with 
higher labels and which have not been explained by Khovanov in~\cite{Kh}.

Let $R=\bC[x_1,\ldots,x_n]$ be the ring of complex polynomials in $n$ 
variables. We introduce a grading on $R$, by defining $\deg x_i=2$, for every $i=1,\ldots,n$. This grading is called the $q$-{\em grading}. For any partition 
$i_1,\ldots,i_k$ of $n$, let $R_{i_1\cdots i_k}$ denote the subring of $R$ of complex polynomials 
which are invariant under the product of the symmetric groups 
$S_{i_1}\times\cdots\times S_{i_k}$.   
For starters we associate a bimodule to each MOY-web  
(see Section~\ref{sec:MOY}). Suppose 
$\Gamma$ is a MOY web with $k$ bottom ends labelled by $i_1,\ldots,i_k$ 
and $m$ top ends labelled by $j_1,\ldots,j_m$. Recall that 
$i_1+\cdots+i_k=j_1+\cdots +j_m$ holds and let $R$ have exactly that number of 
variables. We associate an 
$R_{j_1\cdots j_m}-R_{i_1\cdots i_k}$-bimodule to $\Gamma$. We read 
$\Gamma$ from bottom to top. To the bottom edges we associate the bimodule 
$R_{i_1\cdots i_k}$. When we move up in the web we encounter a $\YGraph$-shaped 
or a $\UYGraph$-shaped bifurcation. 
The $\YGraph$-shape will always correspond to induction and the $\UYGraph$ to 
restriction, e.g. if we first encounter a $\YGraph$-shaped bifurcation which 
splits $i_1$ 
into $i_1^0$ and $i_1^1$, then we tensor $R_{i_1\cdots i_k}$ on the left 
with $R_{i_1^0 i_1^1 \cdots i_k}$ over $R_{i_1\cdots i_k}$. We will write 
this tensor product as 
$$R_{i_1^0 i_1^1 \cdots i_k}\otimes_{i_1\cdots i_k} R_{i_1\cdots i_k}.$$
If we first 
encounter a $\UYGraph$-shaped bifurcation with joints $i_1$ and $i_2$, then 
we restrict the left action on $R_{i_1\cdots i_k}$ to 
$R_{i_1+i_2\cdots i_k}$. Note that the latter action goes unnoticed until 
tensoring again at some next $\YGraph$-shaped bifurcation, in which case we tensor 
over the smaller ring, or until taking the Hochschild homology in a later stage. As we go up in the web use either 
induction or restriction at each bifurcation. This way you obtain an 
$R_{j_1\cdots j_m}-R_{i_1\cdots i_k}$-bimodule associated to $\Gamma$, which 
we denote by $\widehat{\Gamma}$. The identity web in Figure~\ref{fig:idweb-i1ik} 
\begin{figure}[ht!]
\labellist
\tiny\hair 2pt
\pinlabel $i_1$ at  -8 15
\pinlabel $i_2$ at  78 15
\pinlabel $i_3$ at 235 15
\pinlabel $\dotsm$ at 170 105
\endlabellist
\centering
\figs{0.25}{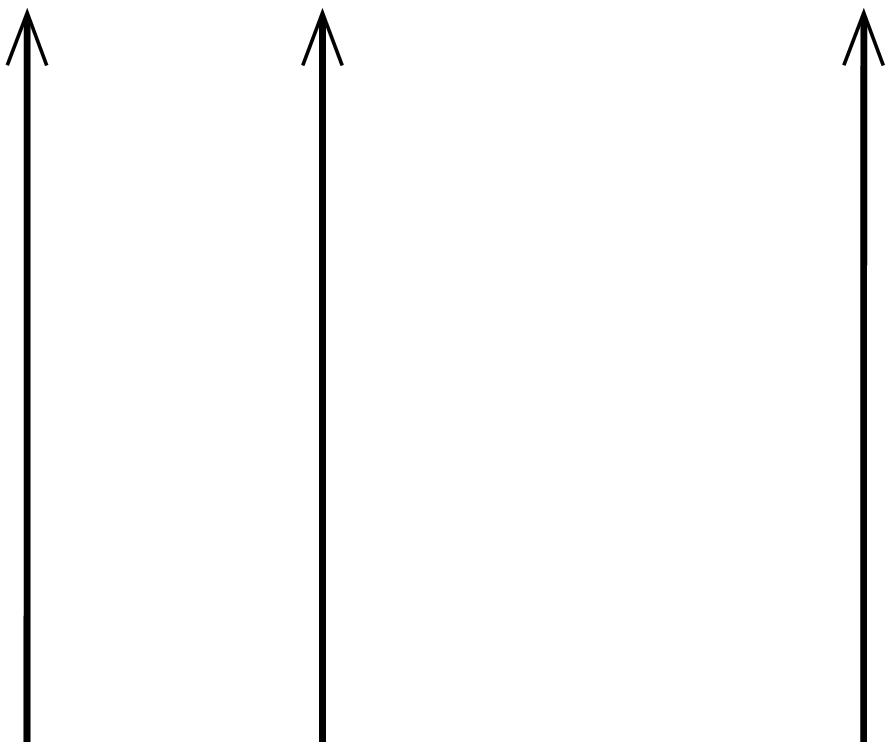}
\caption{The identity web $\arcsup_{i_1\cdots i_k}$}
\label{fig:idweb-i1ik}
\end{figure}
whose edges are labelled by $i_1,\ldots,i_k$ will always be denoted by $\arcsup_{i_1\cdots i_k}$ 
and $\hatarcsup_{i_1\cdots i_k}=R_{i_1\cdots i_k}$.
The dumbell web in Figure~\ref{fig:dumbell-i1i2} 
\begin{figure}[ht!]
\labellist
\tiny\hair 2pt
\pinlabel $i_1$ at -10 10
\pinlabel $i_1$ at -10 205
\pinlabel $i_2$ at 133 10
\pinlabel $i_2$ at 135 205
\pinlabel $i_1+i_2$ at 110 108
\endlabellist
\centering
\figs{0.25}{dumbell}
\caption{The dumbell web $\dumbell_{i_1i_2}$}
\label{fig:dumbell-i1i2}
\end{figure}
whose outer 
edges are labelled $i_1$ and $i_2$ will always be denoted by $\dumbell_{i_1i_2}$ and 
$\hatdumbell_{i_1i_2}=R_{i_1i_2}\ot_{i_1+i_2}R_{i_1i_2}$.  Since the bimodules that we use are graded, we can apply 
a grading shift. In the text and when using small symbols we denote a positive shift of $k$ values 
applied to a bimodule $M$ by $M\{k\}$. When we use MOY-type pictures we denote the same shift by $q^k$. 
Note that we also do not put a hat on top of these pictures to avoid too much notation and 
unnecessarily large figures. We hope that this does not leed to confusion.

Note that for our construction we need to choose a height function on the 
web first. However, it is easy to see that the bimodule does not depend on 
the choice of this height function. 

Now that we know which bimodule to associate to a MOY web, we will 
show some direct sum decompositions for bimodules associated to certain 
MOY webs. These decompositions are necessary to show that our bimodules 
categorify the MOY calculus indeed and also to show that the link homology is 
invariant, up to isomorphism, under the Reidemeister moves.

\begin{lem} 
\label{lem:digon}
Let $\digonup_{ij}$ be the digon in A3. Then we have 
$$\hatdigonup_{ij}\cong 
\qbin{i+j}{i}\hatarcup_{i+j}$$
Note that the quantum binomial can be written as a sum of powers of $q$. 
By $$\qbin{i+j}{i}\hatarcup_{i+j}$$ we mean the corresponding direct sum of copies of $\hatarcup_{i+j}$, 
where each copy is shifted by the correct power of $q$. 
\end{lem}   
\begin{proof}
Note that $R_{ij}$ as an $R_{i+j}$-(bi)module is isomorphic to 
$H^*_{U(i+j)}(G(i,i+j))$, the $U(i+j)$ equivariant cohomology of the 
complex Grassmannian $G(i,i+j)$. Therefore the Schur polynomials 
$\pi_{k_1\cdots k_i}$ in the first $i$ variables, for $0\leq k_i\leq \cdots 
\leq k_1\leq j$ form a basis of $R_{ij}$ as an $R_{i+j}$-(bi)module 
(see~\cite{F} for example). Alternatively we can use $G(j,i+j)$ and 
obtain that the Schur polynomials $\pi'_{\ell_1\cdots\ell_j}$ in the last 
$j$ variables form a basis of $R_{ij}$ as an $R_{i+j}$-(bi)module, 
for $0\leq \ell_j\leq\cdots\leq\ell_1\leq i$. This shows that 
$$R_{ij}\cong \qbin{i+j}{i} R_{i+j}$$
holds. The proof of this lemma follows, 
since $\hatdigonup_{ij}=R_{ij}\otimes_{i+j}R_{i+j}$ and 
$\hatarcup_{i+j}=R_{i+j}$.   
\end{proof}

Before we continue with square decompositions we define the following 
bimodule maps:

\begin{defn} 
\label{defn:1k-mucomu}
We define the $R_{1k}$-bimodule maps 
$$\mu_{1k}\colon \hatdumbell_{1k}\to \hattwoarcs_{1k}\quad\text{and}\quad \Delta_{1k}
\colon \hattwoarcs_{1k}\to 
\hatdumbell_{1k}$$
by 
$$\mu_{1k}(a\stens b)=ab\quad\mbox{and}\quad\Delta_{1k}(1)=\sum_{j=0}^{k}(-1)^j e_j(\sum_{i=0}^{k-j} x_1^{k-j-i}
\stens x_1^i).$$
\noindent The elements 
$e_i$ are the elementary symmetric polynomials in $k+1$ variables. 
\end{defn}
\noindent The formula for $\Delta_{1k}$ can also be written as 
$$ \Delta_{1k}(1)=\sum_{j=0}^{k}(-1)^j x_1^{k-j} \stens e'_j,$$
where $e'_j$ is the $j$-th elementary symmetric polynomial in the last $k$ 
variables $x_2,\ldots,x_{k+1}$.

It is not hard to see that $\mu$ and $\Delta$ are $R_{1k}$ bimodule maps indeed. 
One can check this by direct computation or, as above, note that 
$R_{1k}$ is isomorphic to $H^*_{U(k+1)}(G(1,k))$ as an $R_{k+1}$-module. 
The maps above are well known (it is an immediate consequence of exercise 1.1 of lecture 4 in \cite{F}, 
for example) to be the multiplication and comultiplication 
in this commutative Frobenius extension with respect to the trace defined by 
$\mbox{tr}(x_1^k)=1$. The 
observation now follows from the fact that the 
multiplication and comultiplication in a commutative Frobenius extension $A$ are 
always $A$-bimodule maps. 

When it is not immediately clear to which variables one applies a multiplication or comultiplication 
we will indicate them in a superscript. When there is no confusion possible we 
will write $ab$ for $\mu_{1,k}(a\stens b)$. 
 
\begin{defn}
\label{defn:22-mucomu}
We define the $R_{22}$-bimodule maps 
$$\mu_{22}\colon \hatdumbell_{22}\to \hattwoarcs_{22}\quad\text{and}\quad \Delta_{22}\colon 
\hattwoarcs_{22}\to 
\hatdumbell_{22}$$
by 
$$\mu_{22}(a\stens b)=ab$$
and
$$\Delta_{22}(1)=\pi_{22}\stens 1-\pi_{21}\cdot\pi'_{10}+\pi_{20}\stens 
\pi'_{11}+\pi_{11}\cdot\pi'_{20}-\pi_{10}\cdot\pi'_{21}+1\cdot\pi'_{22}+\leftrightarrow.$$
The $\pi_{ij}$ and $\pi'_{ij}$ are the Schur polynomials in $x_1,x_2$ and $x_3,x_4$ respectively. The 
$\leftrightarrow$ indicates the terms which are obtained from all the previous terms by interchanging 
the $\pi_{ij}$ and the $\pi'_{ij}$ so that $\Delta$ becomes cocommutative.
\end{defn}

\noindent One easily checks that $\mu_{22}$ and $\Delta_{22}$ are $R_{22}$-bimodule maps by direct 
computation. The map $\Delta_{22}$ 
is the comultiplication in $H^*_{U(4)}(G(2,4))$ with respect 
to the trace defined by $\mbox{tr}(\pi_{22})=1$.
Sometimes it will be useful to rewrite the formula of $\Delta_{22}$ entirely in terms of the $\pi'_{ij}$ 
and elements of $R_4$. To do so use the following relations:
\begin{align*}
\pi_{10} &= \pi_{1000}-\pi'_{10}\\
\pi_{11} &= \pi_{1100}-\pi'_{10}\pi_{1000}+\pi'_{20}\\
\pi_{20} &= \pi_{2000}-\pi'_{10}e_1+\pi'_{11}\\
\pi_{21} &= \pi_{2100}-\pi'_{10}(\pi_{2000}+\pi_{1100})+\pi'_{20}\pi_{1000}
+\pi'_{11}\pi_{1000}-\pi'_{21}\\
\pi_{22} &= \pi_{2200}-\pi'_{10}\pi_{2100}+\pi'_{11}\pi_{2000}+\pi'_{20}\pi_{1100}-\pi'_{21}\pi_{1000}
+\pi'_{22}
\end{align*}

We can now prove some square decompositions. 

\begin{lem}
\label{lem:1112-square}
Let $\squareup_{1112}$ be the square in A6. 
The following decomposition holds
$$\hatsquareup_{1112}\cong \hatdumbell_{21}\oplus \hattwoarcs_{21}\{2\}.$$
\end{lem}    
\begin{proof} Note that 
  $$\hatsquareup_{1112}=R_{111}\otimes_{12} R_{111}\otimes_{21} R_{21}\quad\mbox{and}\quad
\hatdumbell_{21}=R_{21}\otimes_3 R_{21}.$$
\begin{figure}[ht!]
\labellist
\tiny\hair 2pt
\pinlabel $2$ at -8  12
\pinlabel $2$ at -8 236
\pinlabel $1$ at 121  12
\pinlabel $1$ at 121 236
\pinlabel $1$ at  15  85
\pinlabel $1$ at  70  25
\pinlabel $1$ at  15 168
\pinlabel $1$ at  70 226
\pinlabel $2$ at 102  90 
\pinlabel $2$ at 102 160
\pinlabel $3$ at  74 125
\pinlabel $2$ at 276  12
\pinlabel $2$ at 276 236
\pinlabel $1$ at 404  12
\pinlabel $1$ at 404 236
\pinlabel $3$ at 378 130
\pinlabel $2$ at 350  78
\pinlabel $2$ at 350 166
\pinlabel $1$ at 295 88
\pinlabel $1$ at 295 164
\pinlabel $1$ at 320  30
\pinlabel $1$ at 320 212
\small\hair 2pt
\pinlabel $\Gamma_1$ at  65 -50
\pinlabel $\Gamma_2$ at 345 -50
\endlabellist
\centering
\figs{0.25}{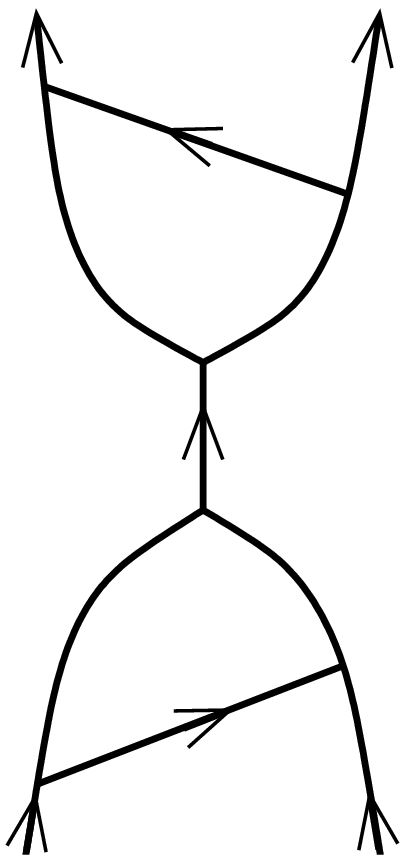}\hspace{8.4ex}\figs{0.25}{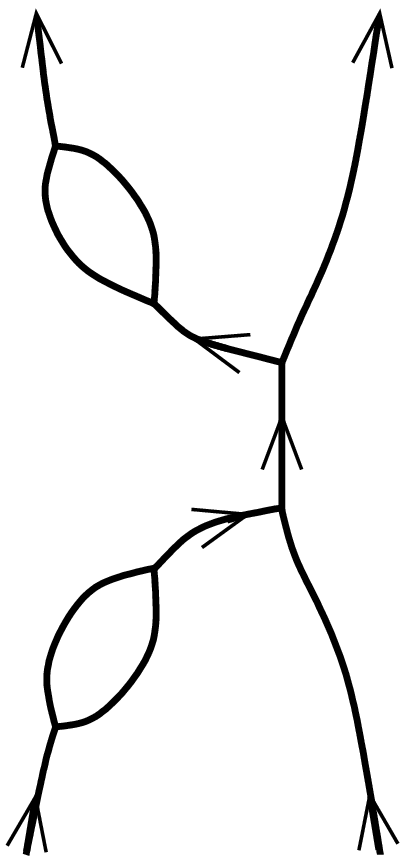}
\\[2.0ex]
\caption{Intermediate webs $\Gamma_1$ and $\Gamma_2$}
\label{fig:s1112-g1g2}
\end{figure}
Let $\Gamma_i$ be the webs in Figure~\ref{fig:s1112-g1g2}, for $i=1,2$. Then 
we have 
$$\widehat{\Gamma_1}=R_{111}\otimes_{12} R_{12}\otimes_{3} R_{111}\otimes_{21} R_{21}\cong 
R_{111}\otimes_{21} R_{21}\otimes_{3} R_{111}\otimes_{21} R_{21}\cong \widehat{\Gamma_2}.
$$  
The map $f\colon \hatsquareup_{1112}\to \hatdumbell_{21}$ is the composite of 
$$\hatsquareup_{1112}{\xrightarrow{\ \ f_1\ \ }}\widehat{\Gamma_1}\{-4\}\cong \widehat{\Gamma_2}\{-4\}
{\xrightarrow{\ \ f_2\ \ }}\hatdumbell_{21}.$$
We define $f_1$ by 
$$f_1(a\stens b\stens c)=a\stens \Delta_{12}(b)\stens c,$$
where $\Delta_{12}$ is defined as in 
Definition~\ref{defn:1k-mucomu}. The map $f_2$ is defined by applying to both 
digons the map   
$R_{111}\to R_{21}\{2\}$ which corresponds to the projection onto the second summand in the 
decomposition $R_{111}\cong R_{21}\oplus x_1R_{21}$, e.g. we have  
$$x_1\mapsto 1\quad\mbox{and}\quad x_2=(x_1+x_2)-x_1\mapsto -1\quad\mbox{and}\quad 
x_1^2=(x_1+x_2)x_1-x_1x_2\mapsto x_1+x_2.$$

Similarly define $g\colon \hatdumbell_{21}\to \hatsquareup_{1112}$ as the composite of 
$$\hatdumbell_{21}{\xrightarrow{\ \ g_1\ \ }}\widehat{\Gamma_2}\cong\widehat{\Gamma_1}
{\xrightarrow{\ \ g_2\ \ }}\hatsquareup_{1112}.$$
We define $g_1$ by applying twice the inclusion map $R_{21}\hookrightarrow R_{111}$ to 
create the digons.
The map $g_2$ is defined by 
$$g_2(a\stens b\stens c\stens d)=a\stens bc\stens d,$$
where $c$ is mapped to $R_{12}$ by the inclusion map before 
applying $\mu_{12}$. One easily verifies by 
direct computation that $fg=\mbox{id}$, so $\hatdumbell_{21}$ is a direct summand of $\hatsquareup_{1112}$. 

To show that $\hattwoarcs_{21}$ is a direct summand as well, we also use an intermediate web, denoted 
$\Gamma_3$ and shown in Figure~\ref{fig:s1112-g3}. 
\begin{figure}[ht!]
\labellist
\tiny\hair 2pt
\pinlabel $2$ at 18 12
\pinlabel $2$ at 18 230
\pinlabel $1$ at 126 12
\pinlabel $1$ at  -5 120
\pinlabel $1$ at  78 120
\endlabellist
\centering
\figs{0.25}{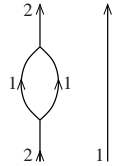}
\caption{Intermediate web $\Gamma_3$}
\label{fig:s1112-g3}
\end{figure} 
Note that $\widehat{\Gamma}_{3}=R_{111}\ot_{21} R_{21}$. 
Define $h\colon \hatsquareup_{1112}\to \hattwoarcs_{21}$ as the composite of 
$$\hatsquareup_{1112}{\xrightarrow{\ \ h_1\ \ }}\widehat{\Gamma_3}{\xrightarrow{\ \ h_2\ \ }}
\hattwoarcs_{21}\{2\}.$$
In this case $h_1$ is given by 
$$h_1(a\stens b\stens c)=ab\stens c$$
and $h_2$ by applying the same projection $R_{111}\to R_{21}\{2\}$ as above. Inversely we define 
$j\colon \hattwoarcs_{21}\to \hatsquareup_{1112}$ as the composite of 
$$\hattwoarcs_{21}\{2\}{\xrightarrow{\ \ j_1\ \ }}\widehat{\Gamma_3}\{2\}{\xrightarrow{\ \ j_2\ \ }}
\hatsquareup_{1112}.$$
The first map is defined by the inclusion $R_{21}\hookrightarrow R_{111}$. The second map is defined 
by 
$$j_2(a\stens b)=\Delta_{11}^{x_2,x_3}(a)\stens b.$$
Again, by direct computation, it is straightforward to check that $hj=-2\mbox{id}$, so 
$\hattwoarcs_{21}\{2\}$ is 
also a direct summand of $\hatsquareup_{1112}$. 

Also by direct computation one easily checks that 
$hg=0$ and $fj=0$. Finally we have to show that $(g,j)$ is 
surjective. Since all maps 
involved are bimodule maps and $R_{111}\cong R_{21}\oplus x_2R_{21}$ and $R_{111}\cong R_{12}\oplus 
x_3R_{12}$, we only have to show that 
$1\stens 1\stens 1$ and $1\stens x_2\stens 1$ are in its image. We have $g(1\stens 1)=1\stens 1\stens 1$ and 
$j(1)+g(x_3\stens 1)=1\stens x_2\stens 1$. This finishes the proof of the lemma.  
\end{proof}

\begin{lem}
\label{lem:1122-square}
Let $\hatsquareup_{1122}$ be the square in A5. Then we have
$$\hatsquareup_{1122}\cong \hatdumbell_{31}\oplus \hattwoarcs_{31}\{2\}\oplus \hattwoarcs_{31}\{4\}.$$
\end{lem}
\proof
The arguments are analogous to the ones used in the proof of Lemma~\ref{lem:1112-square}. To show that 
$\hatdumbell_{22}$ is a direct summand of $\hatsquareup_{1122}$, use the intermediate webs in Figure~\ref{fig:s1122-g1g2}. 
\begin{figure}[ht!]
\labellist
\tiny\hair 2pt
\pinlabel $3$ at  -6  12
\pinlabel $3$ at  -6 235
\pinlabel $1$ at 120  12
\pinlabel $1$ at 120 235
\pinlabel $2$ at  10  83
\pinlabel $1$ at  65  24
\pinlabel $2$ at  10 175
\pinlabel $1$ at  65 226
\pinlabel $2$ at 103  90 
\pinlabel $2$ at 103 160
\pinlabel $4$ at  75 124
\pinlabel $3$ at 277  12
\pinlabel $3$ at 277 235
\pinlabel $1$ at 402  12
\pinlabel $1$ at 402 235
\pinlabel $4$ at 378 124
\pinlabel $3$ at 350  80
\pinlabel $3$ at 350 167
\pinlabel $2$ at 292 86
\pinlabel $2$ at 292 164
\pinlabel $1$ at 320  30
\pinlabel $1$ at 320 220
\endlabellist
\centering
\figs{0.25}{interwebg1}\hspace{8.4ex}\figs{0.25}{interwebg2}
\caption{Intermediate webs for $\squareup_{1122}$}
\label{fig:s1122-g1g2}
\end{figure}

With the same notation as before let 
$$f_1(a\stens b\stens c)=a\stens \Delta_{22}(b)\stens c$$ 
and for $f_2$ use for both digons the map $R_{211}\to R_{31}\{4\}$ which corresponds to the 
projection onto the third direct summand in the decomposition $R_{211}\cong R_{31}\oplus x_3R_{31}\oplus 
x_3^2R_{31}$. 

For $g_1$ we use twice the inclusion $R_{31}\hookrightarrow R_{211}$ to create the digons and 
$$g_2(a\stens b\stens c\stens d)=a\stens bc\stens d.$$

To show that $\hattwoarcs_{22}\{2\}\oplus\hattwoarcs_{22}\{4\}$ is a direct summand 
of $\hatsquareup_{1122}$ use the intermediate 
web in Figure~\ref{fig:s1122-g3}. 

\begin{figure}[ht!]
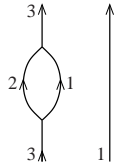

\labellist
\tiny\hair 2pt
\pinlabel $3$ at  20 13
\pinlabel $3$ at  20 230
\pinlabel $1$ at 125 13
\pinlabel $2$ at  -8 120
\pinlabel $1$ at  79 120
\endlabellist
\centering
\figs{0.25}{interwebg3}
\caption{Intermediate web}
\label{fig:s1122-g3}
\end{figure} 

\n Define $h_1$ by 
$$h_1(a\stens b\stens c)=ab\stens c$$
and $h_2$ by applying the map $R_{211}\to R_{31}\{2\}\oplus R_{31}\{4\}$ corresponding to the projection on the last 
two direct summands in the decomposition of $R_{211}$ above. 

For $j_1$ use the map $R_{31}\{2\}\oplus R_{31}\{4\}\to R_{211}$ defined by 
$$(1,0)\mapsto 1\quad\mbox{and}\quad (0,1)\mapsto x_3$$
to create the digon. Define $j_2$ by 
$$j_2(a\stens b)=\Delta_{11}^{x_3,x_4}(a)\stens b.$$ 
  
An easy calculation shows that 
$$hj=
\begin{pmatrix}
1&-x_4\\
0&1
\end{pmatrix}
,$$
which is invertible. This shows that $R_{31}\{2\}\oplus R_{31}\{4\}$ is a direct summand of 
$\hatsquareup_{1122}$. 

Using the rewriting rules for $\Delta_{22}$, which were given below its definition, one easily checks 
that $gf=2\mbox{id}$. Therefore $\hatdumbell_{31}$ is a direct summand of $\hatsquareup_{1122}$ too. 

Another easy calculation shows that $hg=0$ and a slightly harder one that $fj=0$. 

Remains to show that $(g,j)$ is surjective. It suffices to show that $1\stens 1\stens 1, 1\stens x_3\stens 1$ 
and $1\stens x_3^2\stens 1$ are in its image. We have 
\begin{align*}
1\stens 1\stens 1 &= g(1\stens 1)\\
1\stens x_3\stens 1 &= j(1,0)+g(x_4\stens 1) \\
1\stens x_3^2\stens 1 &= j(0,1)+j(x_4,0)+g(x_4^2\stens 1) 
\rlap{\hspace*{20.4ex}\qed}
\end{align*}

\begin{lem}
\label{lem:2113-square}
Let $\hatsquareup_{2113}$ be the square in A7. We have 
$$\hatsquareup_{2113}\cong \hatdumbell_{22}^{13}\oplus \hatHGraph_{22}^{13}\{2\}.
$$
\end{lem}
\begin{proof}
We first define the bimodule map 
$$\phi_1\colon \hatdumbell_{22}^{13}=R_{13}\otimes_4 R_{22}\to R_{121}
\otimes_{31} R_{211}\otimes_{22}R_{22}=\hatsquareup_{2113}.$$ We use the two 
intermediate bimodules $\widehat{\Gamma_1}=R_{121}\otimes_{13}R_{13}
\otimes_{4}R_{211}\otimes_{22}R_{22}$ and $\widehat{\Gamma_2}=R_{121}\otimes_{31}R_{31}
\otimes_{4}R_{211}\otimes_{22}R_{22}$ (see Figure~\ref{fig:s2113-g1g2}).
\begin{figure}[ht!]
\labellist
\tiny\hair 2pt
\pinlabel $2$ at -8  13
\pinlabel $1$ at  -5 235
\pinlabel $2$ at 125  13
\pinlabel $3$ at 121 235
\pinlabel $1$ at 100  90
\pinlabel $1$ at  75  34
\pinlabel $1$ at 100 155
\pinlabel $2$ at  78 216
\pinlabel $2$ at  50  79 
\pinlabel $3$ at  50 169
\pinlabel $4$ at  20 122
\pinlabel $2$ at 274  13
\pinlabel $1$ at 277 235
\pinlabel $2$ at 406  13
\pinlabel $3$ at 403 235
\pinlabel $4$ at 325 122
\pinlabel $1$ at 388  80
\pinlabel $1$ at 388 175
\pinlabel $3$ at 293 80
\pinlabel $3$ at 293 165
\pinlabel $1$ at 340  23
\pinlabel $2$ at 340 224
\small\hair 2pt
\pinlabel $\Gamma_1$ at  64 -50
\pinlabel $\Gamma_2$ at 345 -50
\endlabellist
\centering
\figs{0.25}{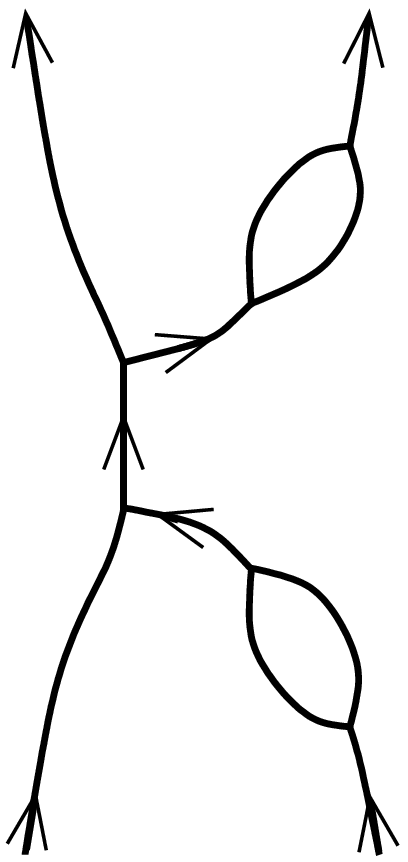}\hspace{8.4ex}\figs{0.25}{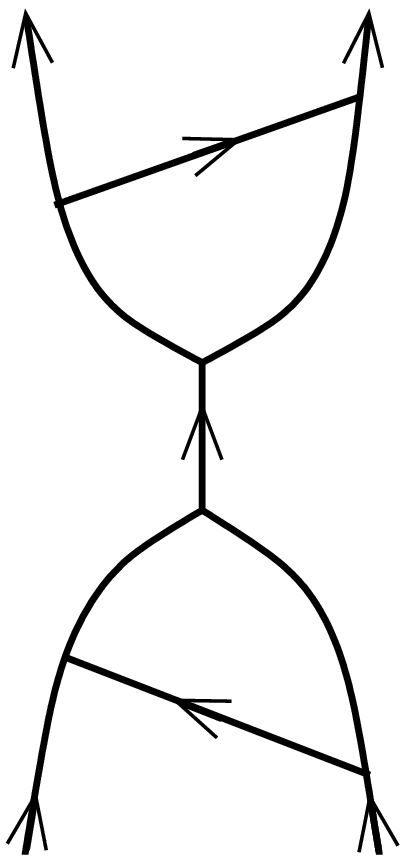}
\\[2.0ex]
\caption{Intermediate webs $\Gamma_1$ and $\Gamma_2$}
\label{fig:s2113-g1g2}
\end{figure}
Then $\phi_1$ is the 
composite 
$$\hatdumbell_{22}^{13}{\xrightarrow{\ \ \phi_1^1\ \ }}\widehat{\Gamma_1}
{\xrightarrow{\ \ \phi_1^2\ \ }}\widehat{\Gamma_2}{\xrightarrow{\ \ \phi_1^3\ \ }}  
\hatsquareup_{2113}
$$
with
\begin{align*}
\phi_1^1(a\stens b) &= 1\stens a\stens 1\stens b\\ 
\phi_1^2(a\stens b\stens c\stens d) &= ab\stens 1\stens c\stens d\\
\phi_1^3(a\stens b\stens c\stens d) &= a\stens bc\stens d.
\end{align*}
Conversely we define $\psi_1$ as the composite 
$$\hatsquareup_{2113}{\xrightarrow{\ \ \psi_1^1\ \ }}\widehat{\Gamma_2}
{\xrightarrow{\ \ \psi_1^2\ \ }}\widehat{\Gamma_1}{\xrightarrow{\ \ \psi_1^3\ \ }}  
\hatdumbell_{22}^{13}
$$
with
\begin{align*}
\psi_1^1(a\stens b\stens c) &= a\stens \Delta_{31}(b)\stens c\\ 
\psi_1^2(a\stens b\stens c\stens d) &= ab\stens 1\stens c\stens d
\end{align*}
and with $\psi_1^3$ defined by applying the maps 
$R_{121}\to R_{13}\{4\}$ and $R_{211}\to R_{22}\{2\}$ which are the projections 
onto the last direct summands in the decompositions 
$R_{121}\cong R_{13}\oplus x_4R_{13}\oplus x_4^2R_{13}$ and 
$R_{211}\cong R_{22}\oplus x_3R_{22}$.  
A short calculation shows that $\psi_1\phi_1=-\mbox{id}$, which proves that 
$\hatdumbell_{22}^{13}$ is a direct summand of $\hatsquareup_{2113}$.

Let us now define $\phi_2\colon \hatHGraph_{22}^{13}\to \hatsquareup_{2113}$. 
Note that $\hatHGraph_{22}^{13}=R_{112}\otimes_{22} R_{22}$. 
Again we use certain intermediate bimodules: $\widehat{\Lambda}_1=R_{1111}
\otimes_{112}R_{112}\otimes_{22} R_{22}$, $\widehat{\Lambda}_{2}=R_{1111}
\otimes_{211}R_{211}\otimes_{22} R_{22}$, $\widehat{\Lambda}_3=R_{1111}
\otimes_{211}R_{211}\otimes_{31}R_{211}\otimes_{22} R_{22}$ and 
$\widehat{\Lambda}_4=R_{1111}\otimes_{121}R_{121}
\otimes_{31}R_{211}\otimes_{22} 
R_{22}$ (see Figure~\ref{fig:s2113-L}).

\begin{figure}[ht!]
\labellist
\tiny\hair 2pt
\pinlabel $2$ at  -8  12
\pinlabel $1$ at  -6 235
\pinlabel $2$ at 146  12
\pinlabel $3$ at 143 235
\pinlabel $1$ at  92  88
\pinlabel $1$ at 145  88
\pinlabel $2$ at 130 138
\pinlabel $1$ at  72 153
\pinlabel $2$ at 307  12
\pinlabel $1$ at 309 235
\pinlabel $2$ at 463  12
\pinlabel $3$ at 460 235
\pinlabel $1$ at 372 146
\pinlabel $1$ at 399  77
\pinlabel $1$ at 459 120
\pinlabel $2$ at 405 183
\pinlabel $2$ at 630  12
\pinlabel $1$ at 625 235
\pinlabel $2$ at 772  12
\pinlabel $3$ at 775 235
\pinlabel $1$ at 695 115
\pinlabel $2$ at 644 102
\pinlabel $3$ at 667  66
\pinlabel $1$ at 670 160
\pinlabel $1$ at 714  29
\pinlabel $1$ at 772 122
\pinlabel $2$ at 720 202
\pinlabel $2$ at  944  12
\pinlabel $1$ at  938 235
\pinlabel $2$ at 1088  12
\pinlabel $3$ at 1090 235
\pinlabel $1$ at 1010 181
\pinlabel $1$ at 1034 125
\pinlabel $3$ at  975  64
\pinlabel $2$ at  977 133
\pinlabel $1$ at 1029  30
\pinlabel $1$ at 1088 120
\pinlabel $2$ at 1045 205
\small\hair 2pt
\pinlabel $\Lambda_1$ at   80 -50
\pinlabel $\Lambda_2$ at  395 -50
\pinlabel $\Lambda_3$ at  710 -50
\pinlabel $\Lambda_4$ at 1024 -50
\endlabellist
\centering
\figs{0.25}{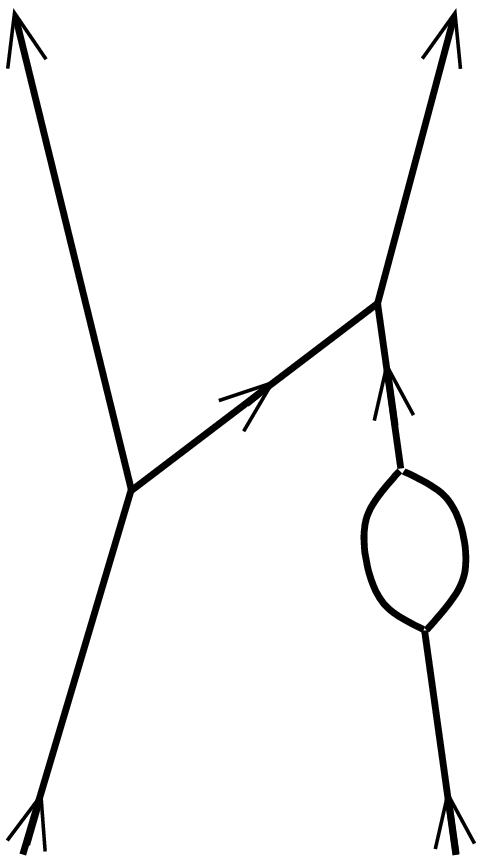}\hspace{8.4ex}
\figs{0.25}{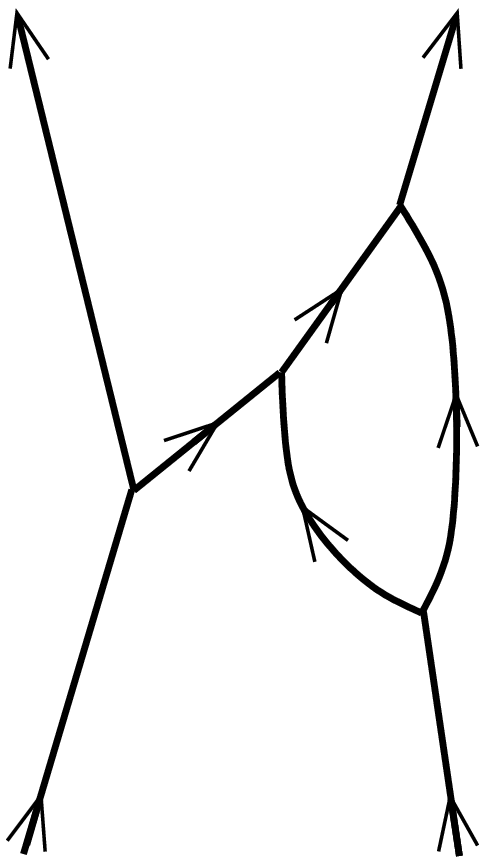}\hspace{8.4ex}
\figs{0.25}{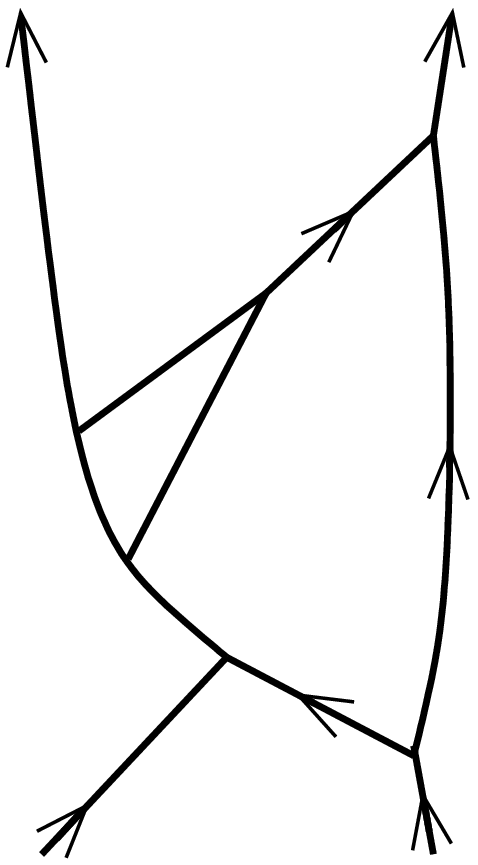}\hspace{8.4ex}
\figs{0.25}{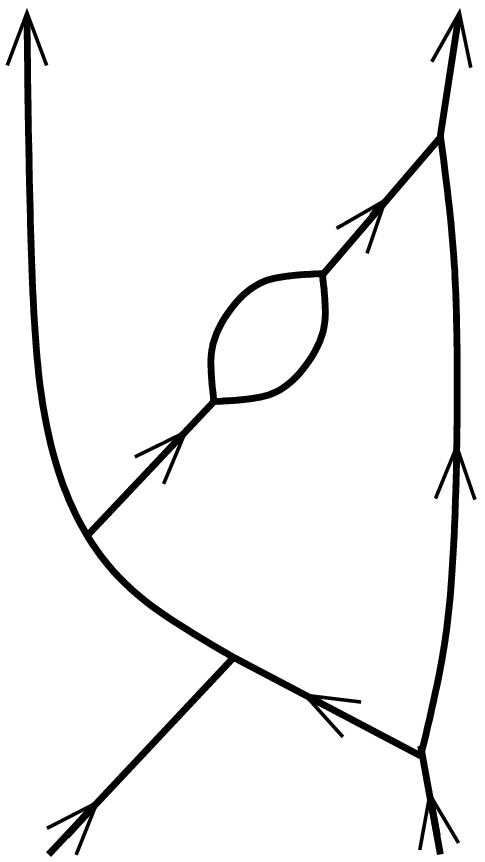}
\\[2.0ex]
\caption{Intermediate webs $\Lambda_1$, \ldots , $\Lambda_4$}
\label{fig:s2113-L}
\end{figure}

We define $\phi_2$ as the composite 
$$\hatHGraph_{22}^{13}
{\xrightarrow{\ \ \phi_2^1\ \ }}\widehat{\Lambda}_1
{\xrightarrow{\ \ \phi_2^2\ \ }}\widehat{\Lambda}_2
{\xrightarrow{\ \ \phi_2^3\ \ }}\widehat{\Lambda}_3
{\xrightarrow{\ \ \phi_2^4\ \ }}\widehat{\Lambda}_4
{\xrightarrow{\ \ \phi_2^5\ \ }}
\hatsquareup_{2113}
$$
with 
\begin{align*}
\phi_1^1(a\stens b) &= 1\stens a\stens b\\
\phi_1^2(a\stens b\stens c) &= ab\stens 1\stens c\\
\phi_1^3(a\stens b\stens c) &= a\stens \Delta_{21}^{x_1x_2x_3}(b)\stens c\\
\phi_1^4(a\stens b\stens c\stens d) &= ab\stens 1\stens c\stens d
\end{align*}
and with $\phi_1^5$ defined by the map $R_{1111}\to R_{121}\{2\}$ which is 
the projection onto the second direct summand in the decomposition 
$R_{1111}\cong R_{121}\oplus x_2R_{121}$.

Conversely, we define $\psi_2$ as the composite
$$
\hatsquareup_{2113}
{\xrightarrow{\ \ \psi_2^1\ \ }}\widehat{\Lambda}_4
{\xrightarrow{\ \ \psi_2^2\ \ }}\widehat{\Lambda}_3
{\xrightarrow{\ \ \psi_2^3\ \ }}\widehat{\Lambda}_2
{\xrightarrow{\ \ \psi_2^4\ \ }}\widehat{\Lambda}_1
{\xrightarrow{\ \ \psi_2^5\ \ }}
\hatHGraph_{22}^{13}$$
with
\begin{align*}
\psi_1^1(a\stens b\stens c) &= 1\stens a\stens b\stens c\\
\psi_1^2(a\stens b\stens c\stens d) &= ab\stens 1\stens c\stens d\\
\psi_1^3(a\stens b\stens c\stens d) &= a\stens bc\stens d\\
\psi_1^4(a\stens b\stens c) &= ab\stens 1\stens c
\end{align*}
and with $\psi_1^5$ defined by the map $R_{1111}\to R_{112}\{2\}$ which is 
the projection onto the second direct summand in the decomposition 
$R_{1111}\cong R_{112}\oplus x_3R_{121}$.
A simple calculation shows that $\psi_2\phi_2=\id$.

One can also easily check that $\psi_2\phi_1=0$ and $\psi_1\phi_2=0$. This shows that 
$(\phi_1,\phi_2)$ is injective. To show that 
$(\phi_1,\phi_2)$ is surjective note that both maps are left $R_{13}$-module maps 
and that the source and the target are both free $R_{13}$-modules of rank 9 with 
the same gradings. 
\end{proof}

\section{The link homology}
\label{sec:HOMFLY-PT}

Let us define the coloured HOMFLY-PT homology for links with components labelled 1 and 2. 
We use a similar setup to the one in~\cite{Kh}. To each braid diagram we associate a complex of bimodules (defined below)
obtained from the 
categorified MOY calculus. This complex is 
invariant up to homotopy under the braidlike Reidemeister II and III moves. Then we take the Hochschild 
homology of each bimodule in the complex, which corresponds to the categorification of the Markov trace. 
This induces a complex of vector spaces whose homology is the 
one we are 
looking for. The latter is still invariant under the second and third Reidemeister moves, because 
the Hochschild homology is a covariant functor, and also under the 
Markov moves, as we will show. 
Therefore we obtain a triply graded link homology. By taking the graded dimensions of 
the homology groups we get a triply graded link polynomial. \\

To define the complex of bimodules associated to a braid it suffices to define it for a positive and for 
a negative crossing only. For an arbitrary braid one then tensors these complexes over all crossings. 
To each crossing with both strands labelled by 2 we associate a complex with three terms. For a positive, resp. negative, 
crossing between strands labelled 2 the terms in the complex are 
$\hattwoarcs_{22}\{6\}\to \hatsquareup_{3111}\{2\}\to \hatdumbell_{22}$, resp. 
$\hatdumbell_{22}\{-8\}\to \hatsquareup_{3111}\{-8\}\to \hattwoarcs_{22}\{-6\}$ 
(see Figures~\ref{fig:poscross}-\ref{fig:negcross}).  
\begin{figure}[ht!]
\labellist
\pinlabel $=$ at 235 112
\pinlabel $q^{2}$ at 720 115
\pinlabel $q^6$ at 300 115
\tiny\hair 2pt
\pinlabel $2$ at -12  19
\pinlabel $2$ at 175  19
\pinlabel $2$ at 313  19
\pinlabel $2$ at 501  19
\pinlabel $2$ at 757  19
\pinlabel $2$ at 758 206
\pinlabel $2$ at 947  19
\pinlabel $2$ at 945 206
\pinlabel $3$ at 780 110
\pinlabel $1$ at 927 110
\pinlabel $1$ at 852  38
\pinlabel $1$ at 852 177
\pinlabel $2$ at 1201  16
\pinlabel $2$ at 1203 206
\pinlabel $2$ at 1383  16
\pinlabel $2$ at 1380 206
\pinlabel $4$ at 1310 110
\pinlabel $d_1^+$ at 580 140
\pinlabel $d_2^+$ at 1050 140
\endlabellist
\centering
\figs{0.23}{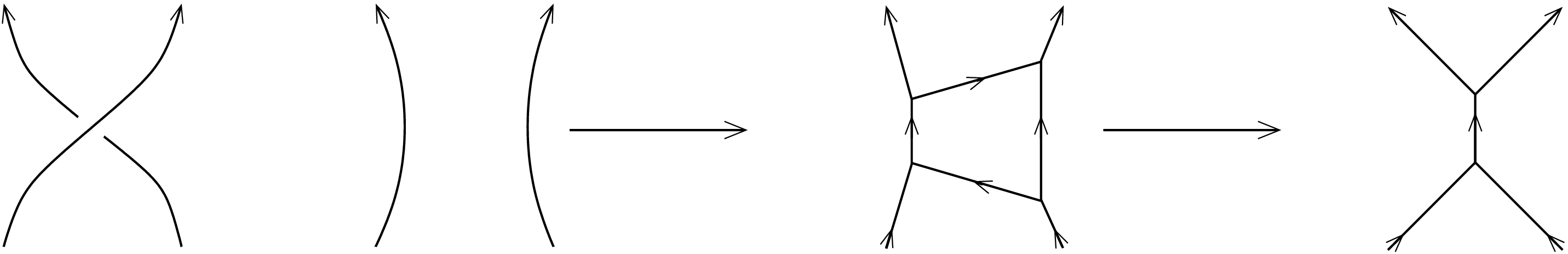}
\caption{The complex of a positive crossing }
\label{fig:poscross}
\end{figure}

\begin{figure}[ht!]
\labellist
\pinlabel $=$ at 235 112
\pinlabel $q^{-6}$ at 1180 115
\pinlabel $q^{-8}$ at 720 115
\pinlabel $q^{-8}$ at 310 115
\tiny\hair 2pt
\pinlabel $2$ at -12 17
\pinlabel $2$ at 175 17
\pinlabel $2$ at 316 17
\pinlabel $2$ at 499 17
\pinlabel $4$ at 426 110
\pinlabel $2$ at 316 204
\pinlabel $2$ at 499 204
\pinlabel $2$ at 756 17
\pinlabel $2$ at 760 204
\pinlabel $2$ at 948 17
\pinlabel $2$ at 945 204
\pinlabel $3$ at 780 107
\pinlabel $1$ at 928 107
\pinlabel $1$ at 854  40
\pinlabel $1$ at 854 176
\pinlabel $2$ at 1204 17
\pinlabel $2$ at 1389 17
\pinlabel $d_1^-$ at 605 140
\pinlabel $d_2^-$ at 1060 140
\endlabellist
\centering
\figs{0.23}{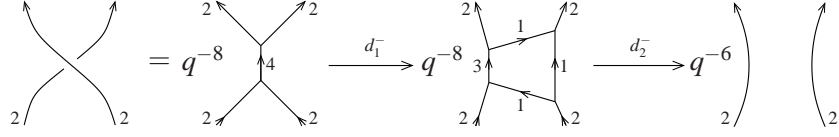}
\caption{The complex of a negative crossing }
\label{fig:negcross}
\end{figure}

In both cases, the cohomological degree is fixed by putting the bimodule $\hatdumbell_{22}$ in cohomological degree 0. 

To define the differentials we need the 
intermediate webs $\Phi, \Psi$ and $\Omega$ (see Figure~\ref{fig:inter-fpo}). 
Note that $\widehat{\Psi}\cong 
\widehat{\Omega}$. 

\begin{figure}[ht!]
\labellist
\tiny\hair 2pt
\pinlabel $2$ at  -8  13
\pinlabel $2$ at 128  13
\pinlabel $2$ at 125 230
\pinlabel $1$ at  68 120
\pinlabel $1$ at 152 120
\pinlabel $2$ at 313  13
\pinlabel $2$ at 316 235
\pinlabel $2$ at 446  13
\pinlabel $2$ at 444 235
\pinlabel $3$ at 341 150
\pinlabel $3$ at 341  90
\pinlabel $1$ at 382  22
\pinlabel $1$ at 382 226
\pinlabel $1$ at 429  70
\pinlabel $1$ at 429 180
\pinlabel $4$ at 396 122
\pinlabel $2$ at 606  13
\pinlabel $2$ at 609 235
\pinlabel $2$ at 740  13
\pinlabel $2$ at 737 235
\pinlabel $4$ at 635 120
\pinlabel $2$ at 664  78
\pinlabel $2$ at 664 166
\pinlabel $1$ at 694  28
\pinlabel $1$ at 694 216
\pinlabel $1$ at 723  80
\pinlabel $1$ at 723 170
\small\hair 2pt
\pinlabel $\Phi$ at   60 -50
\pinlabel $\Psi$ at  380 -50
\pinlabel $\Omega$ at  675 -50
\endlabellist
\centering
\figs{0.25}{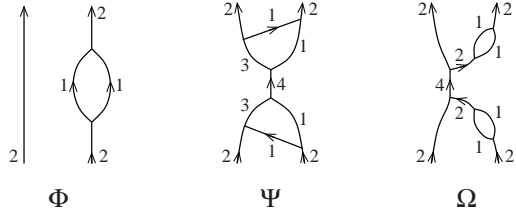}\hspace{8.4ex}
\figs{0.25}{interweb-p}\hspace{8.4ex}
\figs{0.25}{interweb-o}
\\[2.0ex]
\caption{Intermediate webs $\Phi$, $\Psi$ and $\Omega$}
\label{fig:inter-fpo}
\end{figure}
The differential 
$d_1^+$ is the composite of 
$$\hattwoarcs_{22}{\xrightarrow{\ \ d_{11}^+\ \ }}\widehat{\Phi}{\xrightarrow{\ \ d_{12}^+\ \ }}
\hatsquareup_{3111}\{-4\},$$
where $d_{11}^+$ is defined by the inclusion $R_{22}\hookrightarrow R_{211}$ to create the digon and 
$$d_{12}^+(a\stens b)=\Delta_{21}^{x_1,x_2,x_3}(a)\stens b.$$
\noindent The differential $d_2^+$ is the composite of 
$$\hatsquareup_{3111}\{-4\}{\xrightarrow{\ \ d_{21}^+\ \ }}\widehat{\Psi}\{-10\}\cong 
\widehat{\Omega}\{-10\}{\xrightarrow{\ \ d_{22}^+\ \ }}\hatdumbell_{22}\{-6\},$$
where $d_{21}^+$ is defined by 
$$d_{21}^+(a\stens b\stens c)=a\cdot\Delta_{31}(b)\stens c$$
and $d_{22}^+$ by applying twice the map $R_{211}\to R_{22}\{2\}$ given by the 
projection onto the second direct summand in the decomposition $R_{211}=R_{22}\oplus x_3R_{22}$.
Direct computation shows that $d_2^+d_1^+=0$. 

The first differential in the complex associated to a negative crossing is the composite of 
$$\hatdumbell_{22}{\xrightarrow{\ \ d_{11}^-\ \ }}\widehat{\Omega}\cong\widehat{\Psi}
{\xrightarrow{\ \ d_{12}^-\ \ }}\hatsquareup_{3111},$$
where $d_{11}^-$ is defined by applying twice the inclusion $R_{22}\hookrightarrow R_{211}$ to create the 
digons and 
$$d_{12}^-(a\stens b\stens c\stens d)=a\stens bc\stens d.$$
\noindent The differential $d_2^-$ is the composite of 
$$\hatsquareup_{3111}{\xrightarrow{\ \ d_{21}^-\ \ }}\widehat{\Phi}
{\xrightarrow{\ \ d_{22}^-\ \ }}\hattwoarcs_{22}\{2\},$$
where $d_{21}^-$ is defined by 
$$d_{21}^-(a\stens b\stens c)=ab\stens c$$
and $d_{22}^-$ by applying the same projection $R_{211}\to R_{22}\{2\}$ as above. 
Again, direct computation shows that $d_2^-d_1^-=0$. Note 
that this calculation is much easier than the one which 
shows $d_2^+d_1^+=0$. The latter is a direct consequence of 
the former by the observation that all maps involved in the 
definition of $d_2^{\pm} d_1^{\pm}$ are units and counits of 
the two biadjoint functors given by induction and restriction. 
As a matter of fact whenever we 
use a unit in the definition of one we use the corresponding 
counit in the definition of the other. Therefore the fact 
that $d_2^-d_1^-=0$ implies $d_2^+d_1^+=0$.\footnote{We thank Mikhail 
Khovanov for this observation.}

Next we define a complex of bimodules associated to a crossing of a 
strand labelled 1 and a strand labelled 2. To a positive crossing we associate the 
complex 
\begin{equation*}
\labellist
\pinlabel $=$ at 220 112
\pinlabel $q^{2}$ at 285 115
\tiny\hair 2pt
\pinlabel $2$ at -12  16
\pinlabel $1$ at 175  16
\pinlabel $2$ at 313  16
\pinlabel $1$ at 498  16
\pinlabel $1$ at 315 206
\pinlabel $2$ at 498 206
\pinlabel $1$ at 406 125
\pinlabel $2$ at 752  16
\pinlabel $1$ at 754 206
\pinlabel $1$ at 935  16
\pinlabel $2$ at 935 206
\pinlabel $3$ at 863 110
\pinlabel $d^+$ at 605 140
\endlabellist
\centering
\figs{0.23}{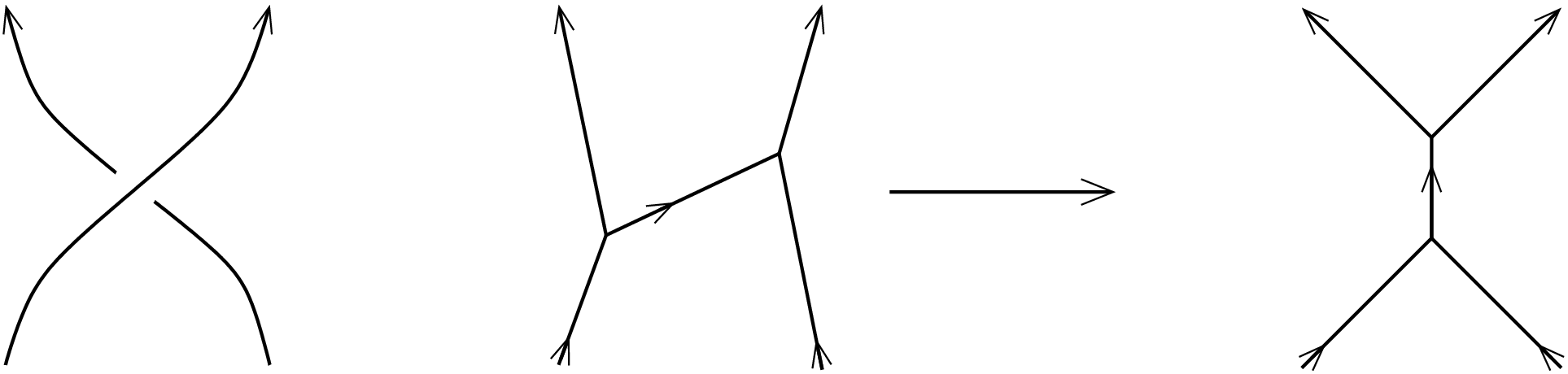}
\vspace{2ex}
\end{equation*}
and to a negative one 
\begin{equation*}
\labellist
\pinlabel $=$ at 220 112
\pinlabel $q^{-4}$ at 293 115
\pinlabel $q^{-4}$ at 720 115
\tiny\hair 2pt
\pinlabel $2$ at -12  16
\pinlabel $1$ at 175  16
\pinlabel $2$ at 313  16
\pinlabel $1$ at 498  16
\pinlabel $1$ at 315 206
\pinlabel $2$ at 498 206
\pinlabel $3$ at 427 110
\pinlabel $2$ at 752  16
\pinlabel $1$ at 754 206
\pinlabel $1$ at 935  16
\pinlabel $2$ at 935 206
\pinlabel $1$ at 840 125
\pinlabel $d^-$ at 610 140
\small\hair 2pt
\endlabellist
\centering
\figs{0.23}{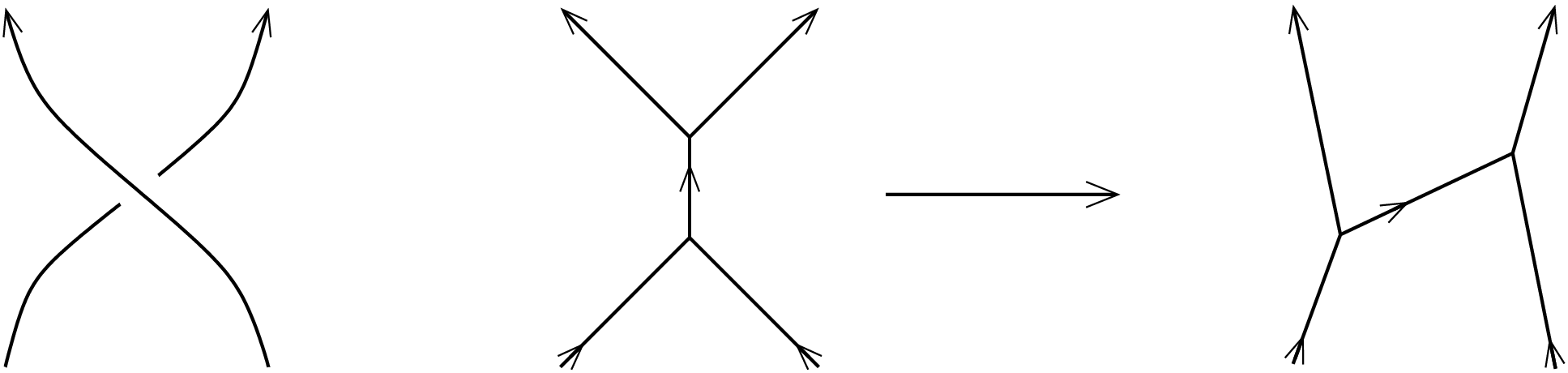}
\vspace{2ex}
\end{equation*}

Again, in both cases, we put the bimodule $\hatdumbell_{21}^{12}$ in the cohomological degree 0.

Note that $\hatHGraph_{21}^{12}=R_{111}\otimes_{21}R_{21}$ and $\hatdumbell_{21}^{12}=
R_{12}\otimes_3 R_{21}$.
We use the intermediate bimodules $\widehat{\Lambda}_1=
R_{111}\otimes_{21}R_{21}\otimes_{3} R_{21}\cong 
R_{111}\otimes_{12}R_{12}\otimes_{3} R_{21}=
\widehat{\Lambda}_2$ (see Figure~\ref{fig:inter-g12}). 
\begin{figure}[ht!]
\labellist
\tiny\hair 2pt
\pinlabel $2$ at  -8  12
\pinlabel $1$ at 122  12
\pinlabel $1$ at  -8 235
\pinlabel $2$ at 122 235
\pinlabel $3$ at  74 116
\pinlabel $2$ at  15 166
\pinlabel $1$ at 110 184
\pinlabel $1$ at  60 226
\pinlabel $2$ at 286  12
\pinlabel $1$ at 286 235
\pinlabel $1$ at 414  12
\pinlabel $2$ at 418 235
\pinlabel $3$ at 336  83
\pinlabel $1$ at 372 176
\pinlabel $1$ at 424 172
\pinlabel $2$ at 392 115
\small\hair 2pt
\pinlabel $\Lambda_1$ at   65 -50
\pinlabel $\Lambda_2$ at  360 -50
\endlabellist
\centering
\figs{0.25}{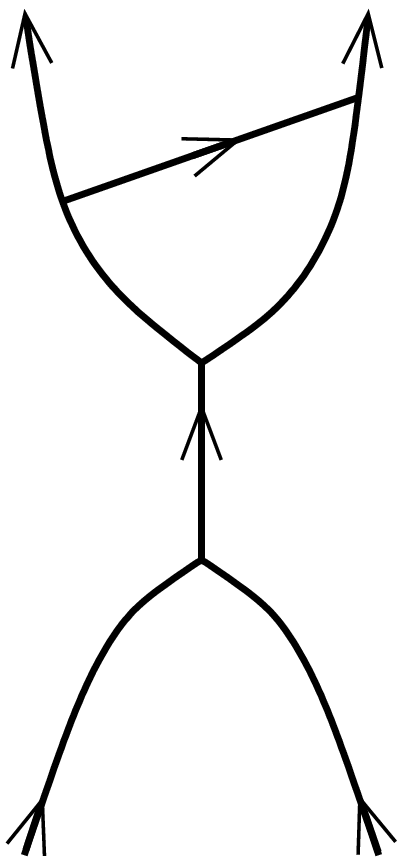}\hspace{8.4ex}
\figs{0.25}{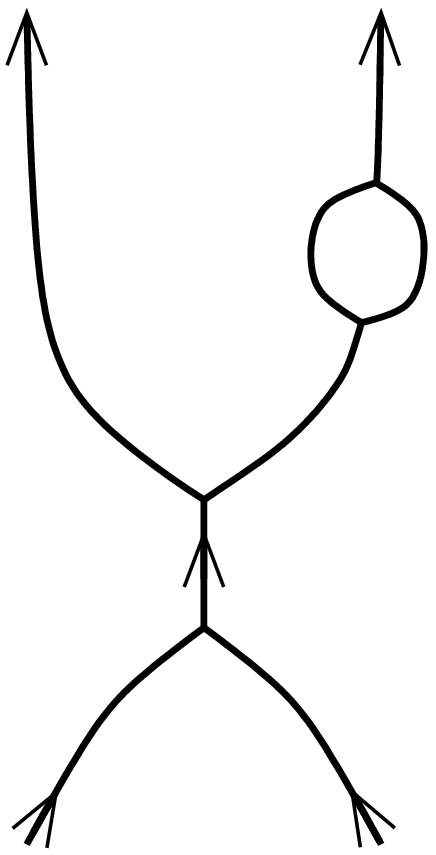}\hspace{8.4ex}
\\[1.5ex]
\caption{Intermediate webs $\Lambda_1$ and $\Lambda_2$}
\label{fig:inter-g12}
\end{figure}

Then $d^+$ is the composite of
$$\hatHGraph_{21}^{12}{\xrightarrow{\ \ d^+_1\ \ }}
\widehat{\Lambda}_1\{-4\}
{\xrightarrow{\ \ \cong\ \ }}
\widehat{\Lambda}_2\{-4\}
{\xrightarrow{\ \ d^+_2\ \ }}
\hatdumbell_{21}^{12}\{-2\}$$
with 
\begin{align*}
d^+_1(a\stens b) &= a\stens \Delta_{21}(b)\\
a\stens b\stens c\ &{\xrightarrow{\cong}}\ ab\stens 1\stens c
\end{align*}
and with $d^+_2$ defined by the map $R_{111}\to R_{12}\{2\}$ corresponding 
to the projection onto the second direct summand in the decomposition 
$R_{111}\cong R_{12}\oplus x_2 R_{12}$. It is easy to compute the image of $d^+$ on 
generators of the $R_{21}-R_{12}$ bimodule $R_{21}\otimes R_{111}$: 
$$d^+(1\stens 1)=x_1\stens 1-1\stens x_3\quad\text{and}\quad d^+(x_2\stens 1)=
1\stens x_1x_2-x_2x_3\stens 1.$$ 

Similarly we define $d^-$ as the composite 
$$\hatdumbell_{21}^{12}{\xrightarrow{\ \ d^-_1\ \ }}\widehat
{\Lambda}_2{\xrightarrow
{\ \ \cong\ \ }}\widehat{\Lambda}_1{\xrightarrow{\ \ d^-_2\ \ }}\hatHGraph_{21}^{12}$$
with 
\begin{align*}
d^-_1(a\stens b) &= 1\stens a\stens b\\
a\stens b\stens c &{\xrightarrow{\cong}} ab\stens 1\stens c\\
d^-_2(a\stens b\stens c) &= a\stens bc.
\end{align*}
Note that this yields $d^-(a\stens b)=a\stens b$.

We get similar complexes for the crossings with 1 and 2 
swapped. The pictures can be obtained from the ones above 
by rotation around the $y$-axis and the shifts are the same. 

For a crossing with both strands labelled 1 we use the same complex of bimodules as 
Khovanov in~\cite{Kh}:
\begin{gather*}
\begin{split}
\raisebox{0pt}{
\labellist
\pinlabel $=$ at 225 112
\pinlabel $q^{2}$ at 325 115
\tiny\hair 2pt
\pinlabel $1$ at -12  16
\pinlabel $1$ at 175  16
\pinlabel $1$ at 376  16
\pinlabel $1$ at 560  16
\pinlabel $1$ at 752  16
\pinlabel $1$ at 754 206
\pinlabel $1$ at 935  16
\pinlabel $1$ at 935 206
\pinlabel $2$ at 864 110
\endlabellist
\centering
\figs{0.23}{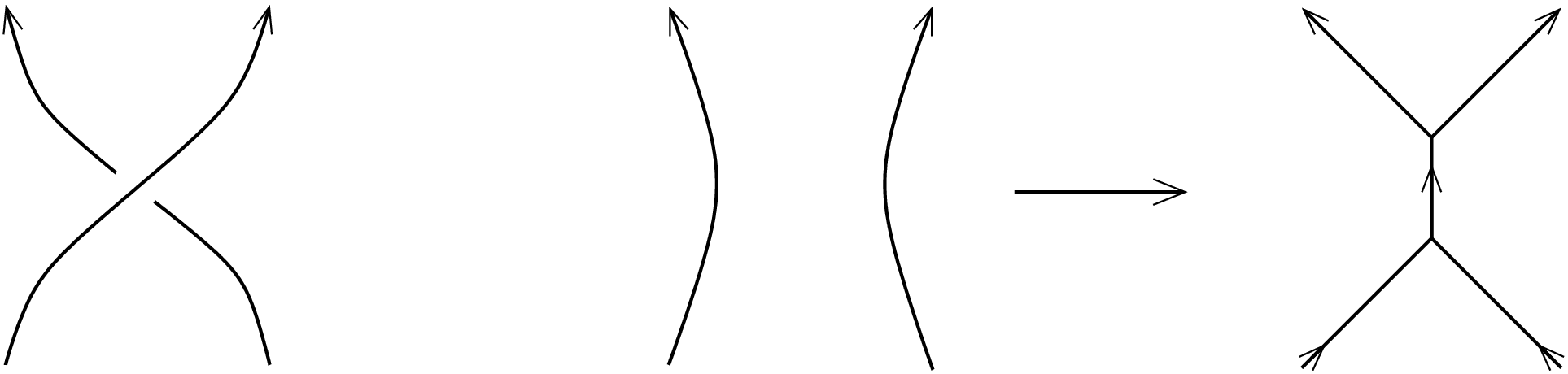}}
\\[1.5ex]
\raisebox{0pt}{
\labellist
\pinlabel $=$ at 225 112
\pinlabel $q^{-2}$ at 315 115
\pinlabel $q^{-2}$ at 710 115
\tiny\hair 2pt
\pinlabel $1$ at -12  16
\pinlabel $1$ at 175  16
\pinlabel $1$ at 375  16
\pinlabel $1$ at 560  16
\pinlabel $1$ at 377 206
\pinlabel $1$ at 560 206
\pinlabel $2$ at 490 110
\pinlabel $1$ at 752  16
\pinlabel $1$ at 935  16
\endlabellist
\centering
\figs{0.23}{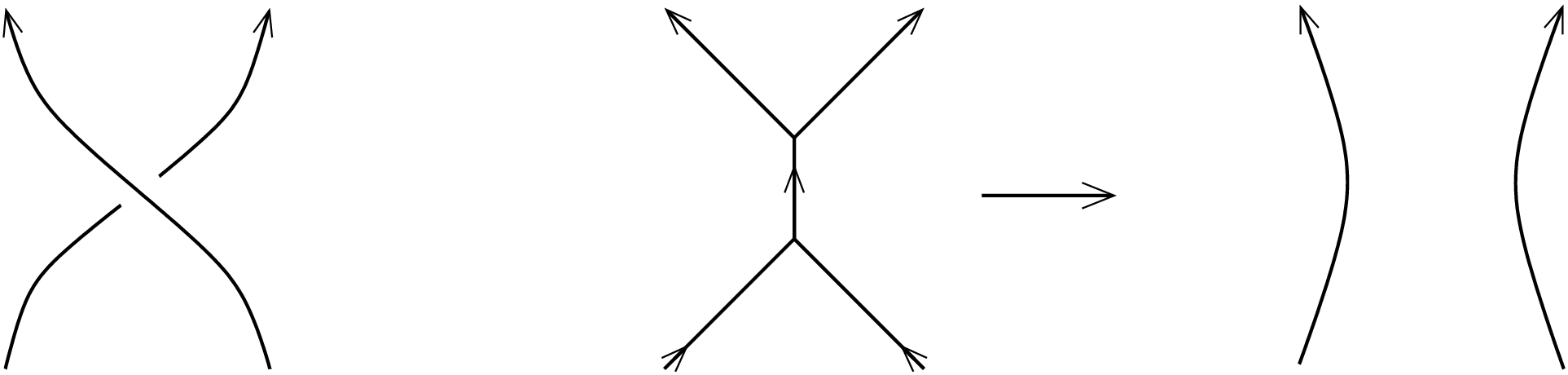}}
\medskip
\end{split}
\end{gather*}
with $\dumbell_{11}$ in cohomological degree $0$.\\

Let $HHH(D)$ denote the triply graded homology we defined above. 
Then to obtain the 1,2-coloured HOMFLY-PT 
homology $H_{1,2}(D)$ we have to apply some overall shifts. We define
\begin{defn}
\begin{eqnarray*}
H_{1,2}(D)=&
\mspace{-170mu}
HHH(D)<\frac{n^{1}_+-n^{1}_--s_1(D)+2n^{2}_+-2n^{2}_--2s_2(D)}{2}>\\
&\left\{\frac{-n^{1}_++n^{1}_-+s_1(D)-2n^{2}_++2n^{2}_-+2s_2(D)}{2},\frac{-n^{1}_++n^{1}_-+s_1(D)-2n^{2}_++2n^{2}_-+2s_2(D)}{2}\right\}.
\end{eqnarray*}
\end{defn}  
The definitions of $n^{i}_+,n^{i}_-$ and $s_i(D)$ were given in Section~\ref{sec:MOY},  
$<j>$ is an upward shift by $j$ in the homological degree and 
$\left\{k,l\right\}$ denotes an upward shift by $k$ in the Hochschild degree and 
by $l$ in the $q$-degree. \\

Finally in the following two sections we prove the following
\begin{thm}
For a given link $L$, $H_{1,2}(D)$ is independent of the chosen braid diagram $D$ which represents it. 
Hence, $H_{1,2}(L)$ is a link invariant.
\end{thm}

\section{Invariance under the R2 and R3 moves}
\label{sec:invariance}

The next thing to do is to prove invariance of the 1,2-coloured HOMFLY-PT homology 
under the second and third Reidemeister moves. If the strands involved are all labelled 
1 we already have invariance by Khovanov's~\cite{Kh} and Khovanov and Rozansky's \cite{KR2} results. 
If there are strands involved which are labelled 2, we will use a trick, inspired by the analogous 
trick in~\cite{MOY}, which reduces to the case with all link components labelled 1. The argument is slightly tricky, so let 
us explain the general idea first. 

Suppose we have a braid $B$ with $n$ strands, all labelled 2. Create a digon 
$\digonup_{11}$ on top of each strand. The triply graded polynomial of 
this braid with digons is $(1+q^2)^n$ times the triply graded polynomial of $B$. We will 
prove in Lemma~\ref{lem:slide22} that sliding the lower 
parts of the digons past the crossings 
does not change the triply graded polynomial. After sliding 
this way the lower parts 
of all digons past all crossings (see Figure~\ref{fig:digonsilde}), we obtain a 
braided diagram $B'$ which is 
the 2-cable of $B$ with the two top 
endpoints, and the two bottom endpoints respectively, of each cable zipped  
together. 
\begin{figure}[ht!]
\labellist
\tiny\hair 2pt
\pinlabel $2$   at  32 35
\pinlabel $2$   at 163 35
\pinlabel $2$ at 224  35
\pinlabel $2$ at 356  35
\pinlabel $2$ at 224 257
\pinlabel $2$ at 356 257
\pinlabel $1$ at 219 210
\pinlabel $1$ at 361 210
\pinlabel $1$ at 266 210
\pinlabel $1$ at 313 210
\pinlabel $2$ at 417  35
\pinlabel $2$ at 548  35
\pinlabel $2$ at 417 257
\pinlabel $2$ at 548 257
\pinlabel $1$ at 410 200
\pinlabel $1$ at 465 200
\pinlabel $1$ at 500 200
\pinlabel $1$ at 555 200
\endlabellist
\centering
\figs{0.26}{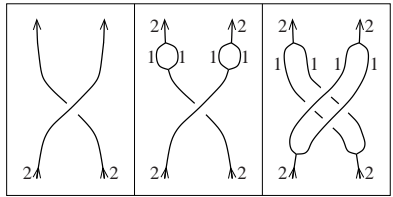}
\caption{Creating and sliding digons}
\label{fig:digonsilde}
\end{figure}
The complex of bimodules associated to $B'$ is the tensor 
product of the HOMFLY-PT complex associated to the 2-cable of $B$ with two complexes, 
one associated to the top endpoints of the cables and one to the bottom endpoints. 
Performing a Reidemeister II or III move on $B$ corresponds to performing a series of  
Reidemeister II and III moves on its 2-cable. By Khovanov and Rozansky's~\cite{KR2} 
results we know that the complex of bimodules associated to the 2-cable is invariant 
up to homotopy equivalence under Reidemeister II and III moves. Therefore the 
complex of bimodules associated to $B'$ is invariant under the Reidemeister II and III 
moves up to homotopy equivalence. Since the triply graded polynomial of $B'$ 
is $(1+q^2)^n$ times the triply graded polynomial of $B$, we see that the latter is 
invariant under the Reidemeister II and III moves as well. 
Note that this method does not give a specific homotopy 
equivalence between the complexes before and after a 
Reidemeister move. It only shows that the corresponding 
Poincar\'{e} polynomials are equal. This would be a problem 
if we wanted to show functoriality under link cobordisms 
of the whole construction. However, even the ordinary 
HOMFLY-PT-homology by Khovanov and Rozansky has not been proven 
to be functorial and probably is not.\\  

Before we prove the 
crucial lemma we have to prove an auxiliary result. 
In the following lemma the top and the bottom of the diagram 
are complexes. The reader can easily check the following auxiliary result. 

\begin{lem}
\label{lem:aux1} 
The diagram below gives a homotopy equivalence between the top and 
the bottom complex.
$$
\xymatrix@R=11.58mm{\\ \cC: \\ \\ \cC': }
\xymatrix@C=28.2mm@R=10mm{ 
&  
B'\oplus B 
\ar@{<-}@/^{1.6pc}/[ddd]|<(0.60)\hole^<(0.4){\bigl(\begin{smallmatrix}h\\0\end{smallmatrix}\bigr) }
\ar@/_{1.6pc}/[ddd]|<(0.594)\hole_<(0.39){(0,-g)}
\ar@<0.1pc>[dl]^{\bigl(\begin{smallmatrix}0 & 0\\ 0 & 1\end{smallmatrix}\bigr)}
\ar@<0.4pc>[dr]^{\bigl(\begin{smallmatrix}-1&u\\ 0& v\end{smallmatrix}\bigr)}
 &  \\
A\oplus B
\ar@{<->}[dd]_{(\id ,0)} \ar[dr]_{(f,g)} 
\ar@<0.4pc>[ur]^{\bigl(\begin{smallmatrix}r & s\\0& 1\end{smallmatrix}\bigr) } 
   &      & 
B'\oplus D
\ar@{<->}[dd]^{(0,\id)}
\ar@<0.1pc>[ul]^{\bigl(\begin{smallmatrix}-1 & 0\\ 0& 0\end{smallmatrix}\bigr)}
 \\             & 
C
\ar[ur]_{\bigl(\begin{smallmatrix}h\\j\end{smallmatrix}\bigr) }   
 &  \\
 A \ar[r]^f        & C\ar[r]^j \ar@{<->}[u]|\id & D
 }$$
\end{lem}

Using the lemma above we can now prove the following crucial lemma.

\begin{lem}
\label{lem:slide22}
We have the homotopy equivalence
\begin{equation}
\label{eq:slide22}
\labellist
\pinlabel $\cong$ at 264 110
\tiny\hair 2pt
\pinlabel $2$ at -8 14
\pinlabel $1$ at -7 177
\pinlabel $1$ at 70 204
\pinlabel $2$ at 184 14
\pinlabel $2$ at 339 14
\pinlabel $1$ at 341 187
\pinlabel $1$ at 434 214
\pinlabel $2$ at 542 14
\endlabellist
\raisebox{-16pt}{\figs{0.2}{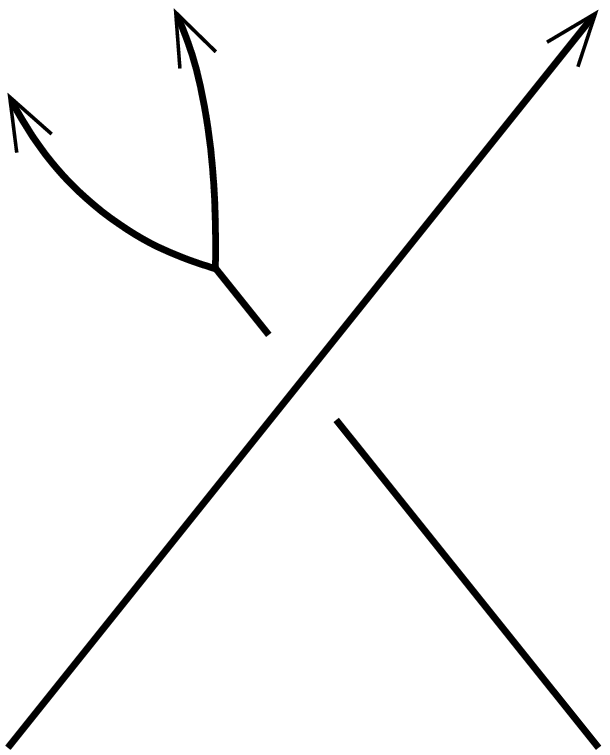}\qquad\quad \figs{0.2}{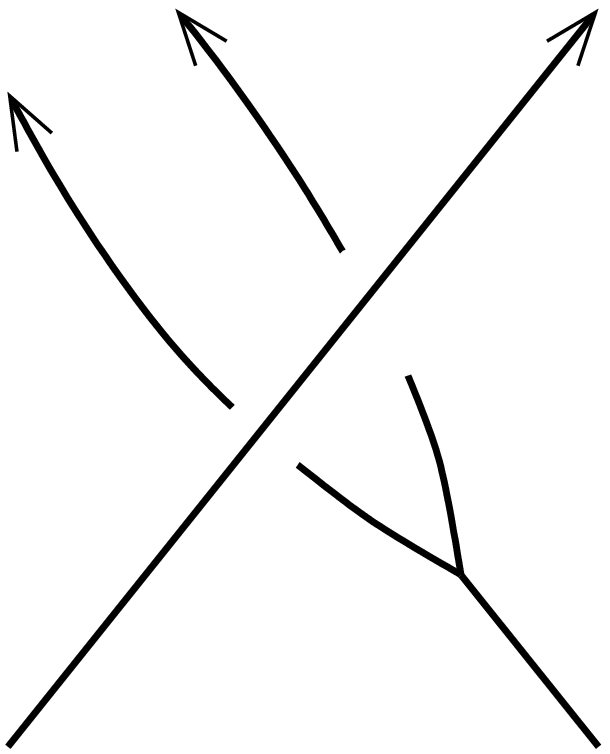}}
\end{equation}
\end{lem}  

\begin{proof} Note that the complex of bimodules $\cC$ associated the r.h.s. of 
Equation~\eqref{eq:slide22} is given by 
\begin{equation*}
\labellist
\pinlabel $\cC\colon$ at -180 340
\pinlabel $q^{2}$ at 460 120
\pinlabel $q^{2}$ at 460 560
\pinlabel $q^{4}$ at -50 340
\tiny\hair 2pt
\pinlabel $2$ at  -6 246
\pinlabel $1$ at  -6 438
\pinlabel $2$ at 188 246
\pinlabel $2$ at 187 438
\pinlabel $1$ at  64 438
\pinlabel $1$ at  60 338
\pinlabel $2$ at  85 368
\pinlabel $1$ at 122 397
\pinlabel $1$ at 100 289
\pinlabel $1$ at 159 330
\pinlabel $2$ at 490 473
\pinlabel $1$ at 492 665
\pinlabel $2$ at 684 473
\pinlabel $2$ at 684 665
\pinlabel $1$ at 560 665
\pinlabel $1$ at 562 562
\pinlabel $1$ at 656 542
\pinlabel $2$ at 610 588
\pinlabel $3$ at 660 595
\pinlabel $1$ at 600 517
\pinlabel $2$ at 490  15
\pinlabel $1$ at 492 205
\pinlabel $2$ at 688  15
\pinlabel $2$ at 687 205
\pinlabel $1$ at 564 205
\pinlabel $3$ at 548 102
\pinlabel $1$ at 660 102
\pinlabel $1$ at 622 176
\pinlabel $2$ at 574 148
\pinlabel $1$ at 596  44
\pinlabel $2$ at  987 244
\pinlabel $1$ at  989 435
\pinlabel $2$ at 1184 244
\pinlabel $2$ at 1184 435
\pinlabel $1$ at 1060 435
\pinlabel $3$ at 1048 315
\pinlabel $1$ at 1158 317
\pinlabel $3$ at 1166 372
\pinlabel $2$ at 1100 368
\pinlabel $1$ at 1100 265
\endlabellist
\centering
\figs{0.22}{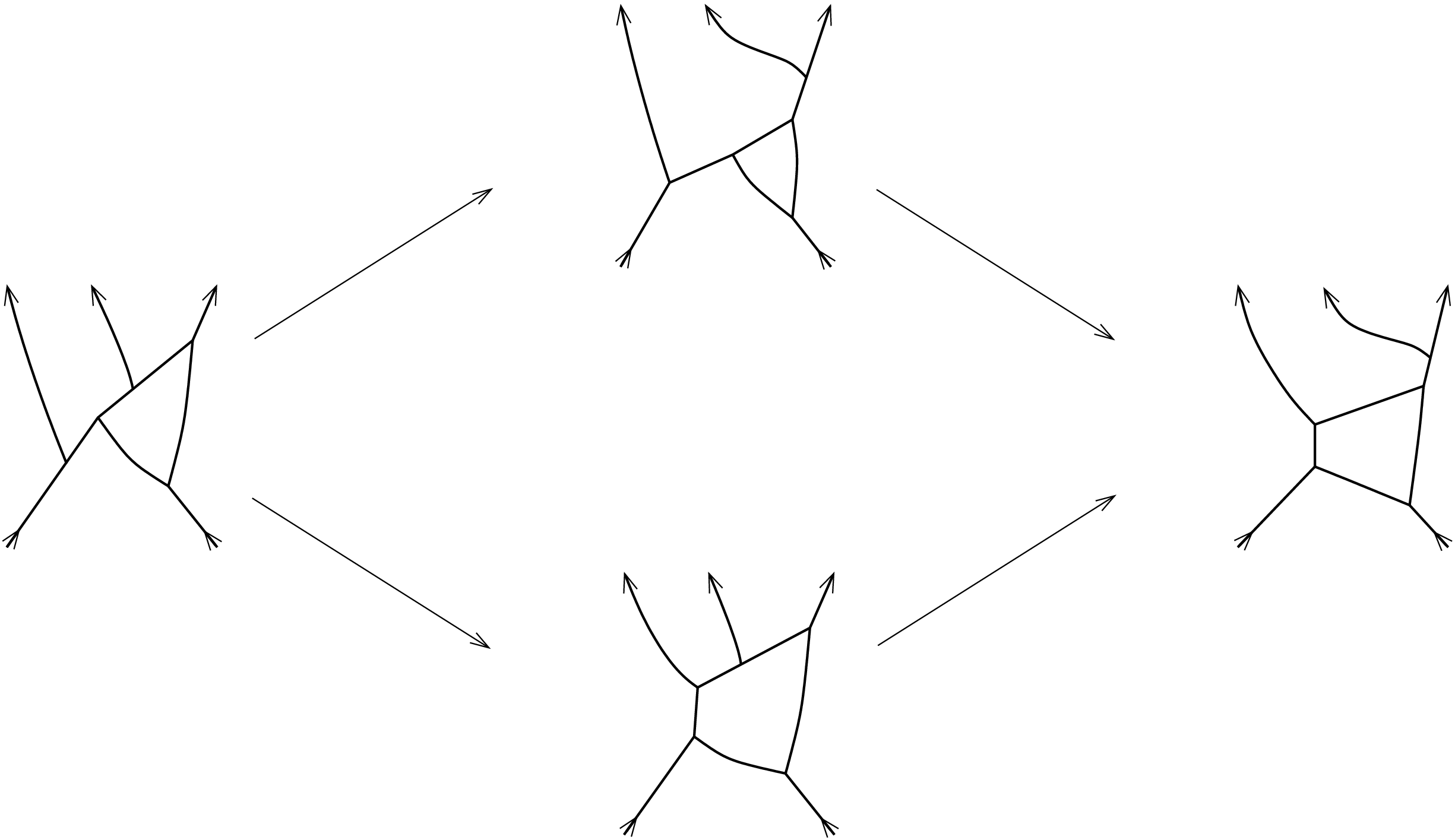}
\end{equation*}
The complex $\cC'$ associated to the l.h.s. is given by 
\begin{equation*}
\labellist
\pinlabel $\cC'\colon$ at -180 114
\pinlabel $q^6$ at 000 114
\pinlabel $q^{4}$ at 460 114
\tiny\hair 2pt
\pinlabel $2$ at  -8  14
\pinlabel $1$ at  -7 205
\pinlabel $2$ at 186  14
\pinlabel $1$ at  88 205
\pinlabel $2$ at 494  14
\pinlabel $1$ at 497 205
\pinlabel $2$ at 689  14
\pinlabel $2$ at 689 205
\pinlabel $1$ at 586 205
\pinlabel $3$ at 550  98
\pinlabel $1$ at 660  98
\pinlabel $1$ at 606 148
\pinlabel $1$ at 596  48
\pinlabel $2$ at 541 138
\pinlabel $2$ at  991  14
\pinlabel $1$ at  994 205
\pinlabel $2$ at 1184  14
\pinlabel $2$ at 1184 205
\pinlabel $1$ at 1084 205
\pinlabel $4$ at 1072  85
\pinlabel $2$ at 1080 148
\endlabellist
\centering
\figs{0.22}{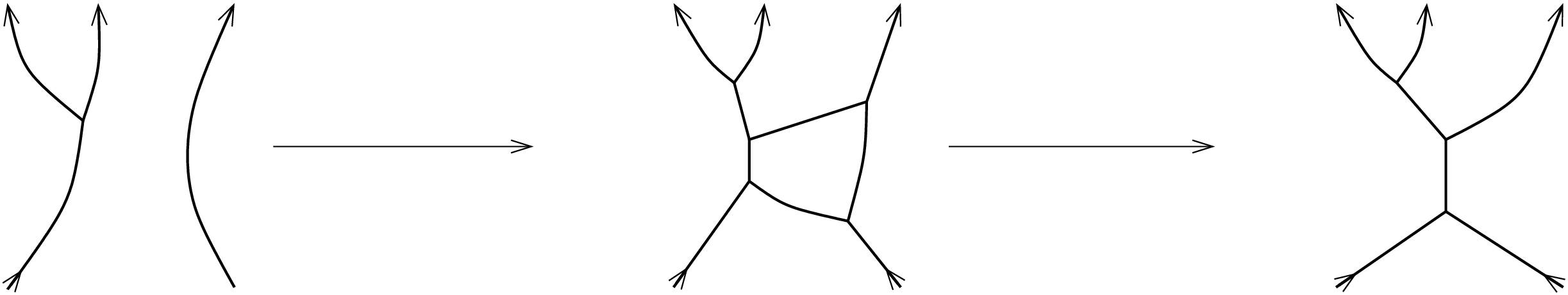}
\end{equation*}

\n By Lemma~\ref{lem:1112-square} we have 
\begin{equation*}
\labellist 
\pinlabel $\cong\  q^2$ at 260 110
\pinlabel $\oplus$ at 560 110
\pinlabel $A$ at 400 -25
\pinlabel $B$ at 720 -25
\tiny\hair 2pt
\pinlabel $2$ at  -8  14
\pinlabel $1$ at  -8 205
\pinlabel $2$ at 186  14
\pinlabel $2$ at 186 205
\pinlabel $1$ at  88 205
\pinlabel $1$ at  58 105
\pinlabel $2$ at  84 133
\pinlabel $1$ at 100  59
\pinlabel $1$ at 160 108
\pinlabel $1$ at 124 167
\pinlabel $2$ at 311  14
\pinlabel $1$ at 311 205
\pinlabel $2$ at 504  14
\pinlabel $1$ at 406 205
\pinlabel $2$ at 626  14
\pinlabel $1$ at 626 205
\pinlabel $2$ at 821  14
\pinlabel $2$ at 821 205
\pinlabel $1$ at 696 205
\pinlabel $1$ at 719  65
\pinlabel $3$ at 783 118
\endlabellist
\figs{0.22}{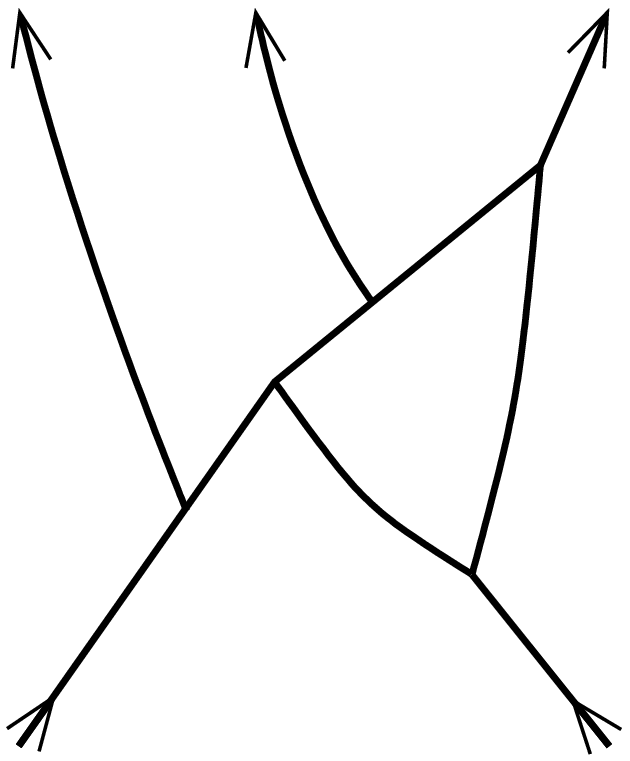}\mspace{50mu}
\figs{0.22}{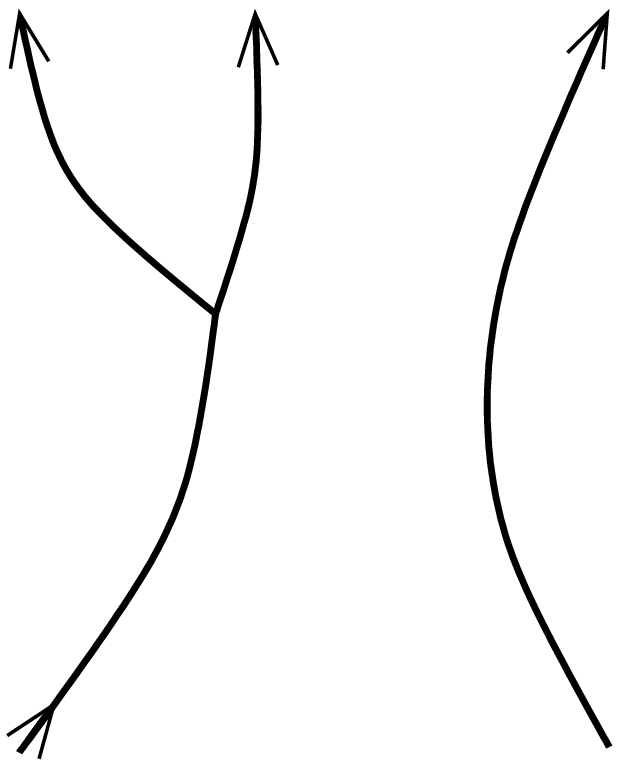}\mspace{50mu}
\figs{0.22}{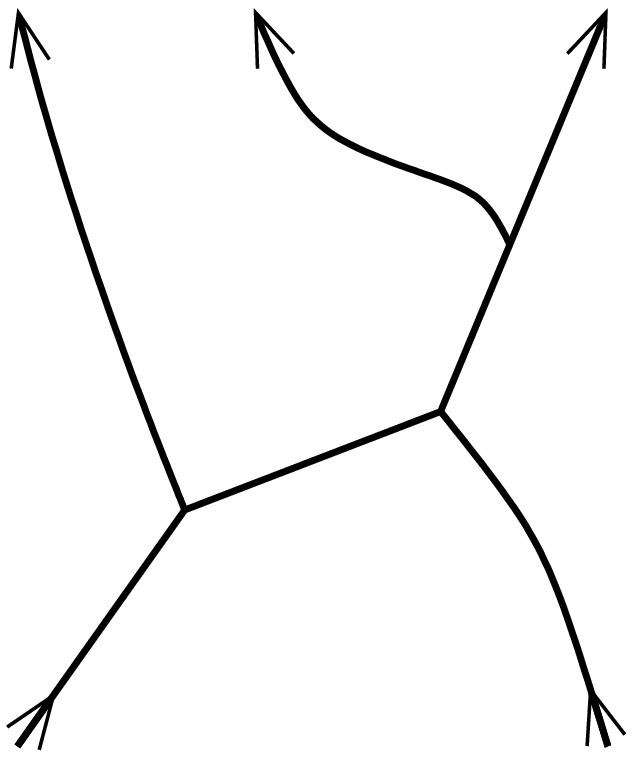}\medskip
\end{equation*}

\n By Lemma~\ref{lem:digon} we have 
\begin{equation*}
\labellist 
\pinlabel $\cong$ at 248 110
\pinlabel $\cong$ at 560 110
\pinlabel $\oplus\ q^2$ at 890 110
\pinlabel $B'$ at 730 -25
\pinlabel $B$ at 1040 -25
\tiny\hair 2pt
\pinlabel $2$ at  -7  14
\pinlabel $1$ at  -7 205
\pinlabel $2$ at 188  14
\pinlabel $2$ at 188 205
\pinlabel $1$ at  65 205
\pinlabel $1$ at  68 100
\pinlabel $2$ at 110 123
\pinlabel $1$ at 108  52
\pinlabel $1$ at 157  80
\pinlabel $3$ at 162 136
\pinlabel $2$ at 315  14
\pinlabel $1$ at 315 205
\pinlabel $2$ at 513  14
\pinlabel $2$ at 510 205
\pinlabel $1$ at 390 205
\pinlabel $1$ at 398 122
\pinlabel $2$ at 464  96
\pinlabel $1$ at 464  50
\pinlabel $1$ at 512  62
\pinlabel $3$ at 486 138
\pinlabel $2$ at 638  14
\pinlabel $1$ at 638 205
\pinlabel $2$ at 835  14
\pinlabel $2$ at 833 205
\pinlabel $1$ at 710 205
\pinlabel $1$ at 731  65
\pinlabel $3$ at 796 118
\pinlabel $2$ at  959  14
\pinlabel $1$ at  961 205
\pinlabel $2$ at 1158  14
\pinlabel $2$ at 1156 205
\pinlabel $1$ at 1031 205
\pinlabel $1$ at 1051  65
\pinlabel $3$ at 1117 118
\endlabellist
\figs{0.22}{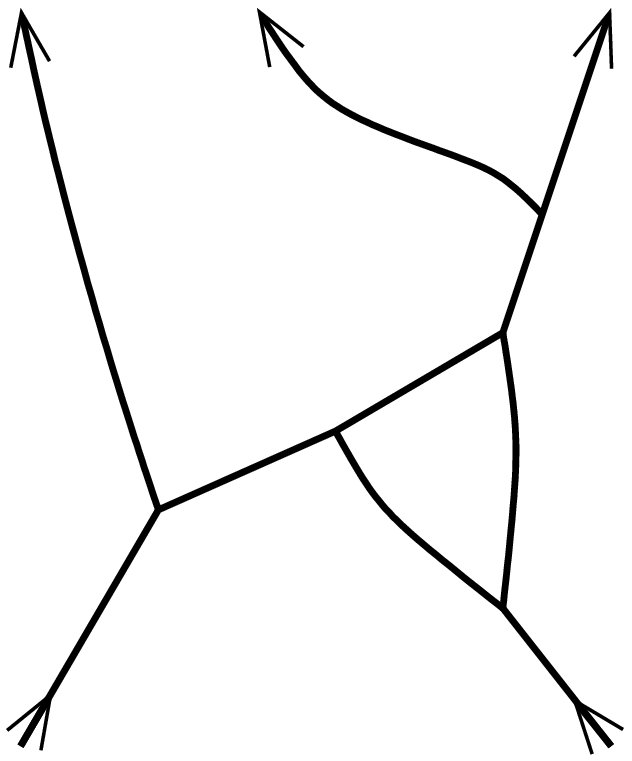}\mspace{50mu}
\figs{0.22}{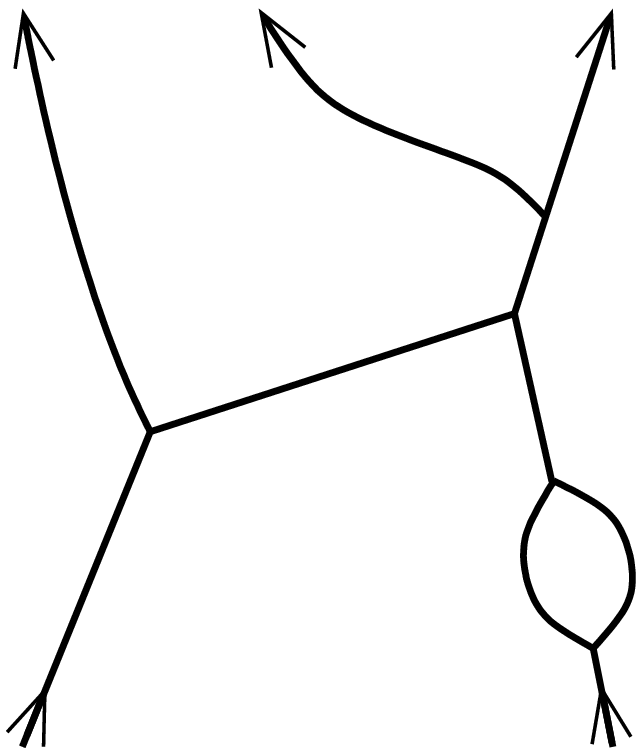}\mspace{50mu}
\figs{0.22}{hdumb-r}\mspace{50mu}
\figs{0.22}{hdumb-r}\medskip
\end{equation*}

\n By Lemma~\ref{lem:2113-square} we have 
\begin{equation*}
\labellist 
\pinlabel $\cong$ at 260 110
\pinlabel $\oplus\ q^2$ at 574 110
\pinlabel $B'$ at 400 -25
\pinlabel $D$ at 720 -25
\tiny\hair 2pt
\pinlabel $2$ at  -8  14
\pinlabel $1$ at  -8 205
\pinlabel $2$ at 186  14
\pinlabel $2$ at 186 205
\pinlabel $1$ at  64 205
\pinlabel $1$ at  58  85
\pinlabel $2$ at 100 133
\pinlabel $1$ at 100  38
\pinlabel $1$ at 160  85
\pinlabel $3$ at 170 144
\pinlabel $2$ at 309  14
\pinlabel $1$ at 311 205
\pinlabel $2$ at 502  14
\pinlabel $2$ at 502 205
\pinlabel $1$ at 400 205
\pinlabel $3$ at 440 126
\pinlabel $4$ at 430  84
\pinlabel $2$ at 628  14
\pinlabel $1$ at 628 205
\pinlabel $2$ at 821  14
\pinlabel $2$ at 821 205
\pinlabel $1$ at 698 205
\pinlabel $1$ at 721  65
\pinlabel $3$ at 784 118
\endlabellist
\figs{0.22}{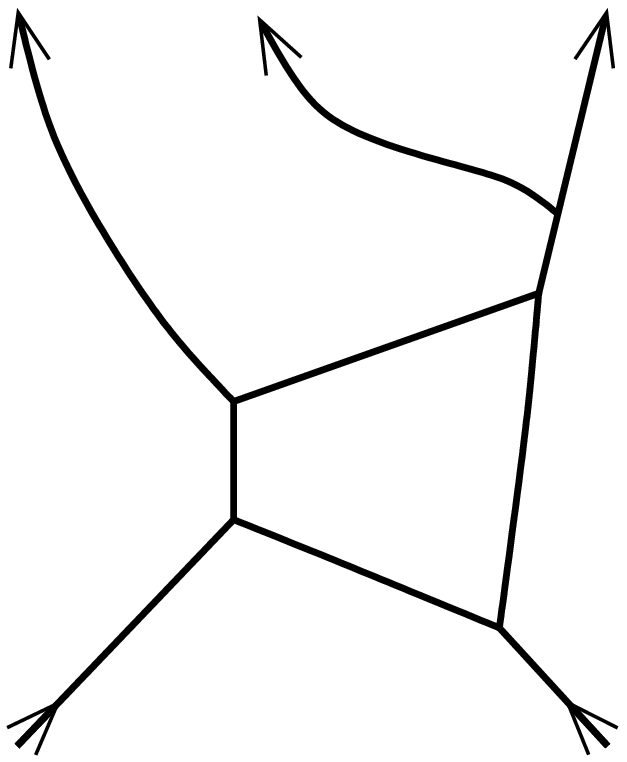}\mspace{50mu}
\figs{0.22}{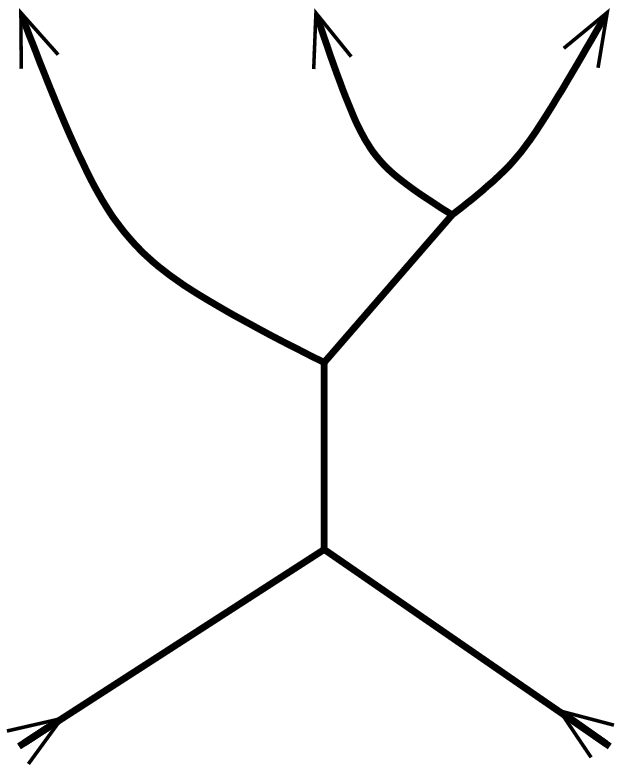}\mspace{50mu}
\figs{0.22}{hdumb-r}\medskip
\end{equation*}

\n Note that
\begin{equation*}
\labellist
\pinlabel $\cong$ at 250 100
\tiny\hair 2pt
\pinlabel $4$ at  64  16
\pinlabel $1$ at  -8 200
\pinlabel $2$ at 190 200
\pinlabel $1$ at  86 200
\pinlabel $2$ at  48  86
\pinlabel $4$ at  400  16
\pinlabel $1$ at  318 200
\pinlabel $2$ at  516 200
\pinlabel $1$ at  420 200
\pinlabel $3$ at  464  86
\endlabellist
\figs{0.16}{vertexleft}\mspace{40mu}\figs{0.16}{vertexright}
\end{equation*}

\n Finally apply Lemma~\ref{lem:aux1}, which is justified because  
\begin{enumerate}
\item $d\colon A\to B$ is zero, 
\item $d\colon B'\to D$ is zero,
\item $d\colon B\to B$ is the identity,
\item $d\colon B'\to B'$ is minus the identity,  \\
\item 
$
\labellist
\pinlabel $q^2$ at -15 110
\pinlabel $\xra{\mspace{30mu}}$ at 240 110
\pinlabel $\xra{\mspace{14mu} d \mspace{8mu}}$ at 620 122
\pinlabel $q^{-2}$ at 745 110
\pinlabel $\xra{\mspace{30mu}}$ at 1020 110
\pinlabel $q^{-2}$ at 1135 110
\tiny\hair 2pt
\pinlabel $2$ at  18  16
\pinlabel $1$ at  -7 202
\pinlabel $1$ at  86 202
\pinlabel $2$ at 112  16
\pinlabel $2$ at 350  16
\pinlabel $1$ at 352 202
\pinlabel $2$ at 553  16
\pinlabel $2$ at 550 202
\pinlabel $1$ at 420 202
\pinlabel $1$ at 418 105
\pinlabel $2$ at 445 133
\pinlabel $1$ at 480 165
\pinlabel $1$ at 525 115
\pinlabel $1$ at 460  60
\pinlabel $2$ at 754  16
\pinlabel $1$ at 758 202
\pinlabel $2$ at 957  16
\pinlabel $2$ at 952 202
\pinlabel $1$ at 827 202
\pinlabel $3$ at 812 102
\pinlabel $1$ at 840 148
\pinlabel $1$ at 888 174
\pinlabel $1$ at 928 108
\pinlabel $1$ at 860  46
\pinlabel $2$ at 1160  16
\pinlabel $1$ at 1164 202
\pinlabel $2$ at 1364  16
\pinlabel $2$ at 1359 202
\pinlabel $1$ at 1258 202
\pinlabel $3$ at 1225  85
\pinlabel $2$ at 1238 131
\pinlabel $1$ at 1278 138
\pinlabel $1$ at 1336  85
\pinlabel $1$ at 1273  37
\endlabellist
\centering
\mspace{10mu}
\raisebox{-17pt}{
\figs{0.205}{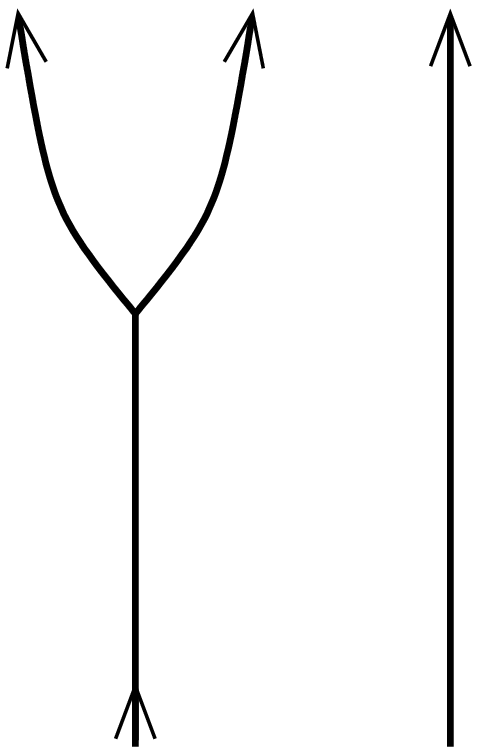}
\hspace{8ex}
\figs{0.20}{triangle5legs}
\hspace{8ex}
\figs{0.20}{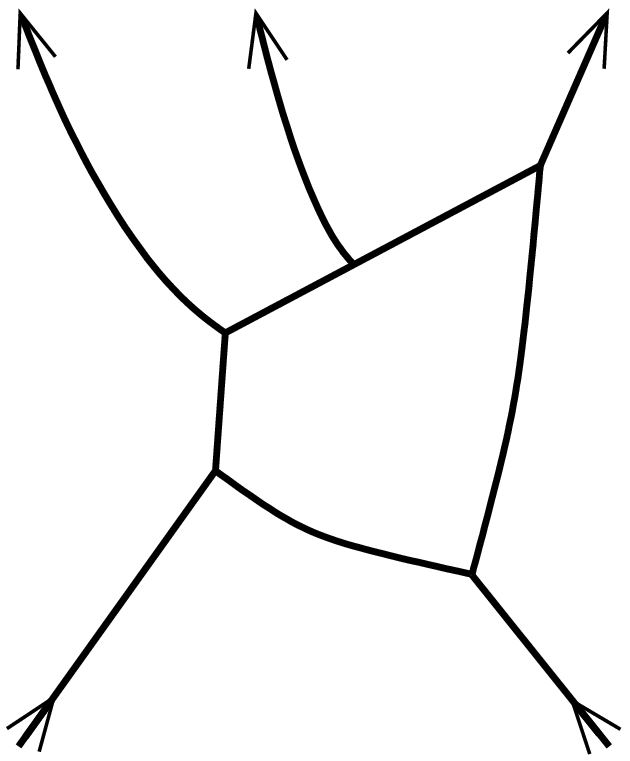}
\hspace{8ex}
\figs{0.20}{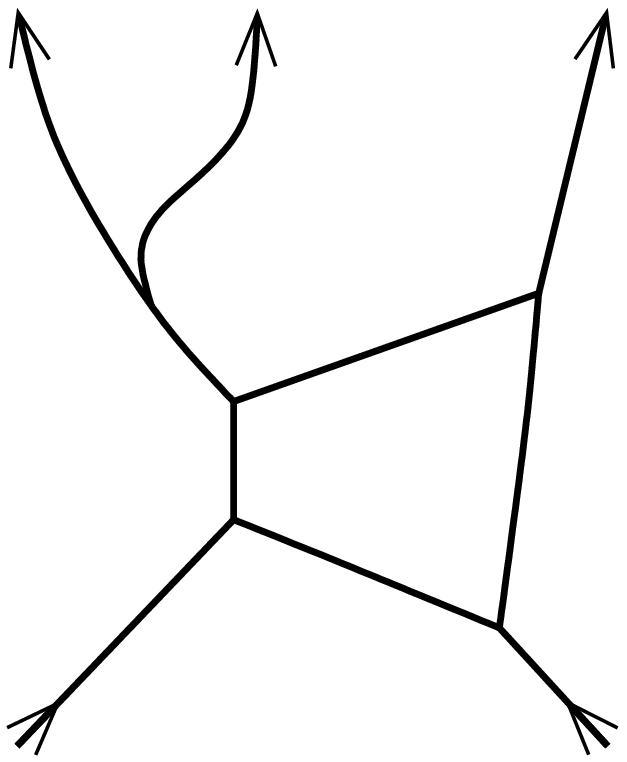}
\qquad}
$\quad
\\[2ex]%
\hspace*{8ex} 
$
\labellist
\pinlabel $=$ at -90 102
\pinlabel $q^2$ at -15 110
\pinlabel $\xra{\mspace{14mu} d \mspace{8mu}}$ at 250 122
\pinlabel $q^{-2}$ at 390 110
\tiny\hair 2pt
\pinlabel $2$ at 18 16
\pinlabel $1$ at  -7 202
\pinlabel $1$ at  86 202
\pinlabel $2$ at 112  16
\pinlabel $2$ at 400  16
\pinlabel $1$ at 404 202
\pinlabel $2$ at 600  16
\pinlabel $2$ at 598 202
\pinlabel $1$ at 498 202
\pinlabel $3$ at 465  85
\pinlabel $2$ at 478 131
\pinlabel $1$ at 518 138
\pinlabel $1$ at 576  85
\pinlabel $1$ at 513  37
\endlabellist
\centering
\mspace{40mu}
\raisebox{-17pt}{
\figs{0.205}{halfdumb}
\hspace{10ex}
\figs{0.20}{sqdumb-l}
}\ ,
$
\\ \llap{\hspace*{-6ex}and}\\
\item 
$
\labellist
\pinlabel $q^{-2}$ at -45 110
\pinlabel $\xra{\mspace{30mu}}$ at 250 110
\pinlabel $q^{-2}$ at 365 110
\pinlabel $\xra{\mspace{14mu} d \mspace{8mu}}$ at 655 122
\pinlabel $q^{-4}$ at 770 110
\pinlabel $\xra{\mspace{30mu}}$ at 1070 110
\pinlabel $q^{-4}$ at 1180 110
\tiny\hair 2pt
\pinlabel $2$ at  -9  16
\pinlabel $1$ at  -7 202
\pinlabel $1$ at  86 202
\pinlabel $2$ at 190  16
\pinlabel $2$ at 188 202
\pinlabel $3$ at  54  86
\pinlabel $2$ at  68 130
\pinlabel $1$ at 108 140
\pinlabel $1$ at 164  86
\pinlabel $1$ at 102  36
\pinlabel $2$ at 390  16
\pinlabel $1$ at 394 202
\pinlabel $2$ at 593  16
\pinlabel $2$ at 590 202
\pinlabel $1$ at 470 202
\pinlabel $3$ at 453 100
\pinlabel $2$ at 478 150
\pinlabel $1$ at 528 176
\pinlabel $1$ at 565 105
\pinlabel $1$ at 500  44
\pinlabel $2$ at 796  16
\pinlabel $1$ at 800 202
\pinlabel $2$ at 997  16
\pinlabel $2$ at 997 202
\pinlabel $1$ at 869 202
\pinlabel $3$ at 858  87
\pinlabel $2$ at 905 138
\pinlabel $3$ at 978 144
\pinlabel $1$ at 970  82
\pinlabel $1$ at 910  36
\pinlabel $2$ at 1200  16
\pinlabel $1$ at 1206 205
\pinlabel $2$ at 1404  16
\pinlabel $2$ at 1400 205
\pinlabel $1$ at 1296 205
\pinlabel $4$ at 1278  90
\pinlabel $2$ at 1294 150
\endlabellist
\centering
\mspace{25mu}
\raisebox{-17pt}{
\figs{0.205}{sqdumb-l}
\hspace{8ex}
\figs{0.20}{pentagon5legs}
\hspace{8ex}
\figs{0.20}{sqdumb}
\hspace{8ex}
\figs{0.20}{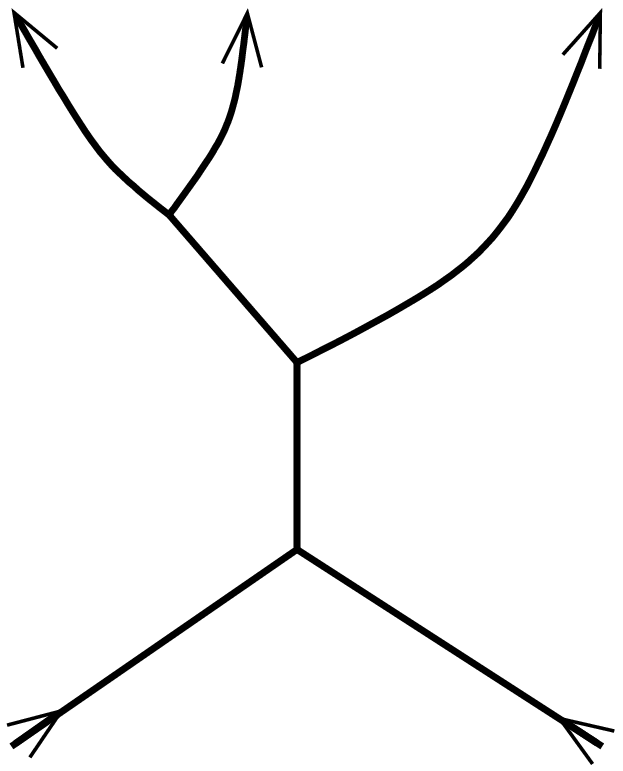}
\qquad}
$\quad
\\[2ex]%
\hspace*{8ex} 
$
\labellist
\pinlabel $=$ at -100 102
\pinlabel $q^{-2}$ at -25 110
\pinlabel $\xra{\mspace{14mu} d \mspace{8mu}}$ at 250 122
\pinlabel $q^{-4}$ at 390 110
\tiny\hair 2pt
\pinlabel $2$ at  -9  16
\pinlabel $1$ at  -7 202
\pinlabel $1$ at  86 202
\pinlabel $2$ at 190  16
\pinlabel $2$ at 188 202
\pinlabel $3$ at  54  86
\pinlabel $2$ at  68 130
\pinlabel $1$ at 108 140
\pinlabel $1$ at 164  86
\pinlabel $1$ at 102  36
\pinlabel $2$ at 440  16
\pinlabel $1$ at 446 205
\pinlabel $2$ at 644  16
\pinlabel $2$ at 640 205
\pinlabel $1$ at 538 205
\pinlabel $4$ at 518  90
\pinlabel $2$ at 536 150
\endlabellist
\centering
\mspace{50mu}
\raisebox{-17pt}{
\figs{0.205}{sqdumb-l}
\hspace{10ex}
\figs{0.20}{dumbdumb}
}\ .$
\end{enumerate}
\medskip

\n All assertions follow from straightforward computations and can easily be checked. 
For the first assertion one only has to 
compute the image of $1\stens 1$, which is zero indeed. For the second, third and 
fourth it suffices to compute the images of $1\stens 1\stens 1$ and $1\stens x_1\stens 1$. 
For the fifth it suffices to compare the images of $1\stens 1$. For the sixth one has 
to compare the images of $1\stens 1\stens 1\stens 1$, $1\stens 1\stens x_3\stens 1$ and 
$x_3\stens 1\stens 1\stens 1$. 
\end{proof}        
Of course there are also homotopy equivalences analogues 
to the one in Lemma~\ref{lem:slide22} for a 
negative crossing or a $11$-splitting of the right top 
strand. \\

In the following lemma the top and bottom parts are 
complexes again.    

\begin{lem}
\label{lem:aux2} 
If $sf=0$, the diagram below defines a homotopy 
equivalence between the top and the bottom complex.
$$
\xymatrix@C=28.2mm@R=10mm{ 
&  
\ A\oplus A'
\ar@{<-}@/^{1.6pc}/[ddd]|<(0.60)\hole^<(0.42){\bigl(\begin{smallmatrix}0\\ f \end{smallmatrix}\bigr)}
\ar@/_{1.6pc}/[ddd]|<(0.594)\hole_<(0.39){(-h,0)}
\ar@<0.1pc>[dl]^{(1,0)}
\ar@<0.4pc>[dr]^{\bigl(\begin{smallmatrix}r&s\\ u&-1\end{smallmatrix}\bigr)}
 &  \\
\mspace{-50mu}\cD\colon\mspace{28mu} A
\ar@{<->}[dd]_{0} \ar[dr]_{h} 
\ar@<0.4pc>[ur]^{\bigl(\begin{smallmatrix}1\\g\end{smallmatrix}\bigr) } 
   &      & 
C\oplus A'
\ar@<0.1pc>[dd]^{(\id,-s)}
\ar@<0.1pc>[ul]^{\bigl(\begin{smallmatrix}0&0\\0&-1\end{smallmatrix}\bigr)}
 \\             & 
B
\ar[ur]_{\bigl(\begin{smallmatrix}j\\ f\end{smallmatrix}\bigr) }   
 &  \\
\mspace{-50mu}\cD'\colon\mspace{24mu} 0\ \ar[r]        & 
B
\ar[r]^j \ar@{<->}[u]|\id   & 
C
\ar@<0.3pc>[uu]^{(\id,0)}
 }$$
\end{lem}       
\proof 

All maps are indicated in the figure, so the reader can check everything 
easily. Two identities are helpful: since the top of the figure is a complex, we 
have
$$-j h=r+s g\quad\mbox{and}\quad f h=g-u.
\rlap{\hspace*{22.3ex}\qed}
$$ 

\begin{lem}
We have the homotopy equivalence
\begin{equation}
\label{eq:slide12}
\labellist
\pinlabel $\cong$ at 264 110
\tiny\hair 2pt
\pinlabel $1$ at -8 14
\pinlabel $1$ at -7 177
\pinlabel $1$ at 70 204
\pinlabel $2$ at 186 14
\pinlabel $1$ at 339 14
\pinlabel $1$ at 339 187
\pinlabel $1$ at 431 214
\pinlabel $2$ at 544 14
\endlabellist
\centering
\raisebox{-16pt}{\figs{0.2}{slide1}\qquad\quad \figs{0.2}{slide2}}
\end{equation}
\end{lem}

\begin{proof} The complex associated to the l.h.s., denoted 
$\cD'$, is given in 
Figure~\ref{fig:slide12cplx-l}. 
\begin{figure}[h!]
\labellist
\pinlabel $\cD'\colon$ at -180 114
\pinlabel $q^{2}$ at -50 114
\tiny\hair 2pt
\pinlabel $1$ at  -8  14
\pinlabel $1$ at  -8 205
\pinlabel $2$ at 188  14
\pinlabel $1$ at 188 205 
\pinlabel $1$ at 118 205
\pinlabel $2$ at  61 132
\pinlabel $1$ at  84  72
\pinlabel $1$ at 492  14
\pinlabel $1$ at 492 205
\pinlabel $2$ at 684  14
\pinlabel $1$ at 682 205
\pinlabel $1$ at 608 205
\pinlabel $3$ at 598  84
\pinlabel $2$ at 584 146
\endlabellist
\centering
\figs{0.22}{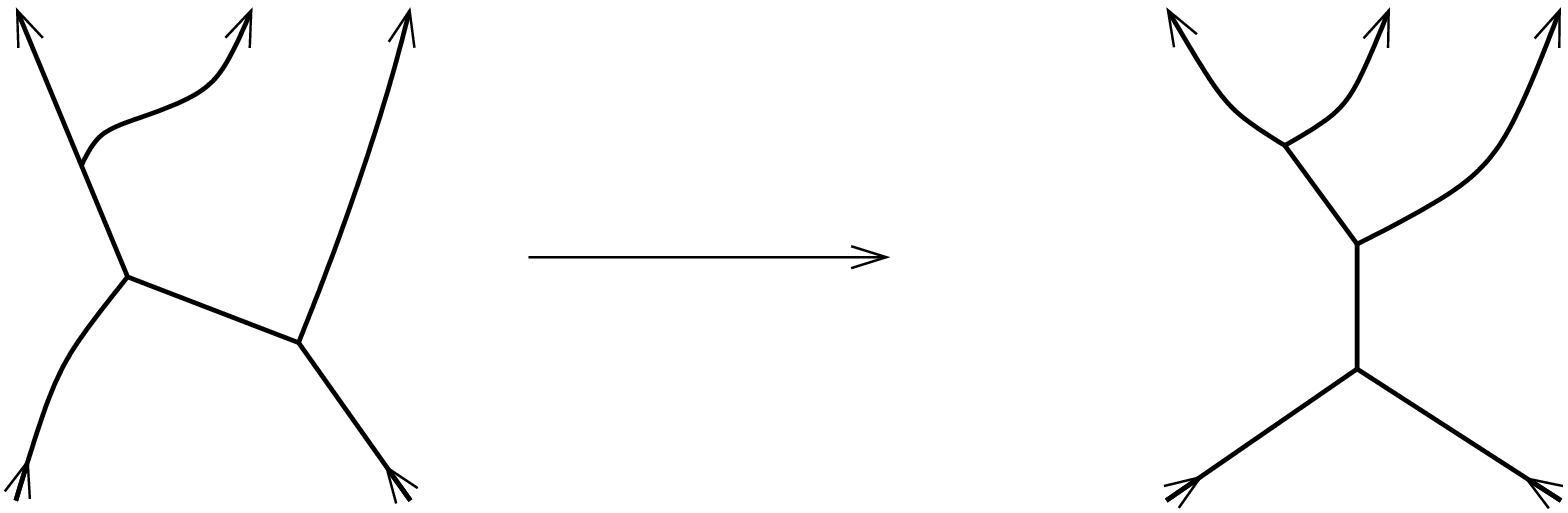}
\caption{Complex of the l.h.s. of Equation~\eqref{eq:slide12}}
\label{fig:slide12cplx-l}
\end{figure}
The complex associated to the r.h.s., denoted $\cD$, is given 
in Figure~\ref{fig:slide12cplx-r}. 
\begin{figure}[h!]
\labellist
\pinlabel $\cD\colon$ at -180 344
\pinlabel $q^{2}$ at 450 120
\pinlabel $q^{2}$ at 450 580
\pinlabel $q^{4}$ at -50 340
\tiny\hair 2pt
\pinlabel $1$ at  -8 246
\pinlabel $2$ at 184 246
\pinlabel $1$ at 184 437
\pinlabel $1$ at  90 437
\pinlabel $1$ at 488 475
\pinlabel $2$ at 682 474
\pinlabel $2$ at 682 664
\pinlabel $1$ at 588 664
\pinlabel $2$ at 619 598
\pinlabel $1$ at 608 548
\pinlabel $1$ at 661 548
\pinlabel $1$ at 489  14
\pinlabel $1$ at 490 204
\pinlabel $2$ at 681  14
\pinlabel $1$ at 681 204
\pinlabel $1$ at 610 204
\pinlabel $2$ at 524 118
\pinlabel $1$ at 584 102
\pinlabel $1$ at  984 244
\pinlabel $1$ at  984 435
\pinlabel $2$ at 1180 244
\pinlabel $1$ at 1178 435
\pinlabel $1$ at 1056 435
\pinlabel $2$ at 1044 316
\pinlabel $1$ at 1154 316
\pinlabel $2$ at 1162 374
\pinlabel $1$ at 1092 366
\pinlabel $1$ at 1092 266
\endlabellist
\centering
\figs{0.22}{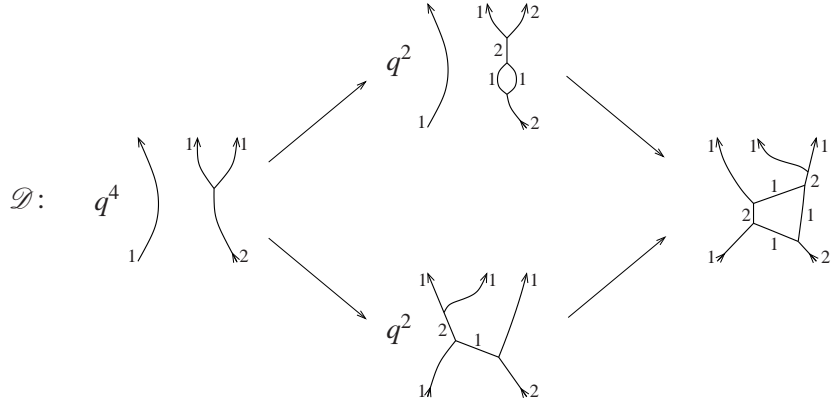}
\caption{Complex of the r.h.s. of Equation~\eqref{eq:slide12}}
\label{fig:slide12cplx-r}
\end{figure}

This time we have
\begin{equation*}
\labellist 
\pinlabel $\cong$ at 205 110
\pinlabel $\oplus\ q^2$ at 488 110
\pinlabel $A$ at 335 -25
\pinlabel $A'$ at 608 -25
\tiny\hair 2pt
\pinlabel $1$ at  -8  12
\pinlabel $1$ at  54 204
\pinlabel $1$ at 143 204
\pinlabel $1$ at  70  85
\pinlabel $1$ at 126  85
\pinlabel $2$ at 112 133
\pinlabel $2$ at 116  12
\pinlabel $1$ at 267  12
\pinlabel $2$ at 392  12
\pinlabel $1$ at 418 205
\pinlabel $1$ at 328 205
\pinlabel $1$ at 542  12
\pinlabel $2$ at 669  12
\pinlabel $1$ at 693 205
\pinlabel $1$ at 603 205
\endlabellist
\figs{0.22}{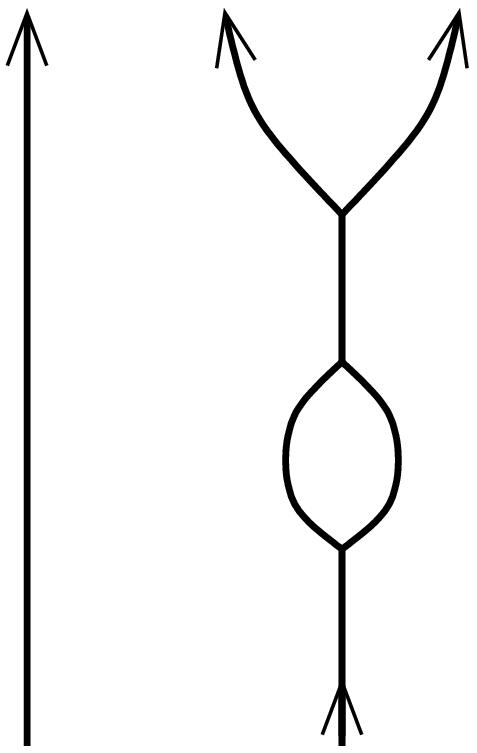}\mspace{50mu}
\figs{0.22}{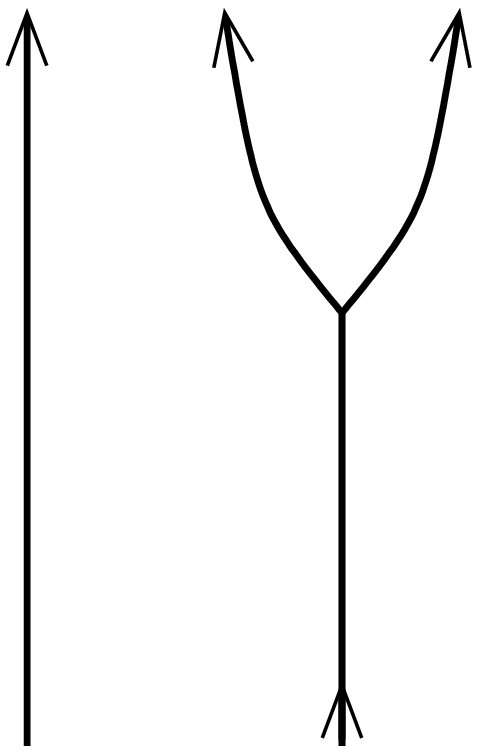}\mspace{50mu}
\figs{0.22}{halfdumb-r}\medskip
\end{equation*}
and
\begin{equation*}
\labellist 
\pinlabel $\cong$ at 205 110
\pinlabel $\oplus\ q^2$ at 490 110
\pinlabel $C$ at 332 -25
\pinlabel $A'$ at 610 -25
\tiny\hair 2pt
\pinlabel $1$ at  -6  14
\pinlabel $2$ at 144  14
\pinlabel $1$ at  -6 204
\pinlabel $1$ at  54 204
\pinlabel $1$ at 144 204
\pinlabel $2$ at  10  78
\pinlabel $1$ at 118  78
\pinlabel $2$ at 122 138
\pinlabel $1$ at  65 124
\pinlabel $1$ at  65  32 
\pinlabel $1$ at 269  12
\pinlabel $1$ at 269 204
\pinlabel $2$ at 386  12
\pinlabel $1$ at 419 204
\pinlabel $1$ at 330 204
\pinlabel $3$ at 314  84
\pinlabel $2$ at 354 114 
\pinlabel $1$ at 546  12
\pinlabel $2$ at 667  12
\pinlabel $1$ at 697 205
\pinlabel $1$ at 605 205
\endlabellist
\figs{0.22}{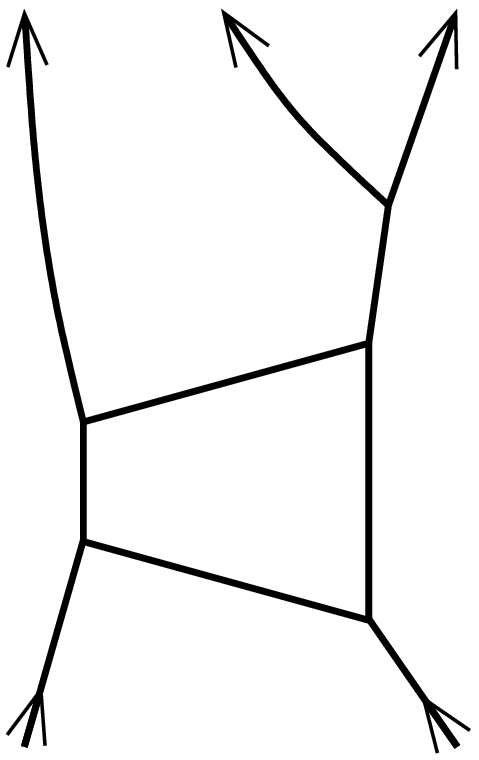}\mspace{50mu}
\figs{0.22}{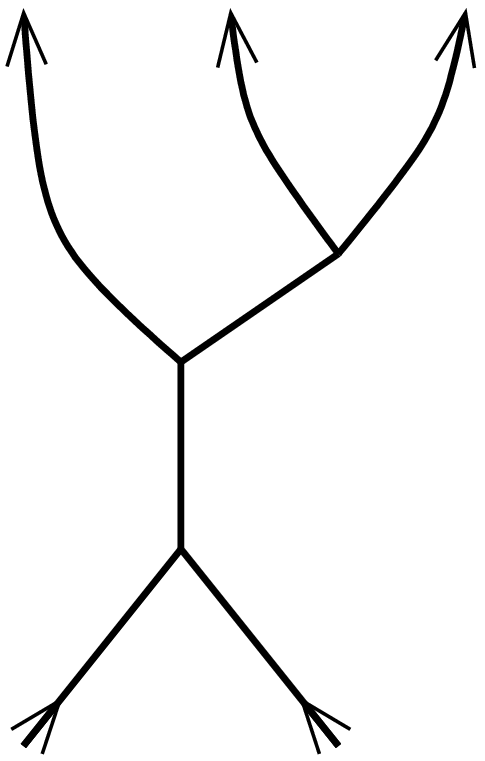}\mspace{50mu}
\figs{0.22}{halfdumb-r}\medskip
\end{equation*}

We can apply Lemma~\ref{lem:aux2} because
\begin{enumerate}
\item $f$ is equal to the identity,
\item $u$ is equal to minus the identity,
\item $s f=0$ and \\
\item 
$
\labellist
\pinlabel $q^{-2}$ at -55 110
\pinlabel $\xra{\mspace{14mu} d \mspace{8mu}}$ at 250 122
\pinlabel $q^{-4}$ at 390 110
\tiny\hair 2pt
\pinlabel $1$ at  -9  16
\pinlabel $1$ at  -7 202
\pinlabel $1$ at 120 202
\pinlabel $2$ at 192  16
\pinlabel $1$ at 190 202
\pinlabel $2$ at  62 128
\pinlabel $1$ at  90  64
\pinlabel $1$ at 440  16
\pinlabel $1$ at 446 205
\pinlabel $2$ at 644  16
\pinlabel $1$ at 640 205
\pinlabel $1$ at 538 205
\pinlabel $3$ at 518  90
\pinlabel $2$ at 536 150
\endlabellist
\centering
\mspace{28mu}
\raisebox{-17pt}{
\figs{0.205}{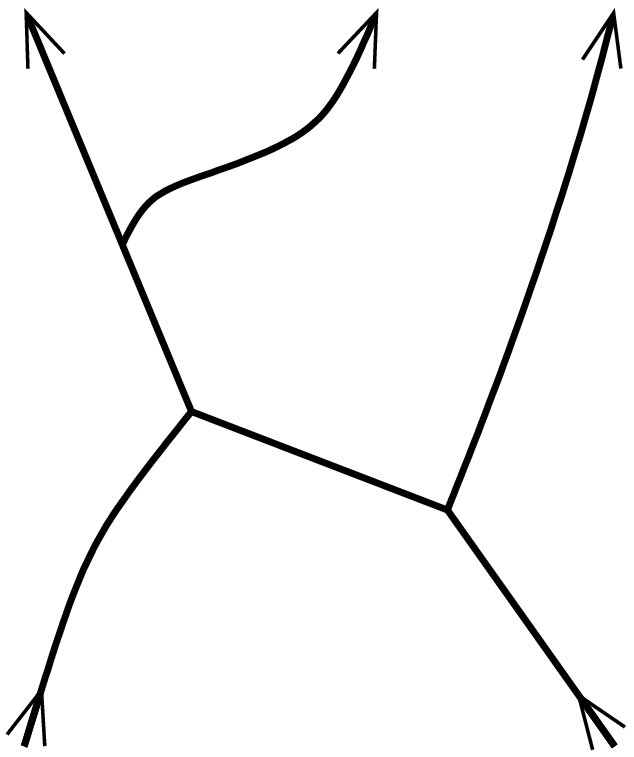}
\hspace{10ex}
\figs{0.20}{dumbdumb}
}$
\quad
\\[2ex]%
\hspace*{8ex} 

$
\labellist
\pinlabel $=$ at -130 102
\pinlabel $q^{-2}$ at -45 110
\pinlabel $\xra{\mspace{14mu}d\mspace{8mu}}$ at 250 122
\pinlabel $q^{-4}$ at 365 110
\pinlabel $\xra{\mspace{30mu}}$ at 655 110
\pinlabel $q^{-4}$ at 770 110
\pinlabel $\xra{\mspace{30mu}}$ at 1070 110
\pinlabel $q^{-4}$ at 1180 110
\tiny\hair 2pt
\pinlabel $1$ at  -9  16
\pinlabel $1$ at  -7 202
\pinlabel $1$ at 120 202
\pinlabel $2$ at 192  16
\pinlabel $1$ at 190 202
\pinlabel $2$ at  62 128
\pinlabel $1$ at  90  64
\pinlabel $1$ at 396  16
\pinlabel $1$ at 396 202
\pinlabel $2$ at 599  16
\pinlabel $1$ at 596 202
\pinlabel $1$ at 470 202
\pinlabel $2$ at 457  86
\pinlabel $1$ at 500 137
\pinlabel $2$ at 578 146
\pinlabel $1$ at 570  86
\pinlabel $1$ at 506  38
\pinlabel $1$ at 796  16
\pinlabel $1$ at 800 204
\pinlabel $2$ at 999  16
\pinlabel $1$ at 997 204
\pinlabel $1$ at 886 204
\pinlabel $3$ at 886  87
\pinlabel $2$ at 936 130
\pinlabel $1$ at 1200  16
\pinlabel $1$ at 1206 205
\pinlabel $2$ at 1404  16
\pinlabel $1$ at 1400 205
\pinlabel $1$ at 1296 205
\pinlabel $3$ at 1278  90
\pinlabel $2$ at 1294 150
\endlabellist
\centering
\mspace{70mu}
\raisebox{-17pt}{
\figs{0.205}{hdumb-l}
\hspace{8ex}
\figs{0.20}{sqdumb}
\hspace{8ex}
\figs{0.20}{dumbdumb-rot}
\hspace{8ex}
\figs{0.20}{dumbdumb}
}
$\ .
\end{enumerate}
\medskip

\n Since we have given all the maps the reader can check the claims 
by straightforward computations. Note that for the first two claims it suffices to 
compute the image of $1\stens 1$. For the last two claims one has to compute the 
images of $1\stens 1\stens 1$ and $1\stens x_2\stens 1$. 
\end{proof}

\noindent Again, there are analogous homotopy equivalences 
for a 
negative crossing or if one swaps the 
$1$- and $2$-strands in the lemma above.

\section{Invariance under the Markov moves}
\label{sec:Markov}

\subsection{Hochschild homology of bimodules as the homology of a Koszul complex 
of polynomial rings}

The Hochschild homology of a bimodule over the polynomial ring can be obtained 
as the homology of a corresponding Koszul 
complex of certain polynomial rings in many variables. This idea was explained 
and used by Khovanov in~\cite{Kh}. Here we shall briefly describe how to 
extend it to our case.


First of all, we change our polynomial notation in this section. Namely, the 
polynomial ring $R_{i_1\ldots i_k}$ can be 
represented as the polynomial ring in the variables which are the $i_1$ elementary symmetric polynomials in the first $i_1$ 
variables, the $i_2$ 
elementary symmetric polynomials in the following $i_2$ variables, etc., and the 
$i_k$ elementary symmetric polynomials in the 
last $i_k$ variables. In this section we will always work with these ``new" 
variables, i.e. the elementary symmetric polynomials, because it is more 
convenient for our purposes here. 
Thus to each strand labelled by $k$, we associate the $k$ variables 
$x_1,\ldots,x_k$, such that the degree of $x_i$ is equal to $2i$, for 
$i=1,\dots, k$.

To begin with, we describe which polynomial ring to associate to a web. 
Take a web and choose a height function to 
separate the vertices according to height. This way chop up the web into 
several layers with only one vertex. To each layer we associate a new set of 
variables. To each vertex with incident edges labelled $i$, $j$ and $i+j$
we associate the $i+j$ polynomials which are the differences of the $k$-th 
symmetric polynomials in the outgoing variables and 
the incoming variables, for every $k=1,\ldots,i+j$. Moreover, if there are two 
different sets of variables $x_1,\ldots,x_k$ and $x'_1,\ldots,x'_k$ associated 
to a given edge labelled $k$, we also associate the $k$ polynomials 
$x_i-x'_i$, for all $i=1,\ldots,k$. Finally, to the whole graph we 
associate the polynomial ring in all the variables modded out by the ideal 
generated by all the polynomials associated to the vertices and edges.

There is an isomorphism between these polynomial rings and the corresponding 
bimodules associated to the graph. Indeed, to 
the tensor product $p(\underline{x})\otimes q(\underline{x})$ in variables $\underline{x}$, corresponds the polynomial 
$p(\underline{x'})q(\underline{x})$, where the variables $\underline{x}$ are of the bottom layer, and $\underline{x'}$ are of 
the top layer. Loosely speaking, the position of the factor in the tensor product corresponds to the same polynomial in the 
variables corresponding to the layer.

Let us do an example:
\begin{exe}
Consider the web $\unorsmoothing_{21}$:
\begin{equation*}
\labellist
\tiny\hair 2pt
\pinlabel $2$ at  5  36
\pinlabel $1$ at  107 36
\pinlabel $3$ at 80 100
\pinlabel $2$ at  3 186
\pinlabel $1$ at 111 186
\pinlabel $x_1,x_2$ at  -20  -10
\pinlabel $y_1$ at  120 -10
\pinlabel $x'_1,x'_2$ at  -10 230
\pinlabel $y'_1$ at  120 230
\endlabellist
\figs{0.25}{dumbell}
\end{equation*}

The polynomial ring associated to this web is the ring 
$P_{\unorsmoothing_{21}}:=\bC[x_1,x_2,y_1,x'_1,x'_2,y'_1]$. The ideal 
$I_{\unorsmoothing_{21}}$ by which we have to quotient, is generated by the 
differences of the symmetric polynomials in all 
top and bottom variables (since the middle edge is labelled by 3). There are 
three elementary symmetric polynomial in 
this case: $\Sigma_1=x_1+y_1$, 
$\Sigma_2=x_2+x_1y_1$ and $\Sigma_3=x_2y_1$, and so
$$I_{\unorsmoothing_{21}}=\langle x_1+y_1-x'_1-y'_1, x_2+x_1y_1-x'_2-x'_1y'_1,x_2y_1-x'_2y'_1 \rangle.$$
Hence, the polynomial ring we associate to it is given by
$$R_{\unorsmoothing_{21}}=P_{\unorsmoothing_{21}}/I_{\unorsmoothing_{21}}.$$

On the other hand, the bimodule $\hatdumbell_{21}$ associated to the web $\unorsmoothing_{21}$ is $R_{21}\otimes_3 
R_{21}$, and its elements are linear combination of elements of the form 
$p(x_1,x_2,y_1)\otimes q(x_1,x_2,y_1)$, for some 
polynomials $p$ and $q$. Finally, the isomorphism between $\hatdumbell_{21}$ and $R_{\unorsmoothing_{21}}$ is given 
by:
$$ p(x_1,x_2,y_1)\otimes q(x_1,x_2,y_1) \longleftrightarrow p(x'_1,x'_2,y'_1)q(x_1,x_2,y_1).$$
\end{exe}

In such a way we have obtained the bijective correspondence between the bimodules 
$\widehat{\Gamma}$ and the polynomial rings 
$R_{\Gamma}$ that are associated to the open trivalent graph $\Gamma$. The closure of $\Gamma$ in the bimodule picture 
corresponds to taking the Hochschild homology of $\widehat{\Gamma}$. In the 
polynomial ring picture this is isomorphic to the 
homology of the Koszul complex over $R_{\Gamma}$ which is the tensor product of the the complexes
\begin{equation}
0\xrightarrow{\quad} R_{\Gamma}\{-1,2i-1\} {\xrightarrow{\ x_i-x'_i\ }} R_{\Gamma} \xrightarrow{\quad} 0,\label{koshoh}
\end{equation}
where the $x_i$'s are the bottom layer variables, and the $x'_i$'s are the 
top layer variables. We put the right-hand side term in (co)homological degree 
zero. The first shift is in the Hochschild (homological) degree, and the second 
one is in the $q$-degree, so that the maps have bi-degree $(1,1)$.

\subsection{Invariance under the first Markov move}

Essentially we have to show that the Hochs\-child homology of the tensor product 
$B_1\otimes B_2$ of the bimodules $B_1$ and 
$B_2$ is isomorphic to the Hochschild homology of $B_2\otimes B_1$, i.e.
\begin{equation}
HH(B_1\otimes B_2) = HH(B_2\otimes B_1).\label{m1}
\end{equation}

We shall prove this by passing to the polynomial ring and Koszul complex description from above. Recall, that to the 
open trivalent graph $\Gamma$ we have associated the polynomial algebra $P_{\Gamma}$ (the ring of polynomials in all 
variables) quotiented by the ideal $I_{\Gamma}$ generated by 
certain polynomials. Since for each layer we introduced new variables, the sequence of these polynomials is regular 
and so we have that $R_{\Gamma}=P_{\Gamma}/I_{\Gamma}$ is the homology of the 
Koszul complex obtained by tensoring together all complexes of the form 
\begin{equation*}
0\xrightarrow{\quad} P_{\Gamma} {\xrightarrow{\ f\ }} P_{\Gamma} \xrightarrow{\quad} 0,
\end{equation*} 
for $f\in I_{\Gamma}$. We call this 
the Koszul complex generated by $I_{\Gamma}$.

Finally, the object associated to the closure of the graph $\Gamma$ is given by the homology of the Koszul complex defined by 
\eqref{koshoh}, and so it is isomorphic to the homology of the Koszul complex 
generated by the polynomials which define $I_{\Gamma}$ together with the 
polynomials $x_i-x'_i$ which come from the closure. If $\Gamma$ is the vertical 
glueing of $\Gamma_1$ and $\Gamma_2$ then in $I_{\Gamma}$ we have the polynomials 
$y_j-y'_j$ which correspond to the edges which are glued together. The Hochschild 
homology of $\widehat{\Gamma_1}\otimes\widehat{\Gamma_2}$ is isomorphic to the 
homology generated by $I_{\Gamma_1},I_{\Gamma_2}$ and the polynomials 
$x_i-x'_i$ and $y_j-y'_j$. Clearly the same holds for $\widehat{\Gamma_2}\otimes\widehat{\Gamma_1}$, with the role of the $x_i-x'_i$ and $y_j-y'_j$ interchanged. 
Thus we have proved \eqref{m1}.

\subsection{Invariance under the second Markov move}

Invariance under the second Markov move corresponds to invariance under the Reidemeister move I. If the 
strand involved is labelled 1, the result was proved by Khovanov and Khovanov and Rozansky~\cite{Kh, KR2} (see below).
 
Remains to show invariance if the strand is labelled 2. For both 
the positive and the negative crossing, we have the same three resolutions:
\begin{equation*}
\labellist
\tiny\hair 2pt
\pinlabel $2$ at -15 114
\pinlabel $2$ at 140 114
\pinlabel $2$ at 530 114
\pinlabel $2$ at 685 114
\pinlabel $3$ at 530 205
\pinlabel $1$ at 685 205
\pinlabel $2$ at 530 300
\pinlabel $2$ at 685 300
\pinlabel $1$ at 610 125
\pinlabel $1$ at 610 283
\pinlabel $2$ at 1077 114
\pinlabel $2$ at 1230 114
\pinlabel $2$ at 1077 300
\pinlabel $2$ at 1229 300
\pinlabel $4$ at 1180 205
\endlabellist
\figs{0.22}{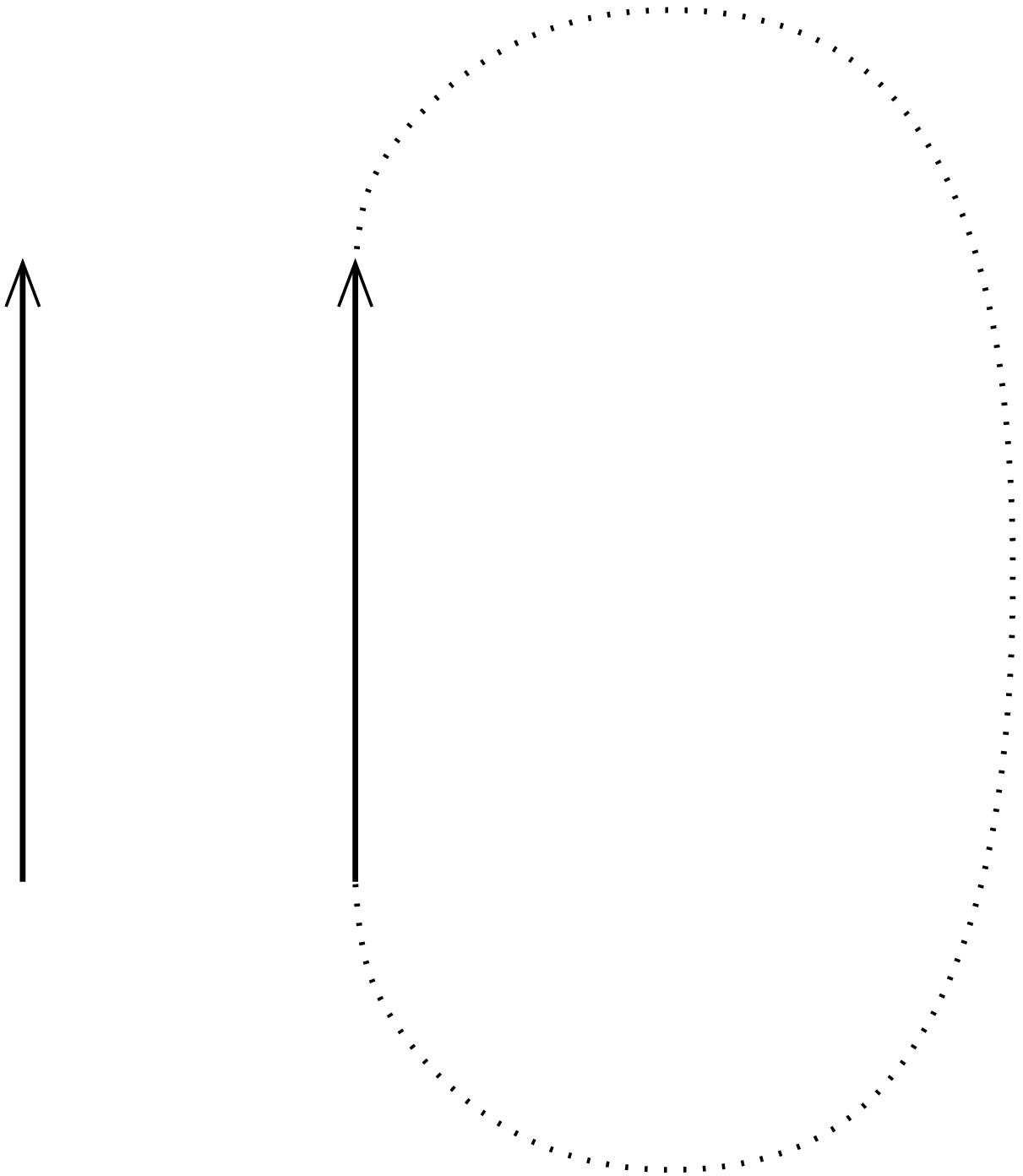}
\qquad\qquad
\figs{0.22}{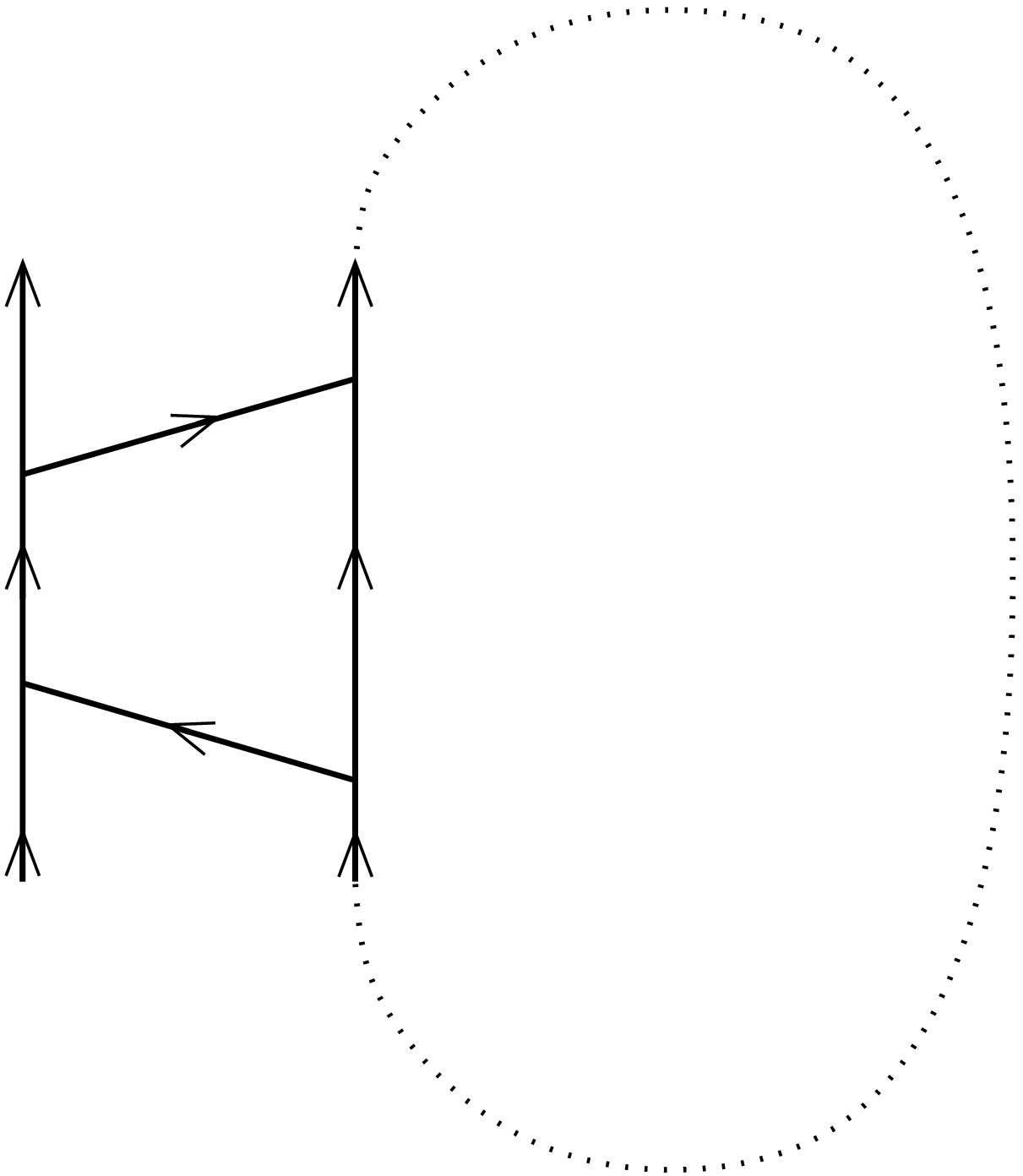}
\qquad\qquad
\figs{0.22}{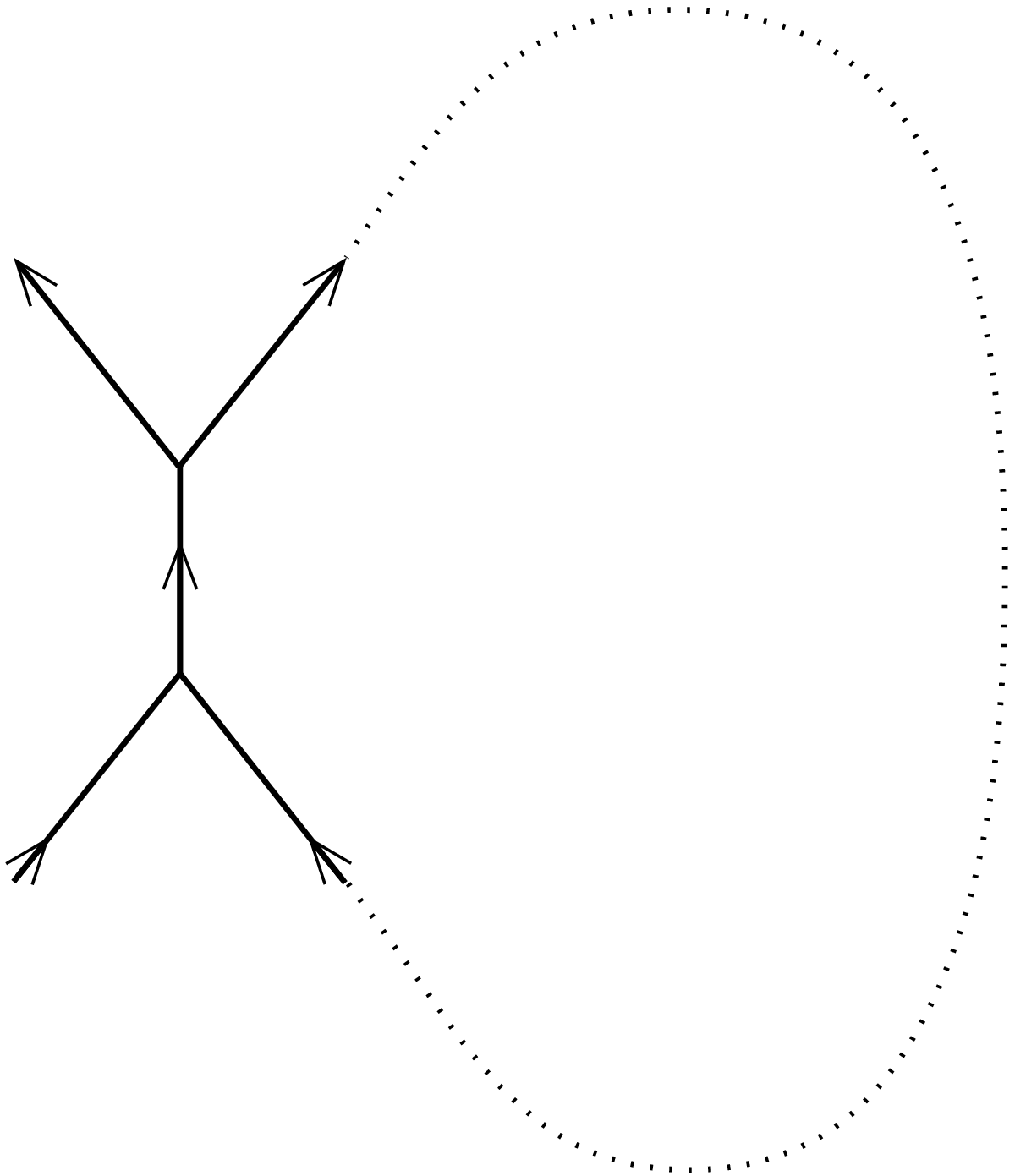}
\end{equation*}

For each of these, we shall give its description in a polynomial language, and compute the homology of the corresponding 
Koszul complex.\\

{\em The resolution} $\twoarcs$: Before closing the right strand, we have an open graph. To the bottom layer we associate variables
$x_1$ and $x_2$ to the left strand, and $y_1$ and $y_2$ to the right strand, while to the top layer we associate variables
$x'_1$ and $x'_2$ to the left strand, and $y'_1$ and $y'_2$ to the right strand. Then the ring $R_{\twoarcs}$ we associate to it, is the ring of polynomials in all these variables modded out by 
the 
ideal generated by the polynomials $x'_1-x_1$, $x'_2-x_2$, $y'_1-y_1$ and $y'_2-y_2$, i.e. it is isomorphic to the ring 
$$B:=\bC[x_1,x_2,y_1,y_2].$$ 

{\em The resolution} $\squareup$: The variables that we associate to the bottom and the top layers are the same as above. To the 
bottom middle strand we associate variable $z_1$, to the top middle strand we associate $z'_1$ and to the right strand we 
associate variable $t_1$. Then the corresponding ring $R_{\squareup}$ is the ring of polynomials in all these variables, 
modded out by the ideal generated by the polynomials $y_1-z_1-t_1$, $y_2-z_1t_1$, $y'_1-z'_1-t_1$, 
$y'_2-z'_1t_1$, 
$x_1+z_1-x'_1-z'_1$, $x_2+x_1z_1-x'_2-x'_1z'_1$, $x_2z_1-x'_2z'_1$. From the first four relations, we can exclude $t_1$ and obtain the quadratic relations for $z_1$ and $z'_1$: $z_1^2=y_1 z_1-y_2$ and ${z'_1}^2=y'_1 z'_1-y'_2$. Hence, 
every element from $R_{\squareup}$ can be written as $a+bz_1+cz'_1+dz_1z'_1$, where $a$, $b$, $c$ and $d$ are 
polynomials only in $x$'s and $y$'s (with or without primes).\\

{\em The resolution} $\unorsmoothing$: The variables we associate to the bottom and top layers are again the same. This time, the 
ring $R_{\unorsmoothing}$ we obtain by quotienting by the ideal generated by the following four polynomials: 
$x_1+y_1-x'_1-y'_1$, $x_2+y_2+x_1y_1-x'_2-y'_2-x'_1y'_1$, 
$x_2y_1+x_1y_2-x'_2y'_1-x'_1y'_2$ and $x_2y_2-x'_2y'_2$.\\

To the closure of the right strands of each of these graphs, corresponds the homology of the tensor product of the following 
two Koszul complexes:
\begin{gather*}
0\ \xrightarrow{\quad\ } R_{\Gamma} \{-1,1\}\ {\xrightarrow{y_1-y'_1}}\ R_{\Gamma} \xrightarrow{\quad\ } 0,\\
0\ \xrightarrow{\quad\ } R_{\Gamma} \{-1,3\}\ {\xrightarrow{y_2-y'_2}}\ R_{\Gamma} \xrightarrow{\quad\ } 0.
\end{gather*}

Hence, in all three cases we can have homology in three homological gradings $0$, $-1$ and $-2$. We denote this homology by 
$HH^R(\Gamma)$.\\

In the case of $\twoarcs$ we have that both differentials are $0$. Hence
\begin{align*}
 HH^R_0(\twoarcs)\mspace{8mu} &=B=\bC[x_1,x_2,y_1,y_2],\\
 HH^R_{-1}(\twoarcs) &=B\{1\}\oplus B\{3\},\\
 HH^R_{-2}(\twoarcs) &=B\{4\}.
\end{align*}

For the other two cases the computations are a bit more involved and we want to explain the general 
idea first. The Hochschild homology is the homology of a complex which is the tensor product of complexes 
of the form 
$$0\ \xrightarrow{\quad\ } P/I \ {\xrightarrow{p}}\ P/I \xrightarrow{\quad\ } 0,$$
where $P$ is a polynomial ring, $I$ is an ideal and $p\in R$ is a polynomial. Let us explain how 
to compute the homology of one such complex. 

The main part is the computation of the kernel and the cokernel of the map 
above. The cokernel is easily
computed, and is equal to the quotient ring $P/\langle I,p \rangle$. Now, let 
us pass to the kernel.
For any polynomial $q\in P$ to be in the kernel, i.e. to be a cocycle, we must have $pq\in I$. 
The ideals are always finitely generated, say $I$ is generated by $i_1,\ldots,i_n$. Therefore we have to 
find all solutions to the equation $a_1i_1+\cdots+a_ni_n=pq$, 
i.e. we rewrite $a_1i_1+\cdots+a_ni_n$ until it becomes of the form $p$ times something, which imposes 
constraints on the $a_i$. The solutions are generated by $q_1,\ldots,q_k$, which generate an ideal $Q$ 
(in the cases we are interested in, it is always a principal ideal, i.e. $Q=qP$, for some $q\in P$). 
Then the kernel is given by $Q/I$ (i.e. $qP/I$). However, this is isomorphic to $Q/\langle I,p \rangle$, since $pq_s\in I$ 
for all $s$, which is the form in which we present the results below. 

It is not hard to extend these calculations to a tensor product of complexes of the above form. In particular, all 
homologies are certain ideals, modded out by the ideal generated by $I_{\Gamma}$ and the polynomials $y_i-y'_i$, $i=1,2$, 
which in particular includes the polynomials $x_i-x'_i$, $i=1,2$ and also $z_1-z'_1$ in the case of the web $\squareup$. 
Hence, in the homology, all variables with primes are equal to the corresponding variables without the primes.\\
 
In the case of $\squareup$, we have that $y_2-y'_2=t_1(z_1-z'_1)=(y_1-z_1)(y_1-y'_1)$, and straightforward 
computations along the lines sketched above give
\begin{align*}
HH^R_0(\squareup)\mspace{8mu} &=A=\bC[x_1,x_2,y_1,y_2,z_1]/\langle z_1^2-y_1z_1+y_2 \rangle,\\
HH^R_{-1}(\squareup) &=\{\left((y_1-z_1)g+(x_2-y_2-z_1(x_1-y_1))h\right),g)| g, h\in A\}\subset A\{1\}\oplus A\{3\},\\
HH^R_{-2}(\squareup) &=(x_2-y_2-z_1(x_1-y_1))A\{4\}.
\end{align*}
Note that we have $A\cong B\oplus z_1 B$.

Finally, in the case of $\unorsmoothing$, we obtain
\begin{align}
HH^R_0(\unorsmoothing)\mspace{8mu} &=B=\bC[x_1,x_2,y_1,y_2],
\label{un220}\\
HH^R_{-1}(\unorsmoothing) &=\{\left(-c(x_2-y_2)+(cy_1+dy_2)(x_1-y_1),c(x_1-y_1)+d(x_2-y_2)\right)| c,d\in B\} 
\label{un221} \\
 & \mspace{30mu}\subset B\{1\}\oplus B\{3\}, \nonumber\\
HH^R_{-2}(\unorsmoothing) &=pB\{4\}, 
\label{un222}
\end{align}
where $p=(x_2-y_2)^2+(x_1-y_1)(x_1y_2-x_2y_1)$.\\

Now, we pass to the differentials. In our polynomial notation, in the case of the positive crossing, the 
maps are (up to a non-zero scalar):

$$R_\twoarcs \rightarrow R_\squareup\{-4\}: 1\mapsto (x_2+x'_2-y_2-y'_2)+(y_1-x'_1)z_1+(y'_1-x_1)z'_1.$$

In the case of the second map ($\squareup \rightarrow q^{-2}\unorsmoothing$), it is given by the coefficient of $z_1z'_1$ 
after multiplication with  
$$
\begin{array}{ll}
&(-y_1^3+2y_1y_2+x'_1(y_1^2-y_2)-x'_2y_1-{y'}_1^3+2y'_1y'_2+x_1({y'_1}^2-y'_2)-x_2y'_1) \\
+&(z_1+z'_1)(x_2-x_1y'_1+{y'_1}^2-y'_2+x'_2-x'_1y_1+{y}_1^2-y_2) +\\
+&z_1z'_1 (x_1+x'_1-y_1-y'_1).
\end{array}
$$
Recall that the elements of $R_{\squareup}$ are of the form $a+bz_1+cz'_1+dz_1z'_1$, where $a$, $b$, $c$ and $d$ are the 
polynomials only in $x$'s and $y$'s, while $z_1^2=y_1 z_1-y_2$ and ${z'_1}^2=y'_1 z'_1-y'_2$. 

For the negative crossing, the maps are the following
$$R_\unorsmoothing \rightarrow R_\squareup:  1\mapsto 1,$$
and
$$R_\squareup \rightarrow R_\twoarcs\{2\}: a+bz_1+cz'_1+dz_1z'_1 \mapsto b+c+dy_1.$$

Since we are interested in the differentials in the case when the right strands are closed, in the 
homology all variables with primes are equal to the ones without the primes 
(as we explained above), and the differentials reduce to the 
following (again up to a non-zero scalar)
\begin{align*}
\figins{-22}{0.7}{idweb-markov}\rightarrow 
q^{-4}\ \figins{-22}{0.7}{sqweb-markov}\ &\colon\
1\mapsto (x_2-y_2)-(x_1-y_1)z_1,
\displaybreak[0]\\[1.2ex]
\figins{-22}{0.7}{sqweb-markov}\rightarrow q^{-2}\
\figins{-22}{0.7}{dumbell-markov}\ &\colon\ a+bz_1 \mapsto a(x_1-y_1)+b(x_2-y_2),
\displaybreak[0]\\[1.2ex]
\figins{-22}{0.7}{dumbell-markov}\ \rightarrow\
\figins{-22}{0.7}{sqweb-markov}\ &\colon\  1\mapsto 1,
\displaybreak[0]\\[1.2ex]
\figins{-22}{0.7}{sqweb-markov} \rightarrow q^2\
\figins{-22}{0.7}{idweb-markov}\ &\colon\ a+bz_1 \mapsto b,
\end{align*}
where $a,b\in B$.\\

Now, by straightforward computation of the homology for both positive and negative crossing in each 
Hochschild degree, we obtain the following simple behaviour of HHH under the Reidemeister I move:
\begin{align*}
HHH\Biggl(
\raisebox{-14pt}{
\labellist
\tiny\hair 2pt
\pinlabel $2$ at 25 14
\endlabellist
\figs{0.16}{R1plus}}
\Biggr) 
&= HHH\Biggl(
\raisebox{-14pt}{\
\labellist
\tiny\hair 2pt
\pinlabel $2$ at 29 14
\endlabellist
\figs{0.16}{arcbent-ur}\ }
\Biggr),
\\[1.2ex]
HHH\Biggl(
 \raisebox{-14pt}{
\labellist
\tiny\hair 2pt
\pinlabel $2$ at 25 14
\endlabellist
\figs{0.16}{R1minus}}
\Biggr)
&= HHH\Biggl(
\raisebox{-14pt}{\
\labellist
\tiny\hair 2pt
\pinlabel $2$ at 29 14
\endlabellist
\figs{0.16}{arcbent-ur}\ }
\Biggr)<2>\{-2,-2\}.
\end{align*}
We recall that $<i>$ denotes the upward shift by $i$ in the homological degree, while $\{k,l\}$ denotes 
the upward shift by $k$ in the Hochschild degree and by $l$ in the $q$-degree.\\

\noindent For the Reidemeister I move involving a strand labelled 1, we get a similar shift 
(see \cite{Kh, KR2}, or apply the same methods as above):
\begin{align*}
HHH\Biggl(
\raisebox{-14pt}{
\labellist
\tiny\hair 2pt
\pinlabel $1$ at 25 14
\endlabellist
\figs{0.16}{R1plus}}
\Biggr) 
&= HHH\Biggl(
\raisebox{-14pt}{\
\labellist
\tiny\hair 2pt
\pinlabel $1$ at 29 14
\endlabellist
\figs{0.16}{arcbent-ur}\ }
\Biggr),
\\[1.2ex]
HHH\Biggl(
 \raisebox{-14pt}{
\labellist
\tiny\hair 2pt
\pinlabel $1$ at 25 14
\endlabellist
\figs{0.16}{R1minus}}
\Biggr)
&= HHH\Biggl(
\raisebox{-14pt}{\
\labellist
\tiny\hair 2pt
\pinlabel $1$ at 29 14
\endlabellist
\figs{0.16}{arcbent-ur}\ }
\Biggr)<1>\{-1,-1\}.
\end{align*}
The overall shift in the definition of $H_2(D)$ compensates this behaviour under the Reidemeister I 
moves and we get a genuine link invariant. 

\subsection{Categorification of the axioms A1 and A2}

In this subsection we shall show that our construction categorifies the axioms A1 and A2 of the MOY calculus, that include 
the closures of the strands. The powers of $q$ in A1 and A2 will correspond to the internal $q$-grading (grading of the 
polynomial ring), while the powers of $t$ will correspond to the Hochschild grading.

First we focus on the A1 axiom for arbitrary $k$. The circle is the closure of the single strand labelled with $k$. Denote 
the graph that consists of this strand by $\idup\mspace{-3mu}_k$, and the variables that we associate to it by $x_1,\ldots,x_k$ (remember 
that in this section we assume that $\deg x_i=2i$). Then to this $\idup\mspace{-3mu}_k$ we associate the ring 
$R_{\idup\mspace{-3mu}_k}=\bC[x_1,\ldots,x_k]$, and to its closure, the Hochschild homology 
$HH(\idup\mspace{-3mu}_k)$, which is the homology of the 
Koszul complex obtained by tensoring
\begin{equation*}
0\xrightarrow{\quad} R_{\idup\mspace{-3mu}_k}\{-1,2i-1\} {\xrightarrow{\ 0\ }} R_{\idup\mspace{-3mu}_k} \xrightarrow{\quad} 0,
\end{equation*}
for all $i=1,\ldots,k$. Hence we have
$$HH(\idup\mspace{-3mu}_k)=\bigotimes_{i=1}^k {\left(\bC[x_1,\ldots,x_k] \oplus \bC[x_1,\ldots,x_k] \{-1,2i-1\}\right)},$$
which categorifies the axiom A1.\\

Now we consider the axiom A2. The left hand side of the axiom is the closure of the right strand of the graph
that we denoted by $\unorsmoothing_{ij}$.
As we said previously, we are only interested in the cases when the indices $i$ and $j$ are 
from the set $\{1,2\}$, and so we have four cases. Each of these cases we treat in the same way as 
the case of $\unorsmoothing_{22}$ which we dealt with in the previous subsection. We associate the 
variables $x_1,\ldots,x_i$ and 
$x'_1,\ldots,x'_i$ to the strands on the left-hand side and the variables $y_1,\ldots,y_j$ and 
$y'_1,\ldots,y'_j$ to the strands on the right-hand side, and compute the homology 
(that we denote by $HH^R(\unorsmoothing_{ij})$) of the tensor product of the Koszul complexes 
\begin{equation*}
0\xrightarrow{\quad} R_{\unorsmoothing_{ij}}\{-1,2l-1\} {\xrightarrow{\ y_l-y'_l \ }} R_{\unorsmoothing_{ij}} \xrightarrow{
\quad} 0,
\end{equation*}
for $l=1,\ldots,j$. In the remaining part of the graph, the variables $y$ don't appear again, while as before, in the 
$HH^R(\unorsmoothing_{ij})$ we have that $x_l=x'_l$, for all $l=1,\ldots,i$. Finally, the right hand side of the axiom A2
is $\idup\mspace{-3mu}_i$ to which we associate the variables $x_1,\ldots,x_i$, and consequently the ring $R_{\idup\mspace{-3mu}_i}=\bC[x_1,\ldots,x_i]$.

Now, in the four cases we are interested in, the homologies are the following:\\

$i=j=1:$ The corresponding ring in this case is $R_{\unorsmoothing_{11}}=\bC[x_1,x'_1,y_1,y'_1]/\langle 
x_1+y_1-x'_1-y'_1,x_1y_1-x'_1y'_1\rangle$, while $HH^R_{\unorsmoothing_{11}}$ is the homology of the complex
\begin{equation*}
0\xrightarrow{\quad} R_{\unorsmoothing_{11}}\{-1,1\} {\xrightarrow{\ y_1-y'_1 \ }} R_{\unorsmoothing_{11}} \xrightarrow{
\quad} 0.
\end{equation*}
Direct computation of this homology gives
\begin{align*}
HH^R_0(\unorsmoothing_{11})\mspace{8mu} &=\bC[x_1,y_1],\\
HH^R_{-1}(\unorsmoothing_{11}) &=(x_1-y_1)\bC[x_1,y_1]\{1\},
\end{align*}
which categorifies A2 in this case.\\

$i=2, j=1:$ The corresponding ring in this case is $R_{\unorsmoothing_{21}}=\bC[x_1,x_2,x'_1,x'_2,y_1,y'_1]/\langle 
x_1+y_1-x'_1-y'_1,x_1y_1+x_2-x'_1y'_1-x'_2,x_2y_1-x'_2y'_1\rangle$, while $HH^R_{\unorsmoothing_{21}}$ is the homology of the 
complex
\begin{equation*}
0\xrightarrow{\quad} R_{\unorsmoothing_{21}}\{-1,1\} {\xrightarrow{\ y_1-y'_1 \ }} R_{\unorsmoothing_{21}} \xrightarrow{
\quad} 0.
\end{equation*}
Then we have
\begin{align*}
HH^R_0(\unorsmoothing_{21})\mspace{8mu} &=\bC[x_1,x_2,y_1],\\
HH^R_{-1}(\unorsmoothing_{21}) &=(x_2-x_1y_1+y_1^2)\bC[x_1,x_2,y_1]\{1\},
\end{align*}
as wanted.\\

$i=1, j=2:$ The corresponding ring in this case is $R_{\unorsmoothing_{12}}=\bC[x_1,x'_1,y_1,y_2,y'_1,y'_2]/\langle 
x_1+y_1-x'_1-y'_1,x_1y_1+y_2-x'_1y'_1-y'_2,x_1y_2-x'_1y'_2\rangle$, while $HH^R_{\unorsmoothing_{12}}$ is the homology of the 
complex obtained by tensoring 
\begin{equation*}
0\xrightarrow{\quad} R_{\unorsmoothing_{12}}\{-1,1\} {\xrightarrow{\ y_1-y'_1 \ }} R_{\unorsmoothing_{12}} \xrightarrow{
\quad} 0,
\end{equation*}
and
\begin{equation*}
0\xrightarrow{\quad} R_{\unorsmoothing_{12}}\{-1,3\} {\xrightarrow{\ y_2-y'_2 \ }} R_{\unorsmoothing_{12}} \xrightarrow{
\quad} 0.
\end{equation*}
Thus, we obtain
\begin{align*}
HH^R_0(\unorsmoothing_{12})\mspace{8mu} &=\bC[x_1,y_1,y_2],\\
HH^R_{-1}(\unorsmoothing_{12}) &=\{\left(c(x_1-y_1)+dy_2,-c+dx_1)\right)| c,d\in \bC[x_1,y_1,y_2]\}\\
 & \mspace{30mu}\subset \bC[x_1,y_1,y_2]\{1\}\oplus \bC[x_1,y_1,y_2]\{3\},\\
HH^R_{-2}(\unorsmoothing_{12}) &=(x_1^2-x_1y_1+y_2)\bC[x_1,y_1,y_2]\{4\},
\end{align*}
which categorifies A2 in this case.\\ 

$i=j=2:$ This case we have already computed in the previous subsection, in the formulas (\ref{un220})-(\ref{un222}), which 
gives the categorification of A2 in this case.

\section{Conjectures about the higher fundamental representations}
\label{sec:conjectures}
Ideas similar to the ones in this paper will probably work for 
link components labelled with arbitrary MOY labels. 

First of all, to the open MOY web $\Gamma$ we associate
the bimodule $\widehat{\Gamma}$ in the same way.
In order to define the resolutions for the crossings, in principal one only has to know which 
maps to associate to the zip, the unzip, the digon creation 
and the digon annihilation. For the zip and the unzip our 
candidates are given in the following definition. 

\begin{defn} 
\label{defn:ij-mucomu}
We define the linear maps 
$$\mu_{ij}\colon \hatdumbell_{ij}\to \hattwoarcs_{ij}\quad
\text{and}\quad \Delta_{ij}
\colon \hattwoarcs_{ij}\to 
\hatdumbell_{ij}$$
by 
$$
\mu_{ij}(a\stens b)=ab
$$
and
$$\Delta_{ij}(1)=\sum_{\substack{
\alpha=(\alpha_1,\ldots,\alpha_i) 
\\ 
j\ge\alpha_1\ge\cdots\alpha_i\ge 0 } }
{(-1)^{|\alpha|} \pi_{\alpha}\otimes \pi'_{\bar{\alpha^*}}},$$ 
where $\bar{\alpha^*}$ is the conjugate partition of the complementary partition $\alpha^*=(j-a_i,\ldots,j-a_1)$, and 
$\pi_{\alpha}$ (respectively $\pi'_{\beta})$ is the Schur polynomial in the first $i$ variables (respectively, last $j$ 
variables). The conjugate (dual) partition of the partition $\beta=(\beta_1,\ldots,\beta_k)$, is the partition 
$\bar{\beta}=(\bar{\beta}_1,\bar{\beta}_2,\ldots)$, where $\bar{\beta}_j=\sharp\{i | \beta_i\ge j\}$.
\end{defn}

\noindent The proof of the following lemma will be part of a forthcoming paper:

\begin{lem} $\mu_{ij}$ and $\Delta_{ij}$ are 
$R_{ij}$-bimodule maps. 
\end{lem}

\noindent Note that the lemma above implies that 
$\Delta_{ij}$ is the coproduct in the commutative 
Frobenius extension $H^*_{U_{(i+j)}}(G(i,i+j))$ with respect 
to the trace defined by $\mbox{tr}(\pi_{j\ldots j})=1$.

The digon creation and annihilation maps are deduced from 
Lemma~\ref{lem:digon}.

Before we explain the complexes that we associate to 
the crossings, we call attention to the following 
isomorphisms. 

\begin{equation*}
\labellist
\pinlabel $\cong$ at 274 100
\tiny\hair 2pt
\pinlabel $a$     at  -8  11
\pinlabel $b$     at 136  12
\pinlabel $a+c+d$ at -51 204
\pinlabel $b-c-d$ at 181 204
\pinlabel $a+c$   at -27 120
\pinlabel $b-c$   at 157  85
\pinlabel $c$     at  65  33
\pinlabel $d$     at  65 176
\pinlabel $a$     at 413  10
\pinlabel $b$     at 558  11
\pinlabel $a+c+d$ at 368 204
\pinlabel $b-c-d$ at 603 204
\pinlabel $b-c$   at 580  85
\pinlabel $c+d$   at 462 184
\pinlabel $c$     at 490 102
\pinlabel $d$     at 520 162 
\endlabellist
\centering
\figs{0.25}{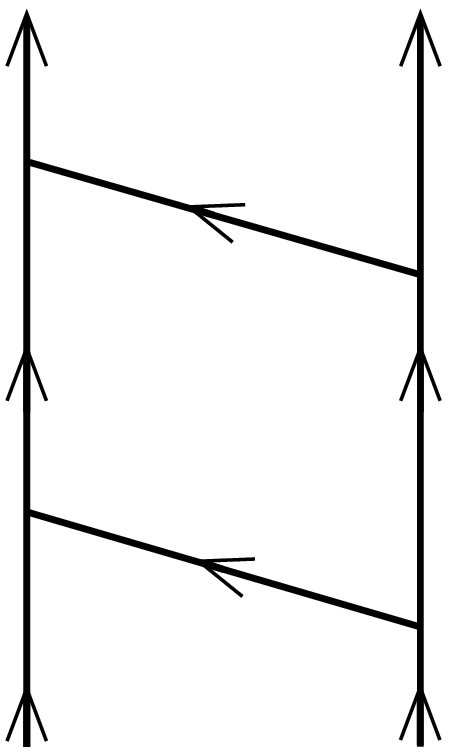}\mspace{120mu}\figs{0.25}{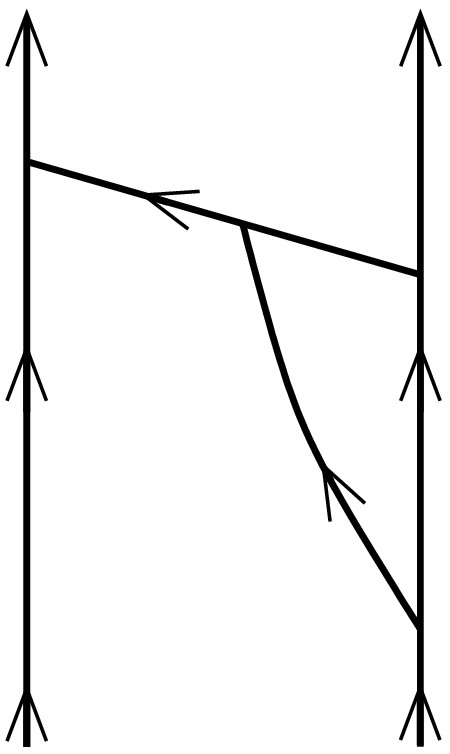}
\end{equation*}

\begin{equation*}
\labellist
\pinlabel $\cong$ at 274 100
\tiny\hair 2pt
\pinlabel $a$     at  -8  11
\pinlabel $b$     at 136  12
\pinlabel $a-c-d$ at -51 204
\pinlabel $b+c+d$ at 181 204
\pinlabel $a-c$   at -27  85
\pinlabel $b+c$   at 157 120
\pinlabel $c$     at  65  33
\pinlabel $d$     at  65 174
\pinlabel $a$     at 413  10
\pinlabel $b$     at 558  11
\pinlabel $a-c-d$ at 368 204
\pinlabel $b+c+d$ at 603 204
\pinlabel $c+d$   at 464  30
\pinlabel $b+c$   at 580 120
\pinlabel $c$     at 520  48
\pinlabel $d$     at 508 146 
\endlabellist
\centering
\figs{0.25}{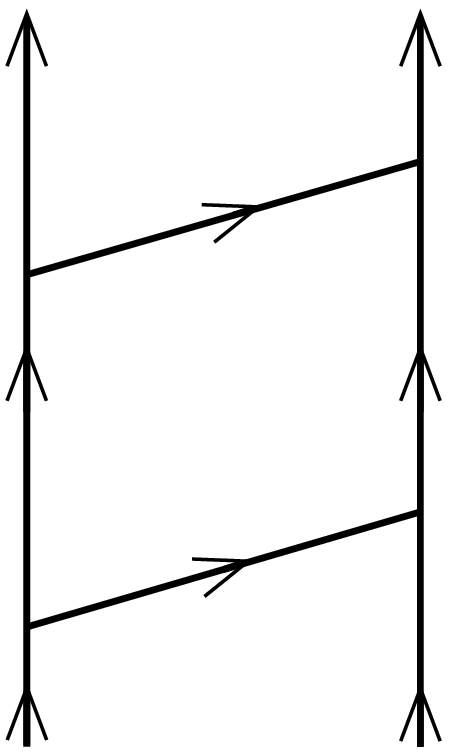}\mspace{120mu}\figs{0.25}{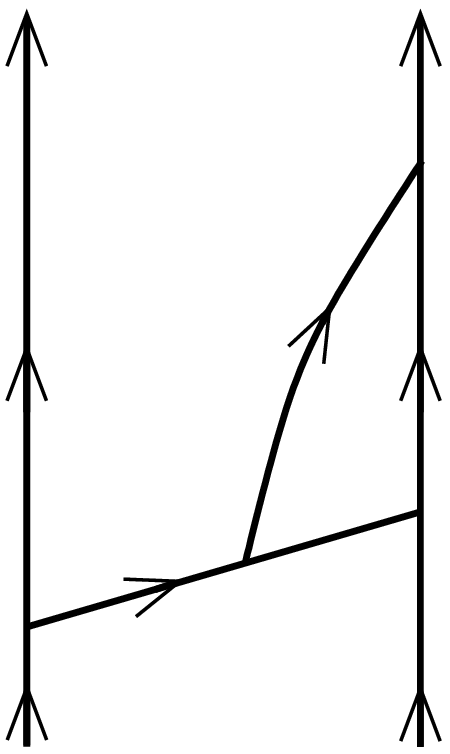}
\end{equation*}

To each positive crossing such that 
$i\leq j$ we associate 
a complex of resolutions (see Figure~\ref{fig:genXingp}) 
$$
\widehat{\Gamma}_0\{i(i+1)\}{\xrightarrow{\ \ d_0^+\ \ }} \widehat{\Gamma}_1\{(i-1)i\}
{\xrightarrow{\ \ d_1^+\ \ }}\widehat{\Gamma}_2\{(i-2)(i-1)\}
{\xrightarrow{\ \ d_2^+\ \ }}\cdots 
{\xrightarrow{\ \ d_{i-1}^+\ \ }}\widehat{\Gamma}_i,
$$
where we put $\widehat{\Gamma}_i$ in the homological degree $0$.
Our candidates for the differentials are indicated in 
Figure~\ref{fig:bigcmplx}. Note that we have used the isomorphisms 
above. 
\begin{figure}[ht!]
\labellist
\tiny\hair 2pt
\pinlabel $j$ at  -6  13
\pinlabel $i$ at 132  13
\pinlabel $j$     at 374  13
\pinlabel $i$     at 374 204
\pinlabel $i$     at 520  13
\pinlabel $j$     at 520 204
\pinlabel $j+k$   at 354 105
\pinlabel $i-k$   at 540 105
\pinlabel $k$     at 446  32
\pinlabel $j+k-i$ at 446 182
\pinlabel $j$     at 756  13
\pinlabel $i$     at 898  13
\small\hair 2pt
\pinlabel $\Gamma_k$ at  448 -50
\endlabellist
\centering
\figs{0.25}{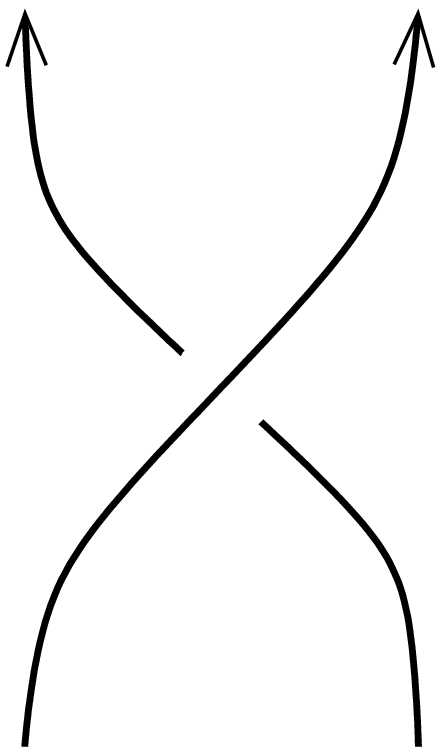}\hspace{12.4ex}
\figs{0.25}{sqweb-rot}\hspace{12.4ex}
\figs{0.25}{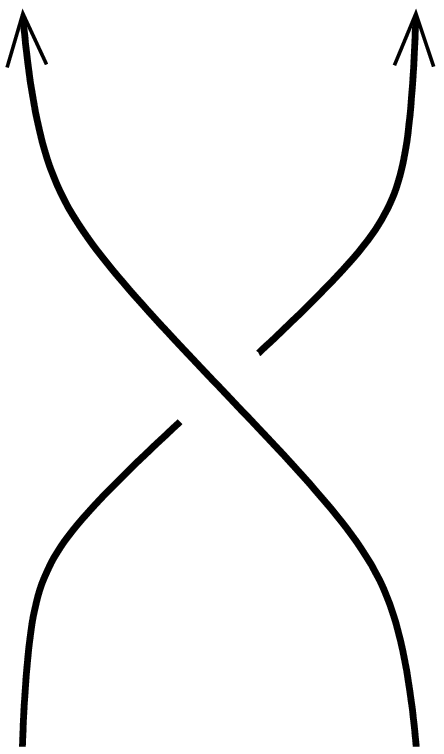}
\\[1.2ex]
\caption{A positive crossing, the web $\Gamma_k$, and a negative crossing}
\label{fig:genXingp}
\end{figure}

\begin{figure}[h!]
\labellist
\pinlabel $q^{-2(j+k)}$ at 1100 580
\pinlabel $q^{-2(j+k)}$ at 360 170
\pinlabel $q^{-2(j+k)}$ at 730 170
\pinlabel $q^{2(k-i)}$ at 1085 170
\tiny\hair 2pt
\pinlabel $j$     at 167 426
\pinlabel $i$     at 167 690
\pinlabel $i$     at 308 426
\pinlabel $j$     at 308 690
\pinlabel $j+k$   at 154 560
\pinlabel $i-k$   at 320 520
\pinlabel $k$     at 235 455
\pinlabel $j+k-i$ at 235 665
\pinlabel $j$       at 534 426
\pinlabel $i$       at 534 690
\pinlabel $i$       at 678 426
\pinlabel $j$       at 678 690
\pinlabel $j+k$     at 522 560
\pinlabel $i-k$     at 689 480
\pinlabel $i-k$     at 689 630
\pinlabel $k$       at 609 455
\pinlabel $j+k-i$   at 609 665
\pinlabel $1$       at 634 560
\pinlabel $i-(k+1)$ at 725 530
\pinlabel $j$       at  905 426
\pinlabel $i$       at  905 690
\pinlabel $i$       at 1045 426
\pinlabel $j$       at 1045 690
\pinlabel $j+k+1$   at  876 558
\pinlabel $j+k$     at  890 600
\pinlabel $j+k$     at  890 520
\pinlabel $i-k$     at 1058 480
\pinlabel $i-k$     at 1058 630
\pinlabel $k$       at  975 455
\pinlabel $j+k-i$   at  975 665
\pinlabel $i-(k+1)$ at  984 558
\pinlabel $1$       at  950 600
\pinlabel $1$       at  950 520
\pinlabel $j$       at 164  14
\pinlabel $i$       at 164 278
\pinlabel $i$       at 309  14
\pinlabel $j$       at 309 278
\pinlabel $j+k+1$   at 135 144
\pinlabel $j+k-i+1$ at 235 274
\pinlabel $i-(k+1)$ at 246 144
\pinlabel $i-k$     at 320  74
\pinlabel $i-k$     at 320 225
\pinlabel $k+1$     at 210 102
\pinlabel $j+k-i$   at 238 240
\pinlabel $k$       at 264  55
\pinlabel $1$       at 278 208
\pinlabel $1$       at 278  88
\pinlabel $j$       at 534  14
\pinlabel $i$       at 534 278
\pinlabel $i$       at 676  14
\pinlabel $j$       at 676 278
\pinlabel $j+k+1$   at 506 144
\pinlabel $i-(k+1)$ at 714  94
\pinlabel $k+1$     at 606 132
\pinlabel $k$       at 605  38
\pinlabel $j+k-i+1$ at 607 160
\pinlabel $j+k-i$   at 607 258
\pinlabel $1$       at 607 208
\pinlabel $1$       at 607  88
\pinlabel $j$       at  902  14
\pinlabel $i$       at  902 278
\pinlabel $i$       at 1044  14
\pinlabel $j$       at 1044 278
\pinlabel $j+k+1$   at  872 144
\pinlabel $i-(k+1)$ at  982 104
\pinlabel $k+1$     at  976  42
\pinlabel $j+k-i+1$ at  976 258
\endlabellist
\centering
\figs{0.34}{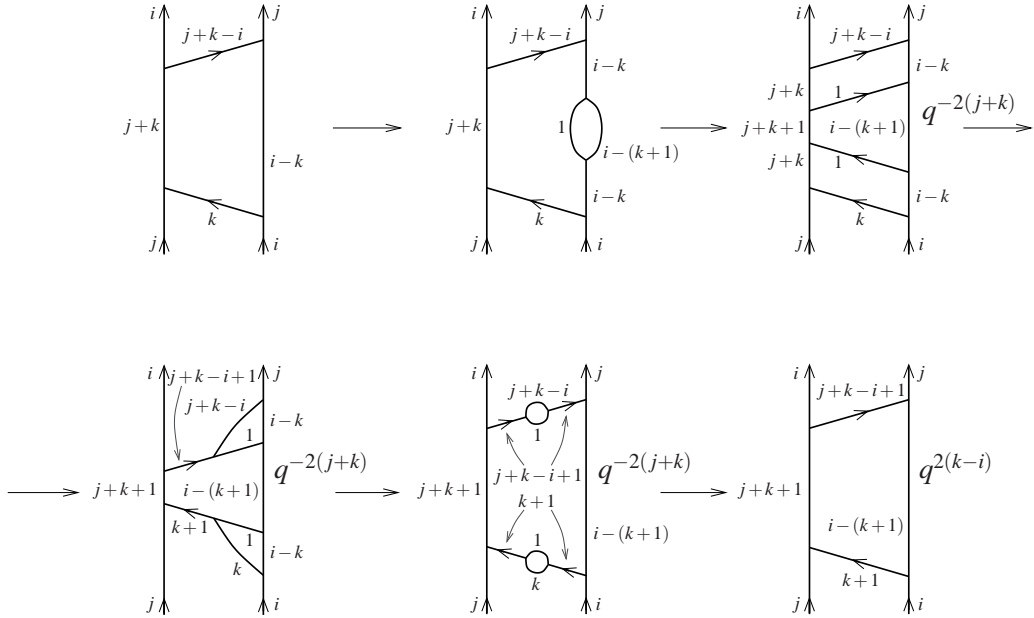}
\caption{The definition of the differentials}
\label{fig:bigcmplx}
\end{figure}

To negative crossings with $i\le j$, we associate the following complex of bimodules:
$$
\widehat{\Gamma}_i\{-2ij\}{\xrightarrow{\ \ d_0^+\ \ }}\widehat{\Gamma}_{i-1}\{-2ij\}
{\xrightarrow{\ \ d_1^+\ \ }}\widehat{\Gamma}_{i-2}\{2-2ij\}
{\xrightarrow{\ \ d_2^+\ \ }}\cdots 
{\xrightarrow{\ \ d_{i-1}^+\ \ }}\widehat{\Gamma}_0\{(i-1)i-2ij\},
$$
where again we put the resolution $\widehat{\Gamma}_i$ in the homological degree 0. The differentials are obtained by
inverting all arrows in the 
Figure~\ref{fig:bigcmplx}. \\

In the case when $i<j$, by looking at the picture from the ``other side of the paper" (i.e. by rotation around the $y$-axis), 
we obtain the crossing of the above form.\\

The following lemma is hard to prove directly, as one has 
to check the statement for a lot of generators. However in 
a subsequent paper we will prove this conjecture using a 
different technique.  
\begin{conj} $d_{k+1}^{\pm}d_{k}^{\pm}=0$, for all 
$0\leq k\leq i-2$. 
\end{conj}
The other tricks remain largely the same. The crucial conjecture to prove is 
\begin{conj}
We have the homotopy equivalence
\begin{equation*}
\label{eq:slidei1}
\labellist
\pinlabel $\cong$ at 264 110
\tiny\hair 2pt
\pinlabel $j$ at -12 14
\pinlabel $1$ at -10 177
\pinlabel $i$ at 73 204
\pinlabel $i+1$ at 215 14
\pinlabel $j$ at 335 14
\pinlabel $1$ at 338 187
\pinlabel $i$ at 437 214
\pinlabel $i+1$ at 575 14
\endlabellist
\raisebox{-16pt}{\figs{0.2}{slide1}\qquad\quad \figs{0.2}{slide2}}
\end{equation*}
\end{conj}


\vspace*{1cm}

\noindent {\bf Acknowledgements} 
The authors thank Mikhail Khovanov and Catharina Stroppel for helpful conversations on the topic of this 
paper. 

The authors were supported by the 
Funda\c {c}\~{a}o para a Ci\^{e}ncia e a Tecnologia (ISR/IST plurianual funding) through the
programme ``Programa Operacional Ci\^{e}ncia, Tecnologia, Inova\-\c
{c}\~{a}o'' (POCTI) and the POS Conhecimento programme, cofinanced by the European Community 
fund FEDER. Marko Stosic is also partially supported by the Ministry of Science of Serbia, project 144032.


\end{document}